\documentclass[12pt]{article}
\usepackage{amsmath, amsthm}
\usepackage{amsmath}
\usepackage{amsfonts}
\usepackage{amssymb}
\usepackage{mathrsfs}
\usepackage{color}

\usepackage{graphicx}
\newtheoremstyle{thm}{1.5ex}{1.5ex}{\itshape\rmfamily}{}
{\bfseries\rmfamily}{}{2ex}{}

\newtheoremstyle{rem}{1.3ex}{1.3ex}{\rmfamily}{}
{\itshape}
{} {1.5ex}{}

\theoremstyle{thm}

\newtheorem{thm}{Theorem}[section]

\newtheorem{cor}[thm]{Corollary}
\newtheorem{corollary}[thm]{Corollary}
\newtheorem{lemma}[thm]{Lemma}
\newtheorem{prop}[thm]{Proposition}
\newtheorem{proposition}[thm]{Proposition}

\theoremstyle{definition}
\newtheorem{definition}[thm]{Definition}
\newtheorem{remark}[thm] {Remark}
\newtheorem*{remm}{Remark}

\def\t{_{\text{\tiny T}}}

\def\b{\textcolor{blue}}
\def\e{\varepsilon}

\usepackage{color} 

\definecolor{purple}{rgb}{0.65, 0, 1}
\definecolor{orange}{rgb}{1,.5,0}

\begin{document}

%
%
%
%
%
%
%
\title
{\Large Transport and Equilibrium in Non--Conservative Systems \footnotemark[0]}
\author
{ L. Chayes$^{\dagger}$ and H. K. Lei$^1$}
\date{}
\maketitle

\footnotetext[1]{thshelen9@gmail.com}

\vspace{-4mm}
\centerline{${}^{\dagger}$\textit{Department of Mathematics, University of California at Los Angeles}}

\abstract{We study, in finite volume, a grand canonical version of the McKean--Vlasov equation where the total particle content is allowed to vary.  The dynamics is anticipated to minimize an appropriate grand canonical free energy; we make this notion precise by introducing a metric on a set of positive Borel measures without pre--prescribed mass and demonstrating that the dynamics is a gradient flow with respect to this metric. Moreover, we develop a JKO--type scheme suitable for these problems.  The latter ideas have general applicability to a class of second order non--conservative problems.  For this particular system we prove, using the JKO--type scheme, that under certain conditions -- not too far from optimal -- convergence to the uniform stationary state is exponential with a rate which is independent of the volume.  By contrast, in related conservative systems, decay rates scale (at best) with the square of the characteristic length of the system.  This suggests that a grand canonical \textit{approach} may be useful for both theoretical and computational study of large scale systems.  }

\section{Introduction}
This paper concerns the evolution and the convergence to equilibrium for a certain class of non--linear diffusion equations which may vaguely be described as of the 
McKean--Vlasov or Keller--Segel type.  Such systems have been well studied in recent years; here the primary distinction will be that the total mass is not conserved locally in time but, rather, is globally determined by the analogue of a \textit{Lagrange multiplier} which is known as the \textit{chemical potential} \b{(see e.g., \cite{RB}, page 129)}.  
Secondly, we work in finite volume.  This setting is arguably (see \cite{CP}) the  physically sensible approach to the mathematical study of approximately homogeneous fluids described by these dynamics.  
Extensive behavior -- static or dynamic -- can only emerge as the infinite volume limit of finite systems where the total mass scales with the volume.  
In this context, the non--conservative setup (\textsc{aka} grand canonical) 
has distinct advantages over its conservative (\textsc{aka} canonical) counterpart.  Indeed, as is quite well known
(see, e.g., \cite{CP})
the latter generically has relaxation times which scale with a power of the characteristic length of the system.  
Here (under some lenient conditions on the initial data and parameter values) 
we demonstrate an exponential convergence to equilibrium with a rate that is uniform in the volume. 
Moreover, this will be proved under conditions where the driving functional relevant to the problem does not necessarily enjoy convexity properties\footnotemark[3].

\footnotetext[3]{These results should be contrasted with several notable earlier works e.g., \cite{CMV} which treat systems in \textit{a priori} infinite volume and obtain exponential convergence to equilibrium with a rate which -- necessarily -- is uniform in volume.  The aforementioned pertain to conservative systems with finite mass; in the absence of external constraints all mass would eventually drift away.  So, in these works, mass is confined by an external potential which render the setting to an effectively finite--volume problem.  Moreover, the curvature of the confining potential provides uniform convexity which drives the exponential convergence.  Scaling (or linear response theory) immediately shows that the actual rate of convergence is the curvature itself which, in turn, is the square of the \textit{effective} length--scale of the system.  The curvature dependence of the rate is explicit in the statement of Theorem 2.1 in \cite{CMV} (c.f.~equation (2.8)).}
Our proofs of these assertions -- precise statements will be presented at the close of this section -- require the parallel development of a theory of 
\textit{optimal transport} for non--conservative systems. 
In particular, as will be
outlined in Section \ref{JKO} below, this necessitates the construction of a distance between positive $L^2$--functions (which, with additional labor, might be extended to general  Borel measures). 
And, associated with this distance and dynamics -- as presented in Section \ref{JKO} -- will be a JKO--type scheme \cite{JKO}, which constitutes the core of the proof. 

Here it is remarked that, since the start of this research, there has been a parallel development of some of these \b{ideas in \cite{LMS}, \cite{KMV} and \cite{CPSV} (also see \cite{LM} and references therein)} in the context of reaction diffusion equations.  However, for us, the construction of a \emph{framework} is only the preliminary step: Our efforts culminate in tangible results for the system which will be described in Eq.\eqref{GCD}.  \b{Moreover, while our focus here is on a particular equation, the methodologies we develop can certainly be applied to a variety of similar systems.}

On a more practical note, it is emphasized that while the equation we will study is akin to a reaction diffusion system, the results we have obtained will not apply to \emph{actual} reaction diffusion systems which, ultimately, \emph{are} conservative.  In particular, unless the overall density is already homogeneous, equilibrium times in reaction diffusion systems will be dominated by diffusive modes which necessitates that the relaxation times scale with the square of the characteristic length of the system. However, in the \emph{grand canonical} (hence non--conservative) versions of these reaction diffusion systems it is anticipated that the convergence rates for uniform equilibria 
will be independent of the volume; similar considerations apply for the types of problems treated in e.g., \cite{CMV}.

At this stage we must underscore some \ae sthetic limitations:  While in conservative cases, the JKO \textit{schemes} necessarily pertain to the dynamic, the \b{underlying} distance involved, \b{usually the Wasserstein distance,} is ``universal" depending e.g., only on the ambient space.  In the current cases, as will become clear, what emerges is that the distance itself \b{evidently} depends on the particulars of the dynamical equation.  
(A somewhat analogous situation -- in mass conserved cases -- was considered in \cite{FGY}.)
Nevertheless we remark that even without the JKO scheme, the grand canonical approach to this general set of problems may have distinctive advantages over the canonical versions.  In this regard, it should be noted that for the problems studied here, for a.e.\hspace{- 3 pt} value of the chemical potential, the steady state solutions of the two systems coincide.  Thus, while exponential convergence uniform in volume is not to be expected in the high density phase, it is not too much to hope that in general the grand canonical systems equilibrate in a reasonable computational time frame.  The corresponding conservative versions often appear to be computationally unviable.

The central focus of this paper concerns the analysis of an inhomogeneous version of the McKean--Vlasov equation in which matter can effuse into and out of the system.  The usual conservative version can be derived in a variety of contexts; the original rendition presumably dates back to \cite{Mc}.  The non--conservative version also admits several derivations.  For the purposes of this motivational section, we will provide, in Subsection \ref{DYN}, a common (sketch of a) derivation based on familiar interacting particle models.  This has the distinct advantage that it connects directly to the \textit{thermodynamics}  (free energetics) which underlie these evolutions.  The latter, which can always be analyzed without recourse to dynamics, is the subject of Subsection \ref{Motivation} below.  In the forthcoming subsections, there will be no pretense to a complete mathematical analysis (however, a full derivation may emerge in some future work).  

\subsection{Motivation}
\label{Motivation}
Consider a function $N(x,t)$ obeying the McKean--Vlasov dynamic
\begin{equation}
\label{ABC}
\frac{\partial N}{\partial t}  =  \triangle N 
+ \nabla \cdot (N\nabla w_{N})
\end{equation}
where
$$
w_{N}(x)  :=  \int_{\mathbb T_L^d} W(x-y) N(y)~ dy.
$$

It may be assumed without too much loss of generality
that $W(\cdot)$ depends only on the modulus of its argument.
While a variety of ambient spaces are possible, for simplicity here and throughout this work, we will use 
$\mathbb T_{L}^{d}$, the $d$--dimensional torus of side length $L$ as indicated above.  The $L^{1}$--norm of $N$ is preserved in time and with 
$\int_{\mathbb T_L^d} Ndx =: \vartheta L^{d}$, this is precisely the problem studied in \cite{CP}.  
As is well known  (e.g., this is discussed in \cite{V}, especially Ch.~8) 
Eq.~\eqref{ABC} is a gradient flow with respect to the Wasserstein distance for the (canonical) functional
$$
\mathscr F_{\vartheta}(N) :=  \int_{\mathbb T_L^d}(N\log N - N)~dx + \frac{1}{2}\int_{\mathbb T_L^d \times \mathbb T_L^d} W(x-y)N(x)N(y)~dxdy.
$$
In the context of minima for $\mathscr F_{\vartheta}$ and/or evolution according to Eq.~\eqref{ABC} it is preferable that $W$ satisfy a condition known as H--stability which, in the present setup, reads that for all $m(x)$ with $m(x) \geq 0$, 
$$
\int_{\mathbb T_L^d \times \mathbb T_L^d} W(x-y) m(x)m(y)~dxdy \geq 0.
$$

We take some time to recollect some results for the minimizers of 
$\mathscr F_{\vartheta}(\cdot)$ all of which are proved in \cite{CP} but some of which date back to an earlier epoch: See \cite{KM}, \cite{vK}, \cite{G}, \cite{GP}, \cite{LP}, \cite{GK}.  It is assumed throughout that $W$ satisfies the H--stability condition.
If $\vartheta$ is sufficiently small, $N\equiv \vartheta$ is the unique minimizer.  When $W$ is of positive type, 
implying convexity of 
$\mathscr F_{\vartheta}(\cdot)$, this actually holds for all $\vartheta$.  Otherwise, the uniform state becomes (linearly) unstable at $\vartheta = \vartheta^\sharp$ which is given by the inverse of the maximum of the absolute value of the negative Fourier modes of $W$.  However, under fairly general circumstances, the existence of non--uniform minimizers
occurs at $\vartheta = \vartheta_{\t} < \vartheta ^{\sharp}$; for 
$\vartheta > \vartheta_{\t}$, the uniform state is no longer a global minimizer.

The grand canonical generalization of $\mathscr F_{\vartheta}$ wherein the integral of $N$ is not fixed is given by
\begin{equation}
\label{GCF}
\mathscr G_{\mu}(N) 
:=  \int_{\mathbb T_L^d}(N\log N - [N + \mu N])~dx + \frac{1}{2}\int_{\mathbb T_L^d \times \mathbb T_L^d} W(x-y)N(x)N(y)~dx~dy
\end{equation}
where, as mentioned earlier, $\mu$ is called the chemical potential.  
Here it is seen that the H--stability condition is, for all intents and purposes, essential.  
(It is also worth noting that some of the older results alluded to above were actually established under the jurisdiction of this grand canonical functional.)  Let us summarize without proof the essential results needed for the background of this work.  
For fixed $\mu$, the set of minimizers is non--empty.  There are well defined  upper and lower integrated densities associated with each $\mu$ both of which are realized by elements in this set. These integrated densities are (both) strictly monotone 
and coincide for a.e.~$\mu$.  If $\mu$ is sufficiently small then the uniform state is the unique minimizer.  The density in the uniform state is given by $\text{\textsc{m}}_{0} = \text{\textsc{m}}_{0}(\mu)$ and satisfies the equation
\begin{equation}
\label{ASKY}
\text{\textsc{m}}_{0}  =  \text{e}^{\mu}\text{e}^{-w\text{\textsc{m}}_{0}}
\end{equation}
with $w = \int_{\mathbb T_{L}^{d}}W(x)dx$.  
It is noted that $N \equiv \text{\textsc{m}}_{0}$ is always a stationary state for $\mathscr G_{\mu}(\cdot)$, i.e., it satisfies the relevant Euler--Lagrange equation which is known in this context as the
Kirkwood--Monroe equation
\cite{KM}.

In particular, $N \equiv \text{\textsc{m}}_{0}$ remains the global minimizer till 
a point of discontinuity  $\mu_{\t}$ is reached where the upper and lower densities do not coincide and, in fact, bracket 
$\vartheta_{\t}$.  For values of $\mu$ greater than $\mu_{\t}$ the uniform density is no longer the minimizer and, at a strictly higher chemical potential, $\mu_{\sharp}$ -- the value of $\mu$ such that $\text{\textsc{m}}_{0} = \vartheta_{\sharp}$ -- the uniform state becomes linearly unstable.

 The implication is that the non--uniform minimizer for $\mathscr F_{\vartheta}(\cdot)$ at 
$\vartheta = \vartheta_{\t}$ is non--homogeneous \b{and presumably} cannot be understood without first understanding the grand canonical version of the transition.  Moreover, simulations of the canonical dynamics 
at $\vartheta \sim \vartheta_{\t}$ 
may require unmanageable computational time scales till a non--uniform minimizer is reached.  See, e.g., \cite{CP} Theorem 2.11.  
But before such questions can be addressed for the grand canonical problem, a dynamic must be presented which corresponds to the functional 
$\mathscr G_{\mu}(\cdot)$.  This is the topic of our next subsection.

\subsection{Dynamics}
\label{DYN}

While it is clear on general grounds that the ``correct'' equation for grand canonical dynamics involves the augmentation of Eq.~\eqref{ABC} by inhomogeneous terms, the form of these terms is not particularly obvious.  
Moreover, the guiding principle is somewhat nebulous: \b{The physics dictates an ``intrinsic uncertainty'' in the particle content of the system; i.e., there is a probability \emph{distribution} for the number of particles.  Here, this translates into an intrinsic uncertainty in $\|N\|_{L^1}$.}  While these matters are well understood in equilibrium, it is not so clear how this uncertainty is supposed to propagate dynamically.
The answer lies in the stipulation that the (nebulous) physics of this intrinsic uncertainty is equivalent, at the microscopic level, to the circumstances where 
individual particles can appear and disappear according to (a) the energetics of the complementary configuration and (b) a parameter, already mentioned, called the \emph{chemical potential}.  In the remainder of this subsection we will provide motivation for the form of the dynamics we wish to study, but this content is not essential to the remainder of this work.  The disinterested reader can proceed directly to Eq.~\eqref{GCD}.  

Let us turn to a (brief and informal) discussion of the relevant lattice models both in the context of equilibrium and dynamics.  Since we have in mind a finite volume problem in $\mathbb R^d$, or on the torus, let $A \subseteq \mathbb R^d$ (or $A = \mathbb T_L^d$) be some regular set and let $\mathbb A_\e$ denote the intersection of $A$ with $\mathbb Z_\e^d$, the integer lattice of spacing $\e$.  Letting $V$ denote the volume of $A$, the number of sites in $\mathbb A_\e$, denoted by $|\mathbb A_\e|$, is approximately $|A_\e| \simeq \e^{-d} V$.  We shall consider particle configurations $(\eta_{_{\mathbf X}}(j) \in \mathbb N \mid j \in \mathbb A_\e)$ where here, $\mathbb N$ includes zero.  Most of the discussion will concern the \emph{conservative} case: 
$$\sum_j \eta_{_{\mathbf X_j}} \equiv n,$$
which is considered to be fixed.  We concern ourselves with an informal discussion of the $\e \rightarrow 0$ limit with the scaling $n\e^d \rightarrow NV$ for some $N > 0$.  The advantage of the lattice discretization is that it enables the usage of \emph{particle systems} to induce dynamics in a straightforward fashion.

We re--emphasize that we make no claims to a rigorous derivation; we simply perform the analog of the 
calculations done in 
\cite{P}
for the Ising case wherein the 
Cahn--Hilliard and Cahn--Allen equations were acquired.  
Explicitly, we expand terms and, scaling time appropriately and neglecting correlations,
we retain only the leading order in $\varepsilon$.
As we will see, the resultant equation for the particle density 
$N$ at $x\in A$ is given by
$\frac{\partial N}{\partial t} = \nabla^{2}N + \nabla\cdot (N\nabla w_{N})$, i.e., exactly the McKean--Vlasov equation, Eq.~\eqref{ABC}.  Then, adding terms which allow the particle content to fluctuate, we shall arrive at the dynamical equation we wish to study. 

Let us now proceed with the discrete calculations.  Starting with (non--interacting) statics, we assign an \emph{a priori} weight $W(\mathbf X)$ to each configuration $\mathbf X$.  Later this will be augmented by an interaction expressed via a \emph{Hamiltonian}.  We choose, on a basis which is not entirely physical, the weights 
\begin{equation}\label{WTT} W(\mathbf X) = \left(\prod_{j \in \mathbb A_\e} \eta_{_{\mathbf X_j}}! \right)^{-1}.\end{equation}
It is noted that 
\[ \sum_{\mathbf X} W(\mathbf X) =: Z_{\e, n, V} = \frac{|\mathbb A_\e|^n}{n!},\]
so that, automatically
\[ \lim_{\e \rightarrow 0} \frac{1}{|\mathbb A_\e|} \log Z_{\e, n, V} = N \log N - N,\]
which is the \emph{free energy} of an ideal gas.  

Next, still in the context of a non--interacting system, let us introduce transition rates $T_{\mathbf X:\mathbf Y}$, the rate at which the system exhibits the configuration $\mathbf Y$ given that it is in the configuration $\mathbf X$.  For simplicity we will always restrict attention to transitions which only involve nearest neighbor jumps of a single particle: $T_{\mathbf X: \mathbf Y} = 0$ unless $\eta_{_{\mathbf Y}}(k) = \eta_{_{\mathbf X}}(k)$ for all $k$ except a pair $i, j$ with $\|i - j\| = \e$ in which case 
\[ \eta_{_{\mathbf Y}}(j) = \eta_{_{\mathbf X}}(j) \pm 1~~~\mbox{while}~~~ \eta_{_{\mathbf Y}}(i) = \eta_{_{\mathbf X}}(i) \mp 1,\]
provided that this move keeps both $\eta_{_{\mathbf Y}}(j)$ and $\eta_{_{\mathbf Y}}(i)$ nonnegative.  In other words, we only allow transitions in which a single particle is transferred to a neighboring site. 

If $\mathbb Q(\cdot)$ is a probability measure on the space of particle configurations the transition rates $T_{_{\mathbf X: \mathbf Y}}$ satisfy the condition of \emph{detailed balance} with respect to $\mathbb Q$ if, for every configuration $\mathbf X$ and $\mathbf Y$ 
\[ \mathbb Q(\mathbf X) T_{_{\mathbf X: \mathbf Y}} = \mathbb Q(\mathbf Y) T_{_{\mathbf Y: \mathbf X}}.\]
When detailed balance is satisfied, the measure $\mathbb Q$ is \emph{invariant} for the process.  For the weights in Eq.~\eqref{WTT} it is clear that detailed balance is satisfied if the rate of transfer of a particle from a given site to a neighboring site is equal (or proportional) to the number of particles at the (given) site.  

In this case we equivalently have $n$ particles executing independent random walks.  In particular, the behavior is weakly diffusive in the sense that if $\Omega_\e$ is the generator for this process, then 
\begin{equation}\label{TTR} \Omega_\e \eta{_{_\mathbf X}}(k) = \sum_{\ell: |\ell - k| = \e} \eta{_{_\mathbf X}}(\ell) - \eta{_{_\mathbf X}}(k) := (\triangle_\e \eta_{_{\mathbf X}})(k).\end{equation}
The right hand side, the discrete Laplacian, is weakly of order $\e^2$.  Since the left hand side more or less corresponds to a time derivative, this necessitates that time be scaled by $\e^2$, i.e., diffusive scaling.  We shall consider this a sufficient discussion of the non--interacting case.

Let us now turn to the problem of interactions.  In the context of classical equilibrium statistical mechanics, interactions are implemented by introducing a Hamiltonian which is a real--valued function of the configurations that we denote by $H(\mathbf X)$.  The \emph{canonical} equilibrium is defined as the probability measure on configurations which is given by the weights $W(\mathbf X) \text{e}^{-H(\mathbf X)}$.  As for dynamics, if $T'_{_{\mathbf X: \mathbf Y}}$ satisfies detailed balance for the non--interacting cases, it is seen that if we define (regardless of the precise form of $H$) the rates
\begin{equation}\label{RRT} T_{_{\mathbf X: \mathbf Y}} = T'_{_{\mathbf X: \mathbf Y}} \text{e}^{\frac{1}{2} [H(\mathbf X) - H(\mathbf Y)]},\end{equation}
then the resulting dynamics will satisfy detailed balance with respect to the canonical measures.  Here, we are interested in interactions which are of the mean--field type.  For $r > 0$, we let $W(r)$ be a smooth function, then we may take the Hamiltonian to be 
\[ H(\mathbf X) = \frac{\e^d}{2} \sum_{k, \ell} \eta_{_{\mathbf X}}(k) \eta_{_{\mathbf X}}(\ell) W_{k, \ell}\]
where $W_{k, \ell}$ is standing notation for $W(\|k - \ell\|)$.
In the above, the customary factor of $n^{-1}$ has been replaced by $\e^d$ and we also implement the convention that $W(0) \equiv 0$. It is noted that with the pre--factor of $\e^d$, the interaction associated with a single site, i.e., $\e^d \eta_{_{\mathbf X}}(\ell) \sum_k \eta_{_{\mathbf X}}(k)W_{k, \ell} $, is of order unity whereas the total interaction is of order $n$ which is ``extensive''.  

Next we calculate the quantity $\frac{1}{2} [H(\mathbf X) - H(\mathbf Y)]$ for the case where as particle has transferred from a particular site $i$ (where $\eta_{_{\mathbf X}}(i) > 0$) to a neighboring site $j$.  I.e., 
\[ \eta_{_{\mathbf Y}}(i) = \eta_{_{\mathbf X}}(i) - 1,~~ \eta_{_{\mathbf Y}}(j) = \eta_{_{\mathbf X}}(j) + 1;~~~ \eta_{_{\mathbf Y}} (k) = \eta_{_{\mathbf X}}(k), ~k \neq i, j.\]
(In the ensuing computations we will assume that $i$ is an interior site.)  The result of the above described computation is 
\[\frac{1}{2} [H(\mathbf X) - H(\mathbf Y)] = -\frac{\e^d}{2} W_{i, j} + \frac{\e^d}{2} \sum_\alpha \eta_{_{\mathbf X}}(\alpha) (W_{i, \alpha} - W_{j, \alpha}). \]
We will neglect the first term and denote the second term by $\frac{1}{2} [A_{_{\mathbf X}}(i) - A_{_{\mathbf X}}(j)]$.  Thus, for the site $i$, the rate of particle transfer from and to site $j$ is given in the display 
\[ -\eta_{_{\mathbf X}}(i) \text{e}^{\frac{1}{2} [A_{_{\mathbf X}}(i) - A_{_{\mathbf X}}(j)]} + \eta_{_{\mathbf X}}(j) \text{e}^{-\frac{1}{2} [A_{_{\mathbf X}}(i) - A_{_{\mathbf X}}(j)]},\]
where the second term is calculated by interchanging $i$ and $j$.  

We may expand these exponents, realizing that the differences $A_{_{\mathbf X}}(i) - A_{_{\mathbf X}}(j)$ are themselves of order $\e$.  The preceding display then reads 
\[ \begin{split} -\eta_{_{\mathbf X}}(i) &\left(1 + \frac{1}{2} [A_{_{\mathbf X}}(i) - A_{_{\mathbf X}}(j)] + \frac{1}{2} \left(\frac{1}{2} [A_{_{\mathbf X}}(i) - A_{_{\mathbf X}}(j)]\right)^2 + \dots \right) \\
&+ \eta_{_{\mathbf X}}(j) \left(1 - \frac{1}{2} [A_{_{\mathbf X}}(i) - A_{_{\mathbf X}}(j)] + \frac{1}{2} \left(\frac{1}{2} [A_{_{\mathbf X}}(i) - A_{_{\mathbf X}}(j)]\right)^2 + \dots \right).
\end{split}\]
Now we claim that all but the first 2 terms in each of the expansions can be neglected.  Indeed, diffusive scaling indicates that we only need to retain to order $\e^2$.  The terms not written are \emph{a priori} at least of order $\e^3$ and higher.  As for the third terms in the preceding display: The presence of $[A_{_{\mathbf X}}(i) - A_{_{\mathbf X}}(j)]^2$ is already of order $\e^2$ but then they combine to yield the pre--factor of $\eta_{_{\mathbf X}}(i) - \eta_{_{\mathbf X}}(j)$ which is weakly of order $\e$.  Thus we may stipulate that 
\begin{equation} \label{GGT}\Omega_\e \eta_{_{\mathbf X}}(i) \simeq \sum_{j: \|j - i\|= \e} [\eta_{_{\mathbf X}}(j) - \eta_{_{\mathbf X}}(i)]  + \frac{1}{2} [\eta_{_{\mathbf X}}(i) + \eta_{_{\mathbf X}}(j)] [ A_{_{\mathbf X}}(j) - A_{_{\mathbf X}}(i)].\end{equation}

The first term in the above display has already been identified as the discrete Laplacian $\Delta_\e\eta_{_{\mathbf X}}(i)$.  The second term can be written as
\[\begin{split} \frac{1}{2} [\eta_{_{\mathbf X}}(i) &+ \eta_{_{\mathbf X}}(j)] [ A_{_{\mathbf X}}(j) - A_{_{\mathbf X}}(i)]\\
&= \eta_{_{\mathbf X}}(i) \sum_{j: \|j - i\|= \e} [A_{_{\mathbf X}}(j) - A_{_{\mathbf X}}(i)] + \frac{1}{2} \sum_{j: \|j - i\| = \e} [\eta_{_{\mathbf X}}(j) - \eta_{_{\mathbf X}}(i)][A_{_{\mathbf X}}(j) - A_{_{\mathbf X}}(i)]. \end{split}\]
Now we identify the first term on the right hand side as $\eta_{_{\mathbf X}}(i) \Delta_\e A_{_{\mathbf X}}(i)$.  To address the second term, we recall the forward and backward lattice gradients (and divergences): Let $f(i)$ be a lattice function and $\hat e_s$ a standard unit vector, then 
\[ \nabla_\e^+ f(i) := \sum_{s=1}^d [f(i + \e\hat e_s) - f(i)] \hat e_s\]
\[ \nabla_\e^- f(i) := \sum_{s=1}^d [f(i) - f(i - \e\hat e_s)] \hat e_s.\]
In this language, the term of interest becomes 
\[ \frac{1}{2} \sum_{j: \|j - i\|=\e} [\eta_{_{\mathbf X}}(j) - \eta_{_{\mathbf X}}(i)][A_{_{\mathbf X}}(j) - A_{_{\mathbf X}}(i)] = \frac{1}{2} \left[\nabla_\e^+ \eta_{_{\mathbf X}} (i) \cdot \nabla^+ A_{_{\mathbf X}}(i) + \nabla^- \eta_{_{\mathbf X}}(i) \cdot \nabla_\e^- A_{_{\mathbf X}}(i)\right].\]

To conclude the conservative case we observe that the stated dynamics for $N(x, t)$ in Eq.~\eqref{ABC} reads (at least classically) that 
\[ \frac{\partial N}{\partial t} = \triangle N + \nabla \cdot(N \nabla w_N) = \triangle N + N \triangle w_N + \nabla N \cdot \nabla w_N.\]
This has been formally reproduced by $\Omega_\e \eta_{_{\mathbf X}}$, the generator for the discrete process acting on the particle density.

The preceding readily generalizes to the case where the particle content is allowed to vary.  In the context of equilibrium statistical mechanics this is implemented by the introduction of the chemical potential, $\mu \in \mathbb R$, and defining the weights 
\[ \tilde W(\mathbf X) = W(\mathbf X) \cdot \text{e}^{-H(\mathbf X)} \cdot \text{e}^{\mu \sum_j \eta_{_{\mathbf X}} (j)}\]
of the \emph{grand canonical} (probability) distribution for the configurations $\mathbf X$.  This is formally the same as $H \rightarrow H - \mu n$ (although, strictly speaking, the latter is \emph{not} referred to as a ``Hamiltonian'') and the transition rates  in Eq.~\eqref{RRT} may be applied.  Starting with the case $H = 0$, we augment the result of Eq.~\eqref{TTR} with the non--conservative transitions allowing $\eta_{_{\mathbf X}}(k) \rightarrow \eta_{_{\mathbf X}}(k) \pm 1$ at rate proportional to $\text{e}^{\frac{1}{2} \mu} - \eta_{_{\mathbf X}}(k) \text{e}^{-\frac{1}{2} \mu}$
(which may be familiar in the context of birth and death chains).  Inserting the full Hamiltonian, the result for the non--conservative transitions becomes 
\[ \Theta_\e \eta_{_{\mathbf X}}(k) \propto \text{e}^{\frac{1}{2} (\mu - A_{_{\mathbf X}}(k))} - \eta_{_{\mathbf X}}(k) \text{e}^{-\frac{1}{2} (\mu - A_{_{_\mathbf X}}(k))}. \]
Consistent with diffusive scaling, we take the constant of proportionality to be $\e^2$ and add the above $\Theta_\e$ to the old $\Omega_\e$ from Eq.~\eqref{GGT} in order to acquire the full generator.  The resultant discrete dynamics is then seen to be in correspondence with 
\begin{equation}
\label{GCD}
\frac{\partial N}{\partial t}   =  \left[\nabla^{2}N + \nabla\cdot (N \nabla w_N)\right]
+ \left[\text{e}^{\frac{1}{2}(\mu - w_N)} - N\text{e}^{-\frac{1}{2}(\mu - w_N)}\right].
\end{equation}
The equation above is the subject of our analysis.  It is here noted that 
$N \equiv \text{\textsc{m}}_{0}$ is always a stationary solution.  The purpose of this work is to show that under conditions of sufficient thermodynamic stability for $\text{\textsc{m}}_{0}$, 
and suitable conditions on the initial density, 
the density converges to this uniform state exponentially with a rate that is \emph{independent} of the volume.  

\subsection{Statements of Main Theorems}

We conclude this section by stating our main result.  Hereafter, we shall use the notation $\text{\textsc{m}}_{0}$ to denote not only the numerical value but also the \textit{stationary density} that is identically equal to this value; it is assumed that no confusion will arise.

We need a few preliminary definitions: For $\kappa \in (0, \frac{1}{2})$ we define the set of functions
\[ \mathcal B_\kappa = \{ N: \mathbb T_L^d \rightarrow \mathbb R: \kappa \textsc{m}_0 < N < \frac{1}{\kappa} \textsc{m}_0\}.\]
Also, for $\alpha > 0$ we define
\[ v_\alpha = \sup_k |k|^\alpha |\hat W(k)|,\]
\b{where $\hat f(k)$ denotes the $k^{\text{th}}$ Fourier coefficient of $f$:}
\[ \b{\hat f(k) = \int_{\mathbb T_L^d} f(x) e^{ik\cdot x}~dx. }\]
(The factor of $L^d$ is restored in the inverse transformation.) 
Moreover, for a function $Y$ and any $m > 0$, 
\begin{equation}\label{DKBK}\b{\|Y\|_{\mathcal D_m} = \frac{1}{L^d}\sum_k |k|^m |\hat{Y}(k)|.}\end{equation}

The main theorem is as follows:

\begin{thm}[Main Theorem] \label{MDKR}
Let $W$ be an $H$--stable interaction kernel with finite range (i.e., $W$ vanishes outside a ball of finite radius around the origin which is assumed to be small relative to $L$).  Under the regularity assumptions that $v_4 < \infty$ and $\|W\|_{\mathcal D_2} < \infty$ let us suppose that $\text{\textsc{m}}_{0}$ is sufficiently small so that the conclusion of 
Proposition \ref{OGTR} holds for some $\kappa^{\prime} < \frac{1}{2}$.  In addition, suppose the initial 
density $N_{0}$ is in $\mathcal B_{\kappa^{\prime}}$ and $\|\log N_0\|_{\mathcal D_2} < \infty$.

Then we have that for all $t$,
$$
\mathscr G_{\mu}(N_{t}) - \mathscr G_{\mu}(\text{\textsc{m}}_{0})
\leq
\left[\mathscr G_{\mu}(N_{0}) - \mathscr G_{\mu}(\text{\textsc{m}}_{0})\right] \cdot \rm{e}^{-\lambda^{\dagger} t}
$$
for some $\lambda^{\dagger} > 0$.  Moreover, the same type of estimate holds for the $L^{2}$--squared difference with the stationary solution:
$$
\|N_{t} - \text{\textsc{m}}_{0}\|_{L^{2}}^{2}
\leq
\b{\frac{1}{\sigma}}\left[\mathscr G_{\mu}(N_{0}) - \mathscr G_{\mu}(\text{\textsc{m}}_{0})\right] \cdot \rm{e}^{-\lambda^{\dagger} t}
$$
for some $\sigma > 0$. 
\end{thm}

Moreover we also have:

\begin{thm}
\label{thmm}
Equation \eqref{GCD}
induces a natural distance $\mathbb D(\cdot, \cdot)$ defined \b{(at least) for Borel measures which have an $L^2$--density with respect to Lebesgue measure and are bounded below.}  Furthermore, there is a discretization scheme of the JKO--type associated with this distance which converges to the continuum evolution.  In particular, we have exponential decay in $\mathbb D(\cdot, \cdot)$:
\[ \mathbb D(N_t, \textsc{m}_0)^2
\leq
\b{\frac{g^2}{\sigma}}[\mathscr G_{\mu}(N_{0}) - \mathscr G_{\mu}(\text{\textsc{m}}_{0})]\text{e}^{-\lambda^{\dagger} t},
\]
where $\lambda^\dagger$ 
and $\sigma$ are the same as in the 
statement of the
main theorem.
\end{thm}

It is (re)emphasized that the convergence rate is uniform in volume; hence this result may be regarded as a requite first step for the -- as of yet unformulated -- infinite volume study of these fluids. 

\section{Otto Distance \& JKO}
\label{JKO}

For many mass conserving parabolic \textsc{pde}'s -- e.g., in particular Eq.~\eqref{ABC} -- the geometric picture uncovered in \cite{O} (see also the book
\cite{AGS}) has provided indispensable theoretical insight as well as certain practical tools.  However, for mass non--conserved cases,
the generalization of these ideas and their corresponding connection to some version of optimal transportation has not been definitive.  Here, with the tangibles provided by Eq.~\eqref{GCD} along with the functional $\mathscr G_\mu(\cdot)$ from Eq.~\eqref{GCF} that this dynamic has a tendency to minimize,
we may parallel and -- to some extent -- extend, the developments of \cite{O}.  (We refer also to \cite{LM}.)

In this section we will lay out the Riemannian structure underlying our evolution equation by introducing an inner product on the space of measures and an associated distance.  Indeed, it is this underlying structure which motivates and clarifies the eventual exponential convergence to equilibrium.  Associated with a distance is a natural time discretization scheme, i.e., the JKO scheme, which we think of as an infinite dimensional analogue of an Euler scheme.  In \cite{JKO}, minimizers of this scheme are used to yield an approximate (weak) discretization to the underlying evolution; there, the relation to the classical mass conserved transportation problem was used as a conduit between this scheme and the original evolution equation.  

In our case, instead of recourse to an explicitly pre--formulated transportation problem, we shall content ourselves with a Benamou--Brenier (see e.g., \cite{BB}) description of the distance, i.e., it is realized as the infimum over a set of advective transportation possibilities.  Further, we shall consider an \emph{approximation} to the distance (over short times) wherein our analogue of the continuity equation shall be linearized at the initial density.  It is with this approximate distance that we shall  define our JKO--type scheme in the next section.  Our ideology, at least in this work, is therefore that the underlying abstract Riemannian structure should be used as a guide to what is ultimately a very concrete approach.  Thus we shall not provide too many rigorous foundations for our discussions in this section; the basic results establishing that we indeed have a reasonable distance can be found in Appendix B.

Our starting point is to consider a suitable collection 
$\mathcal B$ of Borel measures on
$\mathbb T_{L}^{d}$.  For the purposes of the current work, the setting which leads to the most expedient developments is to consider measures given by a density which is \emph{positive} and is also in $L^2$:
\begin{equation}\label{MDef}
\b{\mathcal B = \{\nu \mbox{ a Borel measure on $\mathbb T_L^d$} \mid \nu \in L^2  \mbox{ and } \nu > 0\}}.
\end{equation}

 What is to follow is motivated by writing
Eq.~\eqref {GCD} in advective form.  The transport velocity field, denoted by
$V$, clearly takes the form 
\footnote{In traditional fluid mechanics, see e.g., \cite{YIH},
it is the \textit{positive} gradient of the velocity potential which produces the velocity field.  We adhere to the convention used in \cite{O}
wherein it is the \textit{negative} gradient.}
$$
V  =  -\nabla \Phi_{N},
\hspace{5 pt}
\text{with}
\hspace{6 pt}
\Phi_{N} := \frac{\delta \mathscr G_{\mu}}{\delta N} = \log N -\mu + w_{N}.
$$
The right hand side of Eq.~\eqref{GCD} is obviously not identically zero.  But, it is noted, it has the same sign as $\Phi_{N}$.  Thus, we may rewrite 
Eq.~\eqref{GCD} in the form:
\begin{equation}
\label{XYZ}
\frac{\partial N}{\partial t} = \nabla\cdot (N\nabla \Phi_{N}) -\Omega_{N}\Phi_{N}.
\end{equation}
Here 
\begin{equation}
\label{XXZ}
\Omega_{N} := \frac{N\text{e}^{-\frac{1}{2}(\mu - w_N)}
- \text{e}^{\frac{1}{2}(\mu - w_N)}}{\log N - \mu + w_N}
\end{equation}
is seen to be positive and tending to a definitive limit (which incorporates into the definition) if both numerator and denominator vanish.  We regard 
Eq.~\eqref{XYZ} as the fundamental advective form for the inhomogeneous case.  In particular, we will say that $N$ is \textit{advected} by $Q$, if it satisfies Eq.~{\ref{XYZ} with $\Phi_{N}$ replaced by $Q$ and with $\Omega_{N}$ exactly as in Eq.~\eqref{XXZ}.

\b{We reiterate that our equation is of a form which is often referred to as one of a \emph{reaction diffusion} type.  It is perhaps worth contrasting our case with the case studied in \cite{LMS} (see the final display of Section 1 therein) and \cite{KMV}: Here, instead of a constant -- or a fixed function, as is studied in \cite{CPSV} -- as the weighting factor for the inhomogeneous term, we have the fully nonlinear term $\Omega_N/N$.}

For $N\in \mathcal B$ let us consider the tangent space, $\mathscr T_{N}$
at $N$.  This is understood as the behavior at time $t = 0$ of all trajectories in 
$\mathcal B$ passing through $N$ at $t = 0$ i.e., possible values of 
$\left. \frac{\partial N}{\partial t}\right |_{t = 0}$.  As in the mass conserved cases, 
these objects are in correspondence with potentials which advectively cause 
$\frac{\partial N}{\partial t}$ to take on this value:  Specifically, 
for $M \in \mathscr T_{N}$ we may define $Q = Q(M)$ to be the potential which satisfies the elliptic equation
\begin{equation}
\label{MQM}
M  =  \nabla\cdot(N\nabla Q) - \Omega_{N}Q.
\end{equation}

For $M_{1}, M_{2} \in \mathscr T_{N}$ it is thus natural to define
\begin{equation}
g_{N}(M_{1}, M_{2})  =  -\int _{\mathbb T_{L}^{d}}M_{1}Q_{2}~dx  
=  -\int_{\mathbb T_{L}^{d}} M_{2}Q_{1}~dx :=
\langle \hspace{-.1cm} \langle  \nabla Q_{1}, \nabla Q_{2} \rangle  \hspace{-.1cm}  \rangle_{N}.
\end{equation}
And so, explicitly, we  have
\begin{equation}
\label{VBPL}
\langle \hspace{-.1cm} \langle  \nabla Q_{1}, \nabla Q_{2} \rangle  \hspace{-.1cm}  \rangle_{N}
= \int_{\mathbb T_{L}^{d}} N(\nabla Q_{1}\cdot \nabla Q_{2}) + \Omega_{N}(Q_{1} \cdot Q_{2})~dx
\end{equation}
which is akin to a Sobolev inner product (for potentials) on $\mathbb T_{L}^{d}$.  It is manifest that $g_{N}(\cdot, \cdot)$
is positive definite and therefore defines a requisite inner product for 
elements of $\mathscr T_{N}$.  

Next, we will demonstrate that Eq.~\eqref{GCD} can  be envisioned as 
the gradient flow of $\mathscr G_{\mu}(\cdot)$ with respect to this metric.  
First, let us use this metric $g_{N}(\cdot, \cdot)$ to define a 
$\mathcal B$--gradient.  Consider a simple functional on $\mathcal B$ of the form
$$
\mathcal J(B)  =  \int_{\mathbb T_{L}^{d}} J(B,x) ~dx
$$
where, e.g., $\mathcal J$ is of class $C^1$.  The directional (G\^ateaux) derivative at $N$
in the direction $M$ is defined by
$$
d\mathcal J(N;M)  := \lim_{\varepsilon \to 0}
\frac{\mathcal J(N + \varepsilon M) - \mathcal J(N)}{\varepsilon}
$$
-- when it exists -- and is given explicitly by
$$
d\mathcal J(N;M) = \int_{\mathbb T_{L}^{d}}
\frac{\delta  J}{\delta N}(N)\cdot M~dx.
$$
Therefore, by analogy with the finite dimensional cases, we use the metric to define the gradient via
$$
d\mathcal J(N;M)  :=  g_{N}(\nabla _{\mathcal B}\mathcal J, M).
$$
In light of the explicit form of the directional derivative, we may identify 
$\nabla _{\mathcal B}\hspace{-.025 cm}\mathcal J$ with the associated advective potential
$\frac{\delta J}{\delta N}$.  

This nearly completes the program.  Consider a weakened version of Eq.~\eqref{XYZ} which in the current language reads
$$
-\int_{\mathbb T_{L}^{d}}
 Q\frac{\partial N}{\partial t} ~dx  =  
\langle \hspace{-.1cm} \langle  \nabla Q, \nabla \Phi_{N}  \rangle  \hspace{-.1cm}  \rangle_{N}
$$
for some test function $Q$.  As above, we denote by $M = M(Q)$ the solution of the advective equation Eq.~\eqref{MQM}.  We remind the reader that in the above display, 
$\Phi_{N} = \log N - \mu + w_{N} = \frac{\delta \mathscr G_{\mu}}{\delta N}$ and so this form of Eq.~\eqref{XYZ} can be written as
$$
g_{N}(M,\frac{\partial N}{\partial t})  =  - g_{N}(M, \nabla _\mathcal B\mathscr G_{\mu} ) ~~~(~= -  \int_{\mathbb T_L^d} M \frac{\delta \mathscr G_\mu}{\delta N}~).
$$
Or, informally, against the backdrop of the given $g_{N}(\cdot, \cdot)$, 
$$
\frac{\partial N}{\partial t}  =  - \nabla _\mathcal B\mathscr G_{\mu};
$$
this then fully justifies the terminology ``gradient flow''.

The above metric $g_{(\cdot)}(\cdot,\cdot)$ allows 
for a definition of distance between elements of $\mathcal B$.  \b{Foremost, for 
$V_{1}, V_{2} \in L^{2}$ which are \emph{vector} valued and $Q_1, Q_2 \in L^2$ which are scalar fields, we may define the inner product akin to that which we defined for gradient fields}
\begin{equation}
\label{SDOU}
\b{\langle \hspace{-.1cm} \langle  (V_{1}, Q_1), (V_{2}, Q_2)  \rangle  \hspace{-.1cm}  \rangle_{N}
:= 
\int N(V_{1}\cdot V_{2}) + \Omega_{N}Q_1Q_2~dx.}
\end{equation}
\b{We emphasize that in this definition there is no \emph{a priori} relationship between $V_1$ and $Q_1$, etc. However, notice that if $V_1 = \nabla Q_1, V_2 = Q_2$, the above notation coincides with our prior use of $\langle \hspace{-.1cm} \langle  \nabla Q_1, \nabla Q_2  \rangle  \hspace{-.1cm}  \rangle_{N}$; both notations will be used and the meaning shall always be clear from the context.}

In what follows (and in general in these contexts)
we will use a subscript of $t$ to denote time dependence -- not to be confused with a partial derivative.  
Then, for $N_{0}$, $N_{1}$ in $\mathcal B$ we may consider the set of \b{vector and scalar field pairs} which drive $N_{t}$ from $N_{0}$ at $t = 0$ to $N_{1}$ at $t = 1$ \b{according to the dynamics in the below display} in such a way that $\frac{\partial N_t}{\partial t}$ remains in $L^2(\mathbb T_L^d \times (0, 1))$:
\begin{equation}\label{BDF}\begin{split}
\mathscr V(N_{0}, N_{1})  := \{\b{V \in L^2(N_t), Q} &\in L^2(\Omega_{N_t}) \mid 
\frac{\partial N_{t}}{\partial t} + \nabla \cdot (N_{t}V) = \b{-\Omega_{N_{t}}Q}\\
&\text{ with } N_{t=0} = N_{0} \text{, } N_{t = 1} = N_{1} \mbox{ and } \frac{\partial N_t}{\partial t} \in L^2\}.
\end{split}\end{equation}

We claim that the set $\mathscr V(N_{0}, N_{1})$ is non--empty since we may consider the straight line path
$N_{t}  =  (1-t)N_{0} + tN_{1}$ and find a (time dependent) gradient field which drives $N$ along this path.  Indeed, here, $\frac{\partial N_t}{\partial t}$ is given by $N_1 - N_0$ which is in $L^2$.  Now given a curve in $\mathcal B$ indexed by $N_t$, we may consider the Hilbert space (for potentials) equipped with the inner product $(\phi, \psi)_{N_t}$ given by \b{$\int_0^1 \langle \hspace{-.1cm} \langle (\nabla \phi,\phi), (\nabla \psi, \psi) \rangle  \hspace{-.1cm}  \rangle_{N_{t}}~dt$}.  Since $N_t$ is bounded from below, it turns out that $\Omega_{N_t}$ is also bounded below (c.f., Eq.\eqref{CQGR}).  Thus, the $L^2$--norm of a potential $\phi$ is bounded above by a constant times the norm induced by the Hilbert space.  It then follows that (integration against) $N_1 - N_0$ can be viewed as a bounded linear functional on the Hilbert space and so the required driving gradient field is existentiated by the Riesz Representation Theorem.

We now define the distance $\mathbb D$ via
\begin{equation}
\label{GBVH}
\b{\mathbb D^{2}(N_{0}, N_{1})  =  \inf_{(V, Q) \in \mathscr V(N_{0}, N_{1})}
\int_{0}^{1}
\langle \hspace{-.1cm} \langle (V, Q),(V, Q) \rangle  \hspace{-.1cm}  \rangle_{N_{t}}~
dt,}
\end{equation}
or, equivalently, for $(V, Q)$'s in $\mathscr V_{T}(N_{0}, N_{1})$ which drive 
$N_{0}$ to $N_{1}$ on $[0,T]$,
$$
\b{\mathbb D^{2}(N_{0}, N_{1})  =  \inf_{(V, Q) \in \mathscr V_{T}(N_{0}, N_{1})}
T\int_{0}^{T}
\langle \hspace{-.1cm} \langle (V, Q),(V, Q) \rangle  \hspace{-.1cm}  \rangle_{N_{t}}~
dt.}
$$
\b{We remark that while the minimization problem is envisioned as minimizing over all paths $N_t: N_0 \leadsto N_1$, in fact the only paths which are conceivably of interest are those which can be achieved by some $(V, Q)$ as described.  Since all of this is already encoded in the definition of $\mathscr V(N_0, N_1)$, minimization of the functional over this set is appropriate and sufficient.}
It can be demonstrated that $\mathbb D^{2}(\cdot,\cdot)$ is indeed the square of a distance which separates points and that for all intents and purposes, any minimization program for $\mathbb D^{2}(\cdot,\cdot)$ may be carried out by considering only those fields which are derived from a velocity potential. 
These results have been collected in Appendix B.

\begin{remark}
Here we emphasize that the existence of a distance between points in $\mathcal B$ (and one may hope to presume all Borel measures on $\mathbb T_L^d$) automatically defines an (abstract) optimal transport problem in this context: Indeed, the explicit realization of the distance as an infimum implies a transport problem wherein the ``optimal path'' minimizes the relevant functional.  It is unfortunate that these problems have not been tied to an explicit Monge--Ampere or Kantorovich type formulation.
\end{remark}

Having introduced the preceding metric structure on $\mathcal B$ and demonstrated the gradient flow properties of Eq.~\eqref{XYZ} for the functional 
$\mathscr G_{\mu}(\cdot)$ with respect to this metric, we may then consider the following JKO--type scheme: 
\begin{equation}
\label{VFKIL}
N_{t + h}  =  
\text{Argmin}\{\frac{1}{2}\mathbb D^{2}(N_{t}, N)  + h\mathscr G_{\mu}(N)\}.
\end{equation}
This is a direct generalization of the scheme in \cite{JKO} to these inhomogeneous cases.

\section{The Approximate Functional}

In this section we will proceed to construct an approximate functional whose minimizers will explicitly yield a discretization of our equation.  It should be emphasized that JKO--type functionals, even when summed up over all iterations, do not admit a meaningful $h$ tends to zero functional to be minimized -- these are dissipative systems.  In this sense, all such functionals are finite $h$ ``approximates''.   \b{An alternative approach to discretization (which may have applicability to the system studied here) is to construct regularized functionals, e.g., the so--called WED functional.  Again in this case, while there is strictly speaking no limiting functional, the limit of the minimizers does correspond to a solution of the original system.  See  \cite{RSSS}, \cite{MO}, \cite{S} and references therein. }

Here for motivational purposes it is worthwhile to understand the difference between our situation and the mass conserved case as treated in \cite{JKO}.   In the latter, the \emph{exact} approximate functional (e.g., as displayed in Eq.~\eqref{VFKIL}) was employed.  It was found that the minimizers were an approximate discretization converging to the relevant dynamics.  To accomplish these ends, virtually all of the existing machinery of optimal transportation were deployed.  This includes, but is not limited to: A well formulated and well studied underlying transportation problem, the coupled measure description for the Wasserstein distance, the pushforward formalism, a relation between the Wasserstein distance and variance, and, finally, the connection with the Benamou--Brenier description via transport fields.    

The key difference here is that no such ancillary machinery has as of yet been developed for non--conservative problems.  Indeed, \emph{all} we have is the Benamou--Brenier formalism -- which here defines the distance itself.  Thus, instead of deploying the exact approximate functional, we shall use an \emph{approximate} approximate functional whose \emph{exact} minimizers provide a discretization.  The principle difficulty in our approach is that the discretization arrived at is not as viable as the discretization acquired in \cite{JKO} which (still only) approximated the minimizers.  Hence, here, to obtain the $h$ tends to zero limiting dynamics, an arduous, albeit elementary analysis is required.  However, these technicalities can be neatly quarantined and are the subject of Appendix A. 

\subsection{Definition and Minimization}
The starting point of our program entails a discretization of the distance itself (for small times).  
Let $h > 0$ which we envision to be small and consider times $0 \leq t \leq h$.  Let us replace the previously described distance functional by one where  $N_{t}$ is replaced in two crucial places by $N_{0}$.  
In particular, for all intents and purposes, under the auspices of $h \ll 1$ we are replacing
$N_{t}$ with $N_{0}$ in the inner product:
$\langle \hspace{-.1cm} \langle  \cdot  , \cdot   \rangle  \hspace{-.1cm}  \rangle_{N_{t}}
\to 
\langle \hspace{-.1cm} \langle   \cdot ,  \cdot  \rangle  \hspace{-.1cm}  \rangle_{N_{0}}$
and allowing this to inherit into the (approximate) dynamics.  Starting with the latter, for fixed $\phi$
we write
\begin{equation}
\label{UIC}
\frac{\partial N_{t}}{\partial t}  =  
\nabla \cdot (N_{0}\nabla \phi) - \Omega_{N_0}\phi.
\end{equation}
Then the approximate distance is defined as
$$
\b{\mathbb D_{A}^2(N_{0}, N_h)} := \inf_\phi
\int_{\mathbb T_{L}^{d}} h \int_{0}^{h}N_{0}|\nabla \phi|^{2} + \Omega_{{N_0}}\phi^{2}~dtdx 
$$
where under the above approximate dynamics, $\phi$ gets us to $N_h$ at time $t = h$. \b{(We reiterate that since $N_h$ is considered fixed, corresponding to each $\phi$ is an interpolating curve $N_t$ from $N_0$ to $N_h$.)}  With $\phi$ as argument (not necessarily minimizing anything) we will denote the right hand side by 
$\mathbb {E}_{A}(\cdot)$:
$$
\mathbb {E}_{A}(\phi) :=  
h \int_{0}^{h}
\langle \hspace{-.1cm} \langle \nabla \phi , \nabla \phi \rangle  \hspace{-.1cm}  \rangle_{N_{0}}dt =
\int_{\mathbb T_{L}^{d}} h \int_{0}^{h}N_{0}|\nabla \phi|^{2} + \Omega_{{N_0}}\phi^{2}~dtdx.
$$
Under reasonable conditions, we expect that for fixed $N_{0}$ there is a unique \textit{static} field which drives the system to $N_{h}$ at time $t = h$.  (See Eq.~\eqref{KHYU} in the statement of Proposition \ref{YUDQ} below.)  Since we will be utilizing Hilbert space structures, it is pertinent now to introduce notation for the relevant space of driving fields.

\begin{definition}
We let $\mathcal H_{N_0}$ denote the Hilbert space (of driving fields) with the weighted inner product 
\[\langle \hspace{-.1cm} \langle \nabla \phi , \nabla \psi \rangle  \hspace{-.1cm}  \rangle_{N_{0}} = \int_{\mathbb T_L^d} N_0 (\nabla \phi \cdot \nabla \psi) + \Omega_{N_0} \phi \psi~dx.\]
The dual space will be denoted by $\mathcal H_{N_0}^{-1}$.
\end{definition}

Our first observation is that the static field $\phi$ described above actually minimizes the approximate distance functional: 

\begin{proposition}
\label{YUDQ}
For fixed $N_h - N_0\in \mathcal H_{N_0}^{-1}$ and any driving field $\varphi$, let 
$\mathbb D_{A}(N_{0}, N_{h})$ and $\mathbb {E}_{A}(\varphi)$ be as described.  Then the minimum 
for $\mathbb D_{A}(N_{0}, N_{h})$
is achieved by the unique static $\phi \in \mathcal H_{N_0}$ which satisfies
\begin{equation}
\label{KHYU}
\frac{N_{h} - N_{0}}{h}  =  
\nabla \cdot (N_{0}\nabla \phi) - \Omega_{N_0}\phi.
\end{equation}
\end{proposition}
\begin{proof}
Since $N_h - N_0$ is a bounded linear functional on $\mathcal H_{N_0}$, the existence (and uniqueness) of the required $\phi$ again follows directly from the Riesz Representation Theorem. 

Let us adapt the temporary notation $N_{t}^{[\varphi]}$ for a density driven,
according to the approximate dynamics,
in the time interval 
$0 \leq t \leq h$ by the field $\varphi$.  A general driving field which achieves $N_{h}$ at $t = h$ may be written in the form $\phi + \alpha$ with $\alpha$ (necessarily) depending on time.  We have, weakly,
\begin{align}
\frac{\partial}{\partial t} N_{t}^{[\phi + \alpha]} &=  
\nabla\cdot [N_{0}\nabla (\phi + \alpha)]  - \Omega_0(\phi + \alpha)
\notag
\\
&= \frac{\partial}{\partial t} N_{t}^{[\phi]} + 
\nabla\cdot (N_{0}\nabla \alpha)  - \Omega_0\alpha.
\end{align}
It therefore follows that if $\psi$ is a suitable time independent test function then
$$
0 = \int_{0}^{h}\int_{\mathbb T_{L}^{d}}
 \psi(\nabla\cdot (N_{0}\nabla \alpha)  - \Omega_{0}\alpha)~dxdt
= - \int_{0}^{h}\int_{\mathbb T_{L}^{d}}
N_{0}(\nabla \psi \cdot \nabla \alpha) + \Omega_0\psi\alpha~dxdt.
$$
In particular, plugging in $\phi$, we have 
$$
\int_{0}^{h}\int_{\mathbb T_{L}^{d}}
N_{0}(\nabla \phi \cdot \nabla \alpha) + \Omega_0\phi\alpha~dxdt
= 0.
$$
Now we consider $\mathbb E_{A}(\phi + \alpha)$:
$$
\mathbb E_{A}(\phi + \alpha)  =  
h\int_{0}^{h}\int_{\mathbb T_{L}^{d}}
N_{0}(|\nabla \phi + \nabla \alpha|^{2}) + \Omega_0(\phi + \alpha)^{2}~dxdt
= \mathbb E_{A}(\phi) + \mathbb E_{A}(\alpha)
$$
where, by the preceding display, the cross term has vanished.  Since 
$\mathbb E_{A}(\alpha)$ is positive, the result is established.  
 \end{proof}
 
\begin{definition}\label{DKR}
Given a fixed $N_0$, let us now consider the JKO type functional associated with $\mathbb D_{A}$: $$
\mathbb J_{A}(N_{0},N) := \frac{1}{2}\mathbb D_{A}^2(N_{0}, N) + h\mathscr G_{\mu}(N).
$$
\end{definition}

\begin{remm} \b{Let us observe that if $N_0 \in \mathcal B$ then in fact $N_0 \in \mathcal H_{N_0} ^{-1}$: Indeed, we have
$$|\int_{\mathbb T_L^d} N_0 \phi~dx| \leq \|N_0\|_1 \|\sqrt{N_0}\phi\|_{2}\leq \|N_0\|_1 \|\phi\|_{\mathcal H_{N_0}}.$$}\end{remm}

We first show that the functional $\mathbb J_A(N_0, \cdot)$ can be minimized.

\begin{prop}
\label{UMM}
Let $N_0 \in \mathcal B$.  Then the functional 
$\mathbb J_A({N_0},\cdot)$ has a  minimizer in $\mathcal H_{N_0}^{-1}$.  Furthermore, this minimizer is in $L^1$.
\end{prop}

\begin{proof}
For any $N_0$, we easily have that $\mathbb J_A(N_{0},\cdot)$ is bounded below.  Explicitly, the function $N \log N - (1+\mu)N$ is minimized at $N = \text{e}^\mu$ with value \b{$-\text{e}^{\mu}$}
whereas the term involving $W$ is positive by H--stability so (since we are in finite volume) the full free energy integral is bounded below.  The distance term is of course positive.

Let us then take some minimizing sequence $N^{(j)}$ in $\mathcal H_{N_0}^{-1}$.  By the observation in Definition \ref{DKR}, since $N_0 \in \mathcal B$, it is the case that $N_0 \in \mathcal H_{N_0}^{-1}$ and so $N^{(j)} - N_0 \in \mathcal H_{N_0}^{-1}$.  We now consider the driving fields $\phi^{(j)}$ corresponding to $N^{(j)}$ as given in Proposition \ref{YUDQ} so that 
\begin{equation}\label{EEBF} N^{(j)} - N_0 = h \left[\nabla \cdot (N_0 \nabla \phi^{(j)}) - \Omega_{N_0} \phi^{(j)}\right].\end{equation}
Now 
\[\mathbb D_{A}^2(N_0, N^{(j)}) = h \int_0^h \int_{\mathbb T_L^d} N_0 |\nabla \phi^{(j)}|^2 + \Omega_{N_0} (\phi^{(j)})^2~dxdt \]
must be bounded since the free energy is bounded below 
and, further, the right hand side is just $h^{2}$ times 
$\langle \hspace{-.1cm} \langle \nabla \phi^{(j)} , \nabla \phi^{(j)} \rangle  \hspace{-.1cm}  \rangle_{N_{0}}$.
We may therefore assert that along some further subsequence, if necessary, $\phi^{(j)}$
converges weakly
with respect to the inner product structure to some $\phi^{*} \in \mathcal H_{N_0}$.  Let us next define $N^{*}$ as the density corresponding to this $\phi^{*}$: We let $N^* \in \mathcal H_{N_0}^{-1}$ be such that for all $\psi \in \mathcal H_{N_0}$,
$$
N^*[\psi] = \int_{\mathbb T_L^d} N_0 \psi ~dx- h\int_{\mathbb T_L^d} N_0 (\nabla \phi^* \cdot \nabla \psi) + \Omega_{N_0} \phi^*\psi~dx.
$$

On the basis of the weak convergence of the $\phi^{(j)}$'s we claim that the $N^{(j)}$'s have a weak limit (in $\mathcal H_{N_0}^{-1}$) and that $N^{*}$ is this limit.  
Indeed, letting $\psi$ denote some suitable test function, we have
\begin{align}
\lim_{j\to\infty}\int_{\mathbb T_{L}^{d}}N^{(j)} \psi~ dx  &=  \int_{\mathbb T_L^d} N_0 \psi~dx - h\lim_{j\to\infty}\int_{\mathbb T_{L}^{d}} N_{0}(\nabla\phi^{(j)} \cdot \nabla\psi) + \Omega_{N_{0}}\phi^{(j)}\psi~dx
\notag
\\
&= \int_{\mathbb T_L^d} N_0 \psi ~dx- h\int_{\mathbb T_L^d} N_0 (\nabla \phi^* \cdot \nabla \psi) + \Omega_{N_0} \phi^*\psi~dx\\
&= N^*[\psi].
\end{align}
(We remark that the above realization of $N^*$ as a weak limit also implies that it is nonnegative.)

On the other hand, we claim that $N^*$ is in fact (at least) an $L^1$--function: It is the case that $N^{(j)} \log N^{(j)}$ is integrable and its integral is
uniformly bounded and so it follows (by Jensen's inequality) that $\|N^{(j)}\|_{L^1}$ is uniformly bounded.  Thus we assert that the associated measures converge vaguely and that the limit can be represented by an $L^1$--function which can then be identified with $N^*$ (see for example the exposition in \cite{DST}). 

We now claim that 
\[ \liminf_{j \rightarrow \infty} \mathbb J_A(N_0, N^{(j)}) \geq \mathbb J_A(N_0, N^*).\]
The lower semicontinuity of the terms involving \b{$N \log N - (1+ \mu) N$} and the $\mathbb D_A^2(N_0, N)$ term follow directly from convexity (indeed, $\mathbb D_A^2(N_0, N^{(j)})$ is explicitly convex in the variables $\phi^{(j)}$).  

Now we address the interaction term.  First note that for any function $M$, we have
\[ \int_{\mathbb T_L^d \times \mathbb T_L^d} W(x - y) M(x) M(y) ~dxdy = \frac{1}{L^d}\sum_k \hat{W}(k) |\hat{M}(k)|^2.\]
By the convergence of the $N^{(j)}$'s to $N^*$, it is clear that for any fixed $k$, we have
\[ \hat{N}^{(j)}(k) \rightarrow \hat{N}^*(k).\]
Let us obtain an \emph{a priori} estimate for $\hat{N}^{(j)}(k)$: Explicitly, we have that  
\[\begin{split}
(\hat{N}^* - \hat{N}^{(j)})(k) &= - h \int_{\mathbb T_L^d} e^{ikx} \left[ik \cdot N_0 ((\nabla \phi^* - \nabla \phi^{(j)}) +\Omega_{N_0} (\phi^* - \phi^{(j)}) \right]~dx.
 \end{split}\]
Taking absolute values and using Cauchy--Schwarz, we see that 
\[ |(\hat N^* - \hat N^{(j)})(k)| \leq G|k|\]
for some $G < \infty$ (for $k$ sufficiently large).

Now we apply the formula for the convolution displayed above to the quantity $\int_{\mathbb T_L^d} (W* (N^* - N^{(j)}))(N^* - N^{(j)})~dx$ to show that it tends to zero: We obtain (dropping the factor of $\frac{1}{L^d}$)
\[ \sum_k \hat{W}(k) |(\hat N^* - \hat N^{(j)})(k)|^2 = \sum_{|k|< k_0} \hat{W}(k) |(\hat N^* - \hat N^{(j)})(k)|^2+ \sum_{|k| \geq k_0} \hat{W}(k) |(\hat N^* - \hat N^{(j)})(k)|^2\]
for some fixed $k_0 \gg 1$.  As $j$ tends to infinity, the first term tends to zero.  For the second term, using the estimate derived above, we are left with 
\[ \sum_{|k| \geq k_0} \hat{W}(k) |(\hat N^* - \hat N^{(j)})(k)|^2 \leq G^2\sum_{k\geq k_0} k^2|\hat W(k)|. \]
Since $\|W\|_{\mathcal D_2} < \infty$, the right hand side is the tail of a convergent sum and can be made arbitrarily small.  We 
conclude that $\lim_{j \rightarrow \infty} \int_{\mathbb T_L^d} (W*N^{(j)})N^{(j)}~dx = \int_{\mathbb T_L^d} (W*N^*)N^*~dx$.  

It follows that  
\[ \inf\{\mathbb J_A(N_0, N), N \in \mathcal H_{N_0}^{-1}\} = \lim_{j \rightarrow \infty} \mathbb J_A(N_0, N^{(j)}) \geq \mathbb J_A(N_0, N^*) \]
and so indeed $N^*$ is the minimizing element of $\mathcal H_{N_0}^{-1}$.  
\end{proof}

We will hereafter refer to the minimizer found in the above as $N_h$; while we cannot yet claim that $N_h$ is uniformly bounded below, we do have: 

\begin{prop}\label{MIN}
Let $N_h \in \mathcal H_{N_0}^{-1} \cap L^1$ denote the minimizer of $\mathbb J_A(N_0, \cdot)$ as given in Proposition \ref{UMM}.  Then $N_h$ is positive almost everywhere.
\end{prop}

\begin{proof}
Let \b{$N \in \mathcal H_{N_0}^{-1} \cap L^1$} denote any nonnegative function for which $\mathbb J_A(N_0, N)$ is finite and let 
\[ \mathcal S_0 = \{x: N(x) = 0\}.\]
\b{Note that $\mathcal S_0$ is measurable since it is the complement of $\mbox{supp}(N)$.}
If it were the case that $\mathcal S_0$ has positive (Lebesgue) measure, then, we claim, it is possible to modify $N$ so as to lower $\mathbb J_A(N_0, \cdot)$.  \b{Indeed, let $n$ be the indicator function of $\mathcal S_0$ so that $\int_{\mathcal S_0} n(x)~dx = :n_0 > 0$ is the size of $\mathcal S_0$}.  Now consider the modification $N \mapsto  N + \e n$ for some (small) $\e > 0$.  The key observation is that the effect of this modification on all terms contributing to $\mathbb J_A(N_0, \cdot)$ \emph{except} the entropy term (i.e., the $N \log N$ term) is of order $\e$.  

\b{We first observe that certainly $n \in \mathcal H_{N_0}^{-1} \cap L^1$ and so by Proposition \ref{YUDQ}, there is some $\psi$ so that 
$$n = \nabla \cdot (N_0 \nabla \psi) - \Omega_{N_0} \psi.$$
It therefore follows that $\varphi + \e \psi$ will drive $N_0$ to $N + \e n$.}  For the distance squared term, note that e.g., $\mathbb D_A^2(N_0, N + \e n) \leq  h^2 (\|\phi\|_{\mathcal H_{N_0}}  + \e \|\psi\|_{\mathcal H_{N_0}})^2$.  The interaction term also has a linear (and quadratic) $\e$ modification with bounded coefficients.   Meanwhile, 
\[ \int_{\mathbb T_L^d} (N + \e n) \log (N + \e n) - N \log N ~dx = \int_{\mathcal S_0} n\e \log \e n~dx = n_0\e \log \e \]
which is negative and of considerably larger magnitude as $\e$ tends to zero.

Thus, since $N_h$ is a minimizer, the stated result follows. 
\end{proof}

\subsection{Discretization}

We are now ready to show that successively running our JKO type scheme yields a discretization of our equation.  

\begin{proposition}\label{DSIC}
Let \b{$N_h \in \mathcal H_{N_0}^{-1} \cap L^1$} denote the minimizer of $\mathbb J_{A}(N_0, \cdot)$ as given in Proposition \ref{UMM}. Then $N_0, N_h$ yield a weak discretization of the dynamics in Eq.~\eqref{GCD}. I.e., for all $\psi \in \mathcal H_{N_0}$, 
\begin{equation}
\label{XTCL}
\int_{\mathbb T_L^d} \frac{N_{h} - N_{0}}{h}~\psi  =  - \int_{\mathbb T_L^d} N_{0} (\nabla \Phi_{N_{h}} \cdot \nabla \psi) + \Omega_{N_{0}}\Phi_{N_{h}}\psi,
\end{equation}
i.e., weakly,
\begin{equation}
\label{CVDA}
\b{\frac{N_{h} - N_{0}}{h}} =
\nabla\cdot(N_{0}\nabla\Phi_{N_{h}}) - \Omega_{N_{0}}\Phi_{N_{h}}.
\end{equation}
Further, $\Phi_{N_h} \in \mathcal H_{N_0}$.
\end{proposition}
\begin{proof}
Let us denote by $\phi \in \mathcal H_{N_0}$ the corresponding (static) field which drives the system from $N_{0}$ to $N_{h}$ in the time interval $0 \leq t \leq h$ under the dynamics in Eq.~\eqref{UIC}, as given by Proposition \ref{UMM} (the $\phi$ here corresponds to the $\phi^*$ in the proof of Proposition \ref{UMM}).  Temporarily, letting $\kappa > 0$, we consider the variation
\b{$N_{h} \mapsto N_{h} + \varepsilon \eta$} with
(bounded)
$\eta \in \mathcal H_{N_0}^{-1}$ that is supported on the set $\{N_h(x) > \kappa\}$.

Now there is a corresponding variation in the driving field which we denote by $\varepsilon \psi$, so that  $\phi \mapsto \phi + \varepsilon\psi$ ``drives'' $N_0$ to $N_h + \e \eta$.  Since the relevant equations are linear, $\psi$ and $\eta$ are simply related via
\begin{equation}\label{ETTA}\eta  =  \nabla \cdot (N_{0}\nabla \psi) - \Omega_{N_0}\psi
\end{equation}
and so given $\eta$, the required $\psi \in \mathcal H_{N_0}$ is given by Proposition \ref{YUDQ}.

Now to lowest order in $\varepsilon$,
\begin{equation}
\label{ASB}
\mathscr G_{\mu}(N_{h}) \to 
\mathscr G_{\mu}(N_{h}) + \varepsilon\int_{\mathbb T_{L}^{d}} \eta\frac{\delta \mathscr G_{\mu}}{\delta N}~dx
\noindent = \mathscr G_{\mu}(N_{h}) + \varepsilon\int_{\mathbb T_{L}^{d}} \eta\Phi_{N_{h}}~dx.
\end{equation}
It is readily verified that all higher order terms divided by $\e$ tend to zero as $\e$ tends to zero (all coefficients are explicitly bounded since $\eta$ is supported only where $N_h > \kappa$).

Let us turn attention to the distance--type term.  Here we have, exactly,
\[ \begin{split}
&~~~~~\mathbb D_{A}^{2}(N_0, N_h + \varepsilon \eta) - 
\mathbb D_{A}^{2}(N_0, N_h) \\
&= - \int_{\mathbb T_L^d} (N_h + \e \eta - N_0) (\phi + \e \psi) ~dx
- \int_{\mathbb T_{L}^{d}} (N_h - N_0) \phi ~dx\\
&= -\int_{\mathbb T_L^d} \e\eta \phi + \e(N_h - N_0) \psi ~dx - \e^2 \int_{\mathbb T_L^d} \eta  \psi ~dx;
\end{split}
\]
it is clear that the $\e^2$ term can be neglected.
We now claim that the $(N_h - N_0) \psi$--term reproduces the $\eta \phi$--term: Indeed we have, from  Eq.~\eqref{KHYU}, that  
\begin{equation}\label{DDDS}
\int_{\mathbb T_L^d} (N_h - N_0) \psi~ dx = h\int_{\mathbb T_L^d} \left(\nabla \cdot (N_0 \nabla \phi) - \Omega_{N_0} \phi\right) \psi ~dx = -h\langle \hspace{-.1cm} \langle \nabla \phi, \nabla \psi \rangle  \hspace{-.1cm}\rangle_{N_0}.
\end{equation}
Since the inner product is symmetric, after a formal integration by parts, the role of $\phi$ and $\psi$ can be exchanged and we use the weak form of the elliptic equation defining $\psi$ (as in Eq.\eqref{ETTA}) to replace the expression involving $\psi$ with $\eta$.  

In combination with Eq.\eqref{ASB} we now see that the stationarity condition for the minimizer of $\mathbb J_A(N_0, \cdot)$ yields 
\[ \int_{\mathbb T_L^d} \eta (\phi - \Phi_{N_h})~dx = 0.\]
This implies that $\Phi_{N_h} = \phi$ on the set $\{N_h > \kappa\}$.  By Proposition \ref{MIN}, \b{the sets $\{N_h > \kappa_n\}$ for $\kappa_n \rightarrow 0$ are exhaustive and so $\kappa > 0$ can be made arbitrarily small and we see that $\Phi_{N_h} = \phi$ a.e.}  Since $\phi \in \mathcal H_{N_0}$ we also conclude that $\Phi_{N_h} \in \mathcal H_{N_0}$.  

Now to reproduce some discretization of the dynamics, we replace $\phi$ by $\Phi_{N_h}$ \textcolor{blue}{on the right hand side of} Eq.\eqref{DDDS} to obtain 
\begin{equation}
\label{XTCM}
0 = \int_{\mathbb T_{L}^{d}}\left(\frac{N_h - N_0}{h}\right) \psi + 
N_{0}(\nabla \Phi_{N_{h}} \cdot \nabla \psi) + \Omega_{N_0}\Phi_{N_{h}}\psi~dx
\end{equation}
for all $\psi \in \mathcal H_{N_0}$; i.e., weakly,
Eq.~\eqref{CVDA} is satisfied.
\end{proof}

For $W$ of positive type, the overall $\mathbb J_A(N_0,\cdot)$
is \b{strictly} convex and uniqueness of $N_h$ is guaranteed.  In the more general circumstances of present interest, uniqueness will be established under the restrictive (presumably unnecessary) hypothesis that $N_0$ is classical.
\begin{lemma} \label{UUNQ}
\b{Given $N_0 \in\mathcal B$}, for $h$ sufficiently small depending only on \b{$N_0$ and various norms on $W$} there is a unique solution to Eq.\eqref{CVDA} such that \b{$N_h \in L^1$.} \b{In particular, the minimizer for $\mathbb J_A(N_0, \cdot)$ from Proposition \ref{DSIC} is unique and so in fact $\log N_h \in \mathcal H_{N_0}$.}
\end{lemma}

\begin{proof}

Let $N_a$, $N_b \in L^1$ denote two purportedly different solutions to 
Eq.~\eqref{CVDA}.  We define $\Psi_a := \log N_a$ and similarly for $\Psi_b$.  We also define 
\[ N_{ab} := N_a - N_b, ~~~\Psi_{ab} := \Psi_a - \Psi_b.\]
From Eq.~\eqref{CVDA} we see that $N_{ab}$ satisfies 
\[ N_{ab} = h \left[\nabla \cdot (N_0 \nabla \Psi_{ab}) - \Omega_{N_0} \Psi_{ab}\right] + h\left[\nabla \cdot (N_0(\nabla W*N_{ab})) - \Omega_{N_0} (W*N_{ab}) \right].\]
Assuming towards a contradiction that $N_{ab}$ is not identically zero, we wish to consider a set which we denote by $\mathcal S$ where the value $N_{ab}$ is sufficiently large. 

Let us examine the difference of 
$N_a$ and $N_b$ subtracting a fraction \b{$hc_W > 0$} from the left hand side where $c_W$ is a constant to be determined shortly:
\begin{equation}
\label{GBLXR}
\begin{split}
(1- hc_W)N_{ab}  
&=
h\left[ \nabla\cdot (N_0 \nabla \Psi_{ab}) -\Omega_{N_0}\Psi_{ab}\right]
\\
& +h\left[\nabla\cdot(N_0\nabla (W*N_{ab})) - \Omega_{N_0}(W*N_{ab})
 - c_W N_{ab})\right].
\end{split}\end{equation}
We claim \b{that on some set (corresponding to the $\mathcal S$ alluded to above)} with a proper choice of $c_W$, the terms on the second line of the above display total to a quantity which is pointwise negative, i.e.,
\[\begin{split}
\b{-hV_{ab}(x)} :=  h\left[\nabla\cdot (N_0\nabla (W*N_{ab})) - \Omega_{N_0}W_{ab} - c_WN_{ab}\right](x) < 0
\end{split}\]
for $x$ in the presumed set.
\b{The fact that here $-V_{ab} < 0$ is pertinent  to the remainder of the argument and to establish this negativity, we will need to consider the cases where $N_{ab} \in L^\infty$ and $N_{ab} \notin L^\infty$ separately.  }

First suppose $N_{ab} \in L^\infty$ and let $m_{ab} = \|N_{ab}\|_\infty$.  In this case we let 
\begin{equation}\label{sss} \mathcal S = \{ N_{ab} > \frac{m_{ab}}{2}\},\end{equation}
where without loss of generality we may assume that this set is of positive measure.  For example, for $x \in \mathcal S$, the term $N_0 (\nabla^2W*N_{ab})$ is easily bounded:
\[ \begin{split} |N_0(x) \cdot (\nabla^2W*N_{ab})(x)|
&=
|~N_0(x) \int_{\mathbb T_L^d} \nabla^2 W(x - y)N_{ab}(y)~dy~|\\
&\leq 
m_{ab} \cdot W_2\|N_0\|_\infty,
\end{split}\]
where $W_2 = \int_{\mathbb T_L^d} |\nabla^2 W(y)|~dy$.
The other terms are bounded proportional to $m_{ab}$ as well \b{with constants now involving $\|\nabla N_0\|_\infty$, $W_1$ (with $W_1$ defined similarly to $W_2$) and $\|\Omega_{N_0}\|_\infty$} (which is finite since $N_0 \in L^\infty$).  
Now since $x \in \mathcal S$, we have $N_{ab}(x) > \frac{1}{2}m_{ab}$, so the negative term $-c_WN_{ab}$ can be made to compensate for any positive contributions from the other terms for $c_W$ sufficiently large depending \b{not on $h$ but} only on the particulars of $W$ and $N_0$.

Let us now address the case where $N_{ab} \notin L^\infty$.  We claim that a modification of the preceding argument also shows $-V_{ab} < 0$ on a modified version of $\mathcal S$.  To this end let us define 
\[ M_{ab} = \sup_{x \in \mathbb T_L^d} \int_{B_a(x)} |N_{ab}(y)|~dy\]
where $B_a(x)$ is the ball of radius $a$ around $x$, where we recall that the range of $W$ is also denoted by $a$.  ($M_{ab}$ is guaranteed to be finite since $N_{ab} \in L^1$ but is ostensibly independent of the total volume.)  
Here let us define 
\[ \mathcal S = \{N_{ab} > M_{ab}\}. \]

Since $W(x - y)$ vanishes outside of $B_a(x)$, it follows that e.g., 
\[\begin{split} |\nabla N_0(x) \cdot \nabla(W*N_{ab})(x)| & \leq |\nabla N_0(x)| \cdot |~\int_{B_a(x)} |\nabla W(x - y) N_{ab}(y)~dy~| \\
&\leq M_{ab} \cdot \|\nabla W\|_\infty \|\nabla N_0\|_\infty .
\end{split}\]
Similar estimates hold for the other terms and so the conclusion follows as before. \b{We note particularly from Eq.~\eqref{GBLXR} that the term $M_{ab}$ is directly suppressed by $N_{ab}$ on the set $\mathcal S$ and so as before $c_W$ only depend on $N_0$ and $W$ and \emph{not} on $N_{ab}, \Psi_{ab}$ or $h$.}

Next we will expand the left hand side of Eq.~\eqref{GBLXR} using the notation 
$$N_{ab} = \Psi_{ab} + [\mathcal E_2(\Psi_a) - \mathcal E_2(\Psi_b)],$$
where $\mathcal E_2(x) = e^x - (1+ x)$.  After some rearrangement, Eq.~\eqref{GBLXR} becomes 
\[ \begin{split}
\Psi_{ab} &= \frac{h}{1- hc_W } \left[\nabla \cdot (N_0 \nabla \Psi_{ab})\right]\\
&- \frac{h}{1- hc_W} \left[\Omega_{N_0} \Psi_{ab} + V_{ab} \right] - \left[\mathcal E_2(\Psi_a) - \mathcal E_2(\Psi_b)\right]. 
\end{split}\]
First let us observe that the second line in the above equation is pointwise negative for $x \in \mathcal S$; we will denote the entirety of the second line by $-P_{ab}$.  Next let us define 
\[ \mathbb K(\cdot) = - \frac{1}{1- hc_W} \nabla \cdot [N_0 \nabla (\cdot)].\]

The equation now takes the form 
\[ (\mathbb I + h\mathbb K)\Psi_{ab} = - P_{ab},\]
where $\mathbb I$ denotes the identity operator.  We note that $\mathbb K$ is a nonnegative self--adjoint operator; indeed, the matrix elements in the standard basis are given by 
\[ \mathbb K_{q,p} = \frac{1}{1-hc_W} (p\cdot q) \hat N_0 (p-q).\] 
We may therefore write 
\begin{equation}\label{DKRII} \Psi_{ab} = - (\mathbb I + h\mathbb K)^{-1} P_{ab}.\end{equation}

Let $\e > 0$ which is envisioned to be small as will be specified later.  We claim that there is a subset of $\mathcal S$ which is of nonzero measure such that $|\Psi_{ab} - \Psi_{ab}^*| < \e$ and $|P_{ab} - P_{ab}^*| < \e$ for some values $\Psi_{ab}^*$ and $P_{ab}^*$.  Indeed, all that is required is the observation that e.g., $\mathcal S = \mathcal S \cap \cup_{k}\mathcal S_{ab, k}$ where $\mathcal S_{ab, k} = \{x: \frac{2}{3}(k - \frac{1}{2})\e <  |\Psi_{ab}| < \frac{2}{3} (k + 1)\e\}$; we obtain a similar decomposition for $P_{ab}$. So (up to a set of measure zero) $\mathcal S =  \mathcal S \cap (\cup_k \mathcal S_{ab, k}) \cap (\cup_\ell \mathcal S_{P, \ell})$.  Since all unions are countable, there must exist $k$ and $\ell$ such that $\mathcal S_{ab, k}  \cap \mathcal S_{P, \ell}$ has nonzero measure; let us denote this set by $\mathcal S_\alpha$ and let $\chi_\alpha$ denote the indicator function of this set.
We will now integrate Eq.\eqref{DKRII} on $\mathcal S_\alpha$:
\begin{equation}\label{PEENN} \int_{\mathbb T_L^d} \chi_\alpha \Psi_{ab} ~dx= - \int_{\mathbb T_L^d} \chi_\alpha (\mathbb I + h\mathbb K)^{-1} P_{ab}~dx.\end{equation}

The left hand side of Eq.~\eqref{PEENN} is within $\e$ of $|\mathcal S_\alpha| \Psi_{ab}^*$.  Next we claim that by the positivity and self--adjointness of the operator $\mathbb K$, we may write the operator identity
\[(\mathbb I +h\mathbb K)^{-1} = \mathbb I - h\mathbb K(\mathbb I+ h\mathbb K)^{-1}.\]
The right hand side of Eq.~\eqref{PEENN} can therefore be written as 
\[ -\int_{\mathbb T_L^d}\chi_\alpha  P_{ab}~dx
+\int_{\mathbb T_L^d}\chi_\alpha 
[h\mathbb K(\mathbb I + h\mathbb K)^{-1} 
P_{ab}] ~dx.\]
In the above we observe that both $(1 + h\mathbb K)^{-1}$ and $h\mathbb K (\mathbb I + h\mathbb K)^{-1}$ are bounded operators e.g., in $L^2$ and further that $h\mathbb K(\mathbb I + h \mathbb K)^{-1}$ has operator norm less than one. 

The first term in the above display is within $\e$ of $-|\mathcal S_\alpha|  P_{ab}^*$.  As for the second term, since the relevant operator is self adjoint, 
\[  \begin{split}&\hspace{-1cm}
\int_{\mathbb T_L^d}\chi_{\alpha}
[h\mathbb K(1 + h\mathbb K)^{-1} 
P_{ab}] ~dx\\
&~~~~\leq
\left [
\int_{\mathbb T_L^d} (h\mathbb K(1 + h\mathbb K)^{-1}\chi_\alpha)^2~dx
\right]^{\frac{1}{2}}~~
\left [
\int_{\mathcal S_\alpha} P_{ab}^2~dx
\right]^{\frac{1}{2}} \leq |\mathcal S_\alpha| (P^*_{ab} + \e),
\end{split}
\]
where we have used that the operator norm of $h \mathbb K(\mathbb I + h \mathbb K)^{-1}$ is less than one.  So the terms on the right hand side of Eq.~\eqref{PEENN} add up to no more than $2\e |\mathcal S_\alpha|$.  \b{Now if $\e \ll \Psi^*_{ab}$ (and hence much less than $m_{ab}$ or $M_{ab}$ depending on which case we are in) we would conclude the result (by contradiction of Eq.~\eqref{PEENN}) via the estimates we have just derived.  }
\end{proof}

\subsection{Overview of the Iteration Scheme}
We now provide the overview of how our JKO--type scheme is to be continued.  Starting with some $N_0$, we define $N_1 = \mbox{argmin}\{\mathbb J_{A}(N_0, \cdot)\}$, $N_2 = \mbox{argmin}\{\mathbb J_{A}(N_1, \cdot)\}$, etc.  
\b{However, the abstract methods used so far only yield $N_1 \in \mathcal H_{N_0}^{-1} \cap L^1$ and $\log N_1 \in \mathcal H_{N_0}^1$ whereas to show convergence of the overall scheme and to prove the main theorem we require additional regularity, specifically uniform upper and lower bounds and $\mathcal D_2$ regularity.}  The improved regularity will follow from suitably strong assumptions on $N_0$ 
which will imply that
$N_1$ (and the successive $N_k$'s) in fact coincides with a \emph{classical} solution of Eq.\eqref{CVDA}, with well controlled norms.  The detailed derivation of suitable estimates are the subject of Appendix A; let us summarize the setting of this appendix here:

(a)  The variables used in the appendix are logarithmic:
$$
\Psi  =  \log N.
$$

(b)  We employ ``Fourier norms'':  $f\in \mathcal D_{\ell}$ means that the Fourier coefficients 
of the $\ell^{\text{th}}$ derivatives of $f$ are (absolutely) summable \b{(see Eq.~\eqref{DKBK})}.  These norms are discussed in a bit more detail in Section \ref{PLWF}.  

(c)  We assume that the initial $\Psi_{0}$ is in $\mathcal D_{2}$ and we also adopt the additional regularity assumptions on the interaction potential, namely,
$$
\|W\|_{\mathcal D_2} < \infty \mbox{ and } v_{4}  := \sup_{k}k^{4}|\hat{W}(k)|  < \infty.
$$
\b{(Often, one of these assumptions on $W$ may be redundant: E.g., in $d=1$, $v_4 < \infty$ automatically implies $W \in \mathcal D_2$ whereas in the sufficiently high dimensions one may expect the reverse.)}

We now summarize the logical steps entailed in the program:

\textbf{Step 1.} \b{We assume $N_0 \in \mathcal B$ and $\Psi_0 \in \mathcal D_2$.}

\textbf{Step 2.} We find $N_1 = \inf \left\{\mathbb J_A(N_0, N): N \in \mathcal H_{N_0}^{-1}\right\}$ (see Proposition \ref{UMM}).

\textbf{Step 3.} By a variational argument, we conclude that $N_0$ and $N_1$ provides a one step time discretization of Eq.\eqref{GCD} and in fact $N_1$ is positive almost everywhere (see Proposition \ref{MIN}) and $\Phi_{N_1} \in \mathcal H_{N_0}$ (see Proposition \ref{DSIC}). 

\textbf{Step 4.} Since $N_1$ satisfies the stationarity condition Eq.\eqref{CVDA} and $N_0, N_1$ satisfy the requisite conditions of Lemma \ref{UUNQ}, $N_1$ is uniquely specified. 

\textbf{Step 5.} Lemma \ref{UUNQ} also implies that $N_1$ coincides with the classical solution obtained in Appendix A: I.e., $\Psi_1 \in \mathcal D_2$  (see Corollary \ref{CEB}) and so $N_1 \in \mathcal B$.  (It is noted that since $\|\Psi_1\|_{\mathcal D_2}$ is an upper bound on $\|\nabla^2 \Psi_1\|_\infty$, the $\mathcal D_2$--norm is stronger than the $\mathscr C^2$--norm.)
We may now repeat the previous steps to obtain $N_2, N_3$, etc.  For any fixed $k$, this allows for the production of $N_1, \dots, N_k$, provided that $h$ is sufficiently small. 

\textbf{Step 6.} After $k$ iterations, the macroscopic time achieved is only $kh$ -- thus vanishing with $h$.  However,  we achieve a guaranteed nonzero macroscopic time, i.e., for some fixed $T > 0$ and all $h$ sufficiently small, the process can be carried out for at least the order of $h^{-1} T$ iterations (see Proposition \ref{PYBV}).

\textbf{Step 7.} Via a comparison with the continuum solution (see Proposition \ref{TEE}) it is shown that the macroscopic time can be extended indefinitely; here $h$ has to be suitably small depending on the prescribed macroscopic time of simulation.

\subsection{Convergence}
\label{SecConvergence}

Here we will show that the discretization scheme based on Eq.~\eqref{CVDA} indeed converges to a solution to Eq.\eqref{GCD}.  We reiterate:  Starting with some $N_0$, we define $N_1$, $N_2$, \dots as far as can be done.  
On occasion, we will denote $N_k$, the $k^{\text{th}}$ iterate by $N_{t}^{[h]}$ for time step $h$ when $k$ satisfies $kh \leq t < (k+1)h$; it is in this context that we take the $h\to 0$ limit.}

Assuming that $N_{t}^{[h]}$ exists for nonzero $t$ uniformly in $h$, the extraction of a \b{vague limit} is relatively easy: Indeed, since each step of the iteration only lowers the free energy we have that $N \log N$ is integrable and hence so is $N$ and so a (subsequential) vague limit certainly exists.  Further, limited results pertaining to continuity in time -- H\"older--1/2 -- can also be deduced from the structure implicit in the JKO type scheme, along the lines of what was done in \cite{JKO}.  However, these ideas do not suffice for a demonstration that the limiting object actually satisfies Eq.\eqref{GCD}. 

In order (to acquire enough control) to show that the limiting $N_t$ satisfies the requisite equation, we have need for rather strong estimates, which we provide in Appendix A using Fourier methods.  The analysis in Appendix A is performed essentially in the context of classical solutions, but, by the uniqueness statement in \b{Lemma \ref{UUNQ}}, this solution will coincide with the minimizer of the iterative scheme.  The setting for Appendix A was summarized in the previous subsection. 

For the purposes of the next theorem, \b{consistent with the use in the proof of Proposition \ref{TEE},} let us use the notation $[\cdot]_{t}^{[h]}$ for the various quantities encountered. 

\begin{thm}
\label{FVJJ}
Let $T > 0$ be arbitrary (so that the iterative process is suitably valid for all $h < h_T$ with $h_T$ as in Proposition \ref{TEE}).  
Letting $\Psi_{t}^{[h]} = \log N_t^{[h]}$, we have that $\Psi_t^{[h]}$ converges to a weak solution, $\Psi_t$,
 of Eq.~\eqref{GCD} (written in these logarithmic variables) as $h$ tends to zero, i.e., if $b \in \mathscr C^1(\mathbb T_L^d \times (0, T))$,
\[\b{\int_{\mathbb T_L^d \times (0, 1)} N_t \frac{\partial b}{\partial t}~dxdt = \int_{\mathbb T_L^d \times (0, 1)} N_t (\nabla \Phi_{N_t} \cdot \nabla b) + \Omega_{N_t} \Phi_{N_t} b~dxdt}, \]
where as before $\b{\Phi_{N_t} = \log N_t -\mu + w_{N_t}}$.
 
Moreover, 

(A) This convergence is strong in the $\mathcal D_{1}$--norm and uniform in the $\mathcal D_{0}$--norm.

(B) $N_t = e^{\Psi_t}$ is the unique solution to the continuous time equation as given by Eq.~\eqref{GCD} which is $\mathscr C^\infty$ for positive times and $N_t \rightarrow N_0$ strongly in $\mathcal D_0$ as $t \rightarrow 0$.
\end{thm}

\begin{proof}
Item (A) will be established in Appendix A after the proof of Proposition \ref{TEE} and item (B) will be addressed briefly at the end of the proof.  Let us now address the main convergence result.  We will first establish that if $N_t$ is a weak limit of $N_t^{[h]}$ as $h$ tends to zero, then $N_t$ is a weak solution to Eq.\eqref{GCD}. (It is clear, e.g., from the discussion before the statement of this theorem that one can always extract a weak limit.) 

Now consider some $b \in \mathscr C^1(\mathbb T_L^d \times (0, T))$
which is integrated against both sides of the iteration equation as given in Eq.~\eqref{XTCL} and then summed over the order of $Th^{-1}$ iterations (and using $\Phi_t^{[h]} = \Psi_t^{[h]} - \mu + w_{N_t^{[h]}}$):
\begin{equation}
\begin{split}
&\hspace{1.5cm}\sum_k\int_{\mathbb T_{L}^{d}}\left(\frac{N_{k+1} - N_k}{h}\right)\b{b_k}~dx=\\
 &-\sum_k \int_{\mathbb T_L^d} N_k\left(\nabla (\Psi_{k+1} + w_{N_{k+1}}) \cdot \nabla \b{b_k}\right) + \Omega_{N_k}\left(\Psi_{k+1} - \mu + w_{N_{k+1}}\right)\b{b_k}~dx,
\end{split}
\end{equation}
where $b_k$ is a suitable time average over the interval $hk \leq t < h(k+1)$. 
The left hand side, after summation by parts, weakly converges to the integral of $-N_t^{[h]} (\partial b/\partial t)$.  As for the right hand side,  for convenience we will now go over to the notation $N_t^{[h]}$ instead of $N_k$ and $N_{t+h}^{[h]}$ instead of $N_{k+1}$ etc., and then the sum over $k$ can be replaced by an integral over $[0, T]$.  First we observe that if it were the case that all the indices were in agreement and e.g., equal to $k+1$, then the right hand side can be realized entirely as a weak equation for $N_t^{[h]}$ (with most of the burden of differentiation passed on to $b$) which would converge weakly to the relevant limit.  What we must estimate then is the differences caused by the discrepancy in indices.  For example, in the term containing $\Psi$, forcing the indices to match yields the residual term 
$$
-\int_{0}^{T}\int_{\mathbb T_{L}^{d}} (N_{t + h}^{[h]}  - N_{t}^{[h]}) (\nabla\Psi_{t+h}^{[h]} \cdot \nabla b)
~dxdt.
$$
By the results obtained in Appendix A, \b{specifically Corollary \ref{CEB2}, iii),} we have that
$|\nabla \Psi_{t}^{[h]}|$ is uniformly bounded  (e.g., in $L^{\infty}$) in both $h$ and $t$
while $N_{t + h}^{[h]} - N_{t}^{[h]}  =  \text{e}^{\Psi_{t + h}^{[h]}} - \text{e}^{\Psi_{t }^{[h]}}$
is bounded above by $h$ times a function which, again, has a uniform $L^{\infty}$ bound.  
Hence, this error term disappears from consideration in the $h\to 0$ limit.

Identical considerations apply to the term $N_{t}^{[h]} (\nabla w_{N_{t + h}}^{[h]} \cdot \nabla b)$.  However, here the situation is even less demanding since 
$\nabla w_{N_{t+h}}^{[h]}$ does not even involve gradients of $\Psi$.  As for the inhomogeneous term, 
it is slightly easier to do the reindexing on the $\Phi$--terms.  We write
$$
\Phi_{t + h}^{[h]} \Omega_{N_t }^{[h]}  =
\Phi_{t}^{[h]} \Omega_{N_t }^{[h]}
+
(\Phi_{t + h}^{[h]} - \Phi_{t }^{[h]})\Omega_{N_t }^{[h]}.
$$
The leading term on the right of the above display 
is of the correct form.  Examining the definition of 
$\Omega_{N_t}^{[h]}$, it is clear that if $N_t$ is bounded \b{in $L^\infty$ (which follows from Corollary \ref{CEB2}, iii))} 
then so is $\Omega_{N_t}^{[h]}$.  Since $\Phi_{N}  =  \Psi_N -\mu + w_{N}$, \b{from Corollary \ref{CEB2} ii) and iii),} we have that $|\Phi_{t + h}^{[h]} - \Phi_{t }^{[h]}|$ is bounded by order $h$ and this term also disappears 
in the $h \rightarrow 0$ limit.  

Finally, by standard regularity results about (uniformly) parabolic equations, we have that $N_t$ is smooth (\cite{esk}) for positive times and the convergence to initial data can be easily gleamed from item (A) and Proposition \ref{XTY}. 
\end{proof}

\section{Proof of the Main Theorem}
In this section, we provide a proof of the principal \textit{result} of this work.  Namely:  If the initial $N_{0}$ is in the vicinity of the uniform state, and the latter is ``sufficiently stable'' then the subsequent dynamics is characterized by exponential convergence to this state.  

\subsection{Convexity Estimates} 

In this subsection, we aggregate all the results concerning \textit{convexity} of the function
$\mathscr G_{\mu}(\cdot)$ which will be used in the proof of the main theorem.  First, it is seen that 
if $W$ is of positive type then $\mathscr G_{\mu}(\cdot)$ is always a \b{convex functional of $N$} for all $\mu$.
But, it is also known that such circumstances foreclose any possibility of a phase transition.  
However, even here, the rate of convergence to equilibrium is still of interest.  
More pertinently in the general cases \b{under study}, it is not unreasonable to assume that if $\text{e}^\mu$ is sufficiently small \textit{and} overall the fluid is reasonably homogeneous with a density not too far from the uniform state that some local convexity properties should ensue.

First,
recall the definition of (the density of) the uniform state 
$\text{\textsc{m}}_{0}$ which is the solution to 
$\text{\textsc{m}}_{0}  =  \text{e}^{\mu-w\textsc{m}_0}$ with $w$ being the integral of $W$, as described following Eq.~\eqref{ASKY}.  In what follows, instead of using 
$\mu$ -- which is conceivably large and negative -- as our parameter we will use the quantity
$\text{\textsc{m}}_{0} = \text{\textsc{m}}_{0}(\mu)$ as our (small) parameter.

\begin{proposition}
\label{OGTR}
\b{Let $N_t \in \mathscr C^2$ be a classical solution of Eq.~\eqref{GCD}.}
Let $\kappa$ be any number such that
$0 < \kappa < \frac{1}{2}$  and suppose that at time $t_{0} \geq 0$
the density $N_{t_0}$ satisfies the pointwise bounds
$$
\kappa \text{\textsc{m}}_{0} < N_{t_0}(x) 
< \frac{1}{\kappa} \text{\textsc{m}}_{0}.
$$
Then, if $\text{\textsc{m}}_{0}$ is sufficiently small, this condition persists for all time 
$t > t_{0}$.  
\end{proposition}

\begin{proof}
Examining Eq.~\eqref{GCD} and recalling that we can reason classically, let us assume that 
$x_{\sharp}$ is a point of maximum or minimum.  Then at $x = x_{\sharp}$, we have
$$
\frac{\partial N_{t}(x_{\sharp})}{\partial t} \geq \b{N_t}\nabla^{2}w_{N_{t}} -
\left[N_{t}\text{e}^{-\frac{1}{2}(\mu - w_{N_{t}})} - \text{e}^{+\frac{1}{2}(\mu - w_{N_{t}})}\right]
$$
for a minimum and with the opposite inequality if $x_{\sharp}$ is a maximum.  

Now we claim that for 
$\text{\textsc{m}}_{0}$ sufficiently small, \b{it is the case that for all $x$, provided $\kappa \textsc{m}_0 \leq N_t(x) < \frac{1}{\kappa} \textsc{m}_0$, we have}
$$
-\kappa\text{\textsc{m}}_{0}
\text{e}^{-\frac{1}{2}(\mu - w_{N_{t}})} +
 \text{e}^{+\frac{1}{2}(\mu - w_{N_{t}})}
 \geq 
 \kappa \text{\textsc{m}}_{0}^{\frac{1}{2}}
$$
and 
$$
-\frac{1}{\kappa}\text{\textsc{m}}_{0}
\text{e}^{-\frac{1}{2}(\mu - w_{N_{t}})} +
 \text{e}^{+\frac{1}{2}(\mu - w_{N_{t}})}
 \leq 
 -\text{\textsc{m}}_{0}^{\frac{1}{2}}.
 $$
 Indeed, \b{since $\textsc{m}_0^{-\frac{1}{2}} = \text{e}^{-\frac{1}{2}(\mu - w\textsc{m}_0)}$,} the second display amounts to the inequality
 $\text{e}^{\frac{1}{2}(w_{N_{t}} - w\text{\textsc{m}}_{0})}  -
 \kappa\text{e}^{-\frac{1}{2}(w_{N_{t}} - w\text{\textsc{m}}_{0})}
 \geq \kappa$
 and we can use
 $w_{N_{t}}  \geq  - \frac{1}{\kappa}w_0\text{\textsc{m}}_{0}$ \b{(where $w_0$ is the integral of $|W|$)} while the first display reduces to 
 $\text{e}^{\frac{1}{2}(w\text{\textsc{m}}_{0} - w_{N_{t}})}  -
 \kappa\text{e}^{-\frac{1}{2}(w\text{\textsc{m}}_{0} - w_{N_{t}})}
 \geq \kappa$ and we can also use $w_{N_{t}}  \leq  \frac{1}{\kappa}w_0\text{\textsc{m}}_{0}$.  The claimed result is now manifest 
 for $\text{\textsc{m}}_{0}$ sufficiently small.

Let us suppose then that at some time $t_{\sharp}$, for the first time, the density achieves the 
 value $\frac{1}{\kappa}\text{\textsc{m}}_{0}$ and this occurs at the point $x = x_{\sharp}$ -- which is its maximum.  
 Then we would have \b{(with $w_2$ being the integral of $|\nabla^2 W|$)}
 $$
 \frac{\partial N_{t_{\sharp}}(x_{\sharp})}{\partial t} 
 \leq - \text{\textsc{m}}_{0}^{\frac{1}{2}} + \frac{1}{\kappa}\text{\textsc{m}}_{0}w_2,
 $$
 which is strictly negative for $\text{\textsc{m}}_{0}$ sufficiently small.  
 While this immediately implies that at the point $x_{\sharp}$, the density can grow no bigger, 
 it actually implies, by continuity, that such happenstance could never occur in the first place:  At 
$t = t_{\sharp}^{-}$ before the density at $x = x_{\sharp}$ achieved 
$\frac{1}{\kappa}\text{\textsc{m}}_{0}$, the derivative was already negative.  

Similar considerations apply -- for  $\text{\textsc{m}}_{0}$ sufficiently small -- if we investigate 
the first time that the density has fallen as low as $\kappa \text{\textsc{m}}_{0}$.
\end{proof}

Consider, then, the convex set $\mathcal B_{\kappa} \subseteq \mathcal B$ consisting of those densities which satisfy the bounds featured in Proposition \ref{OGTR}.  (It is noted that the parameters of the upper and lower bounds need not be related.  However, the condition is natural for the variable $\Psi = \log N$.)  Our next claim is that if $\kappa  \text{\textsc{m}}_{0}$ is sufficiently small then the functional 
$\mathscr G_{\mu}(\cdot)$ restricted to $\mathcal B_{\kappa}$ is convex:
\begin{proposition}
\label{BBGG}
For $\text{\textsc{m}}_{0}/\kappa < \vartheta^{\sharp}$
where
$$
\frac{1}{\vartheta ^{\sharp}}  =  \max_{k}\left\{|\hat{W}(k)|\mid \hat{W}(k) < 0\right\},
$$
the functional $\mathscr G_{\mu}(\cdot)$ restricted to $\mathcal B_{\kappa}$ is convex.
And, therefore, $N \equiv \text{\textsc{m}}_{0}$ is the unique minimizer in 
$\mathcal B_{\kappa}$.  
In the above we may take $\vartheta^{\sharp} = \infty$ if the interaction is of positive type.
\end{proposition}
\begin{proof}
Let $N_{A}$, $N_{B}$ be temporary notation for densities in 
$\mathcal B_{\kappa}$ and similarly, let us define $N_{s} : = (1-s)N_{A} + sN_{B}$ and
$R := N_{B} - N_{A}$.  A direct calculation shows
$$
\frac{d^{2}\mathscr G_{\mu}(N_{s})}{ds^{2}}
=
\int_{{\mathbb T_{L}^{d}}}\frac{R^{2}}{N_{s}}dx
+
\int_{{\mathbb T_{L}^{d}}\times {\mathbb T_{L}^{d}}}
\hspace{-.25cm}
W(x-y)R(x)R(y)dxdy.
$$
The first term on the right is larger than $(\kappa /\text{\textsc{m}}_{0}) \cdot \|R\|^{2}_{L^{2}}$ and as for the second, we have
$$
\int_{{\mathbb T_{L}^{d}}\times {\mathbb T_{L}^{d}}}
\hspace{-.25cm}
W(x-y)R(x)R(y)dxdy  =  \frac{1}{L^{d}}\sum_{k}\hat{W}(k)|\hat{R}(k)|^{2}
\geq -\frac{1}{\vartheta^{\sharp}}\cdot \|R\|^{2}_{L^{2}}
$$
and the primary statement is proved.  The secondary statement is immediately clear since
$N \equiv \text{\textsc{m}}_{0}$ is always a stationary point and the convexity that was just proved is actually strict.
\end{proof}

\begin{remark}
We remark that notwithstanding factors of order unity -- e.g., $\kappa$ -- the estimates here (and presumably those in Proposition \ref{OGTR}) are reasonably sharp.  Indeed, 
$\text{\textsc{m}}_{0} = \vartheta^{\sharp}$ is the point where the stationary solution
$N_{t} \equiv \text{\textsc{m}}_{0}$ is \textit{linearly} unstable and, translating the results of
\cite{CP} to the current context, when $\text{\textsc{m}}_{0} = \vartheta_{\t} < \vartheta^{\sharp}$, 
already there are non--trivial minimizers for $\mathscr G_{\mu}(\cdot)$.
\end{remark}

\subsection{Proof of the Main Theorems}

\emph{Proof of Theorem \ref{MDKR}.}
Let $t > 0$ and $T > t$. 
Let \b{$\kappa^\prime > \kappa$}.  By Theorem \ref{FVJJ} we may consider $h$'s sufficiently small so that throughout
$(0,T)$, the actual
continuum solution $N_{t}$ and the discretization $N_{t}^{[h]}$ differ only slightly in
e.g., the $\mathcal D_{1}$--norm so that
for all $t\in [0,T]$, we have $N_{t}^{[h]} \in \mathcal B_{\kappa}$.  It follows by Proposition \ref{BBGG} that $\mathscr G_\mu(\cdot)$ is convex for these $N_t^{[h]}$'s. 

In the following, we will examine one iteration of the process at fixed $h$.  To avoid clutter, 
we will again employ the (inconsistent) notation that 
$N_{0}$ is the initial density and $N_{1}$ is the final density for this step.  
Let us define, for $\lambda > 0$
$$
M_{\lambda}^{(0)}  :=  (1 - h\lambda)N_{0}  +  h\lambda \text{\textsc{m}}_{0}
$$
so that
$M_{\lambda}^{(0)} - N_{0} =  h\lambda(\text{\textsc{m}}_{0} - N_{0})$.
Let us also define $Q$ to be the potential which pushes $N_{0}$ all the way to 
$\text{\textsc{m}}_{0}$ in unit time under the approximate dynamics:
$$
\text{\textsc{m}}_{0} - N_{0} =:
\nabla\cdot (N_{0}\nabla Q) - \Omega_{N_{0}}Q.
$$
Further, the approximate distance (all the way) to $\text{\textsc{m}}_{0}$
is given by
$$
\mathbb D_{A}^{2}(N_{0}, \text{\textsc{m}}_{0})  =
\int_{\mathbb T_{L}^{d}} N_{0}|\nabla Q|^{2} + \Omega_{N_{0}}Q^{2}~dx.
$$
It is underscored, informally, that $\mathbb D^{2}_{A}(N_{0}, \text{\textsc{m}}_{0})$ -- and $Q$ --
are of order unity relative to $h$ with $h \ll 1$.  We have (since the relevant equations are linear)
$$
\mathbb D^{2}_{A}(N_{0}, M_{\lambda}^{(0)}) 
=
h^{2}\lambda^{2}\cdot \mathbb D_{A}^{2}(N_{0}, \text{\textsc{m}}_{0}).
$$

We now adjust $\lambda$ so that this distance is exactly the distance which is traveled under the auspices of the JKO type process:
$$
h^{2}\lambda^{2} \cdot \mathbb D_{A}^{2}(N_{0}, \text{\textsc{m}}_{0})
=
\mathbb D_{A}^{2}(N_{0}, N_{1}).
$$
Now since we must have
$\mathbb J_{A}(N_{0}, N_{1}) \leq \mathbb J_{A}(N_{0},M_{\lambda}^{(0)})$, it follows that
$\mathscr G_{\mu}(N_{1}) \leq \mathscr G_{\mu}(M_{\lambda}^{(0)})$.  
Using convexity of $\mathscr G_{\mu}(\cdot)$, we have
$$
\mathscr G_{\mu}(N_{1}) - \mathscr G_{\mu}(\text{\textsc{m}}_{0})
\leq
(1 - h\lambda) \cdot \left[\mathscr G_{\mu}(N_{0}) - \mathscr G_{\mu}(\text{\textsc{m}}_{0})\right].
$$

Thus, if we can get $\lambda$ uniformly bounded below for an indefinite number of iterations of the process, then in the standard (discretization) notation, the above becomes\begin{equation}\label{DISP}
\mathscr G_{\mu}(N^{[h]}_{(k+1)h}) - \mathscr G_{\mu}(\text{\textsc{m}}_{0})
\leq
(1 - h\lambda)\cdot \left[\mathscr G_{\mu}(N^{[h]}_{kh}) - \mathscr G_{\mu}(\text{\textsc{m}}_{0})\right]
\end{equation}
and so in the $h\to 0$ limit,
$\mathscr G_{\mu}(N_{t}) - \mathscr G_{\mu}(\text{\textsc{m}}_{0})
\leq
\text{e}^{-\lambda t}\cdot \left[\mathscr G_{\mu}(N_{0}) - \mathscr G_{\mu}(\text{\textsc{m}}_{0})\right]$. 
We turn our investigations to $\lambda$.  
Let us start with some preliminary estimates on $\mathbb D_{A}^{2}(N_{0}, \text{\textsc{m}}_{0})$.

\noindent \b{\textbf{Claim 1.} We have 
\begin{equation}
\label{PDFD} \mathbb D_A^2(N_0, \textsc{m}_0) \leq g^2 \cdot \|N_0 - \textsc{m}_0\|_{L^2}^2, \end{equation}
where $g^2 := (\kappa \textsc{m}_0)^{-1/2}$. }

\noindent \emph{Proof of Claim.} We start with the identities
\begin{equation}
\label{VYCS}
-\int_{\mathbb T_{L}^{d}}(\text{\textsc{m}}_{0} - N_{0})Q~dx
=
\mathbb D_{A}^{2}(N_{0}, \text{\textsc{m}}_{0})
=
\int_{\mathbb T_{L}^{d}} N_{0}|\nabla Q|^{2} + \Omega_{N_{0}}Q^{2}~dx.
\end{equation}
So, using inequalities on both ends:
\begin{equation}
\label{UYCS}
\|N_{0} -\text{\textsc{m}}_{0}\|_{L^{2}} \cdot \|Q\|_{L^{2}} \geq
\int_{\mathbb T_{L}^{d}}\Omega_{N_{0}}Q^{2}~dx.
\end{equation}
It is now claimed that, pointwise,
\begin{equation}
\label{CQGR}
\Omega_{N}  \geq  N^{\frac{1}{2}}.
\end{equation}
Indeed, \b{this follows from the known inequality $(a-b)/\log(a/b) \geq \sqrt{ab}$,} but in any case (for completeness) we write
$$
\Omega_{N}  =  \frac{N^{\frac{1}{2}}}{\Phi_{N}}
\left(N^{\frac{1}{2}}\text{e}^{-\frac{1}{2}(\mu - w_{N})}
-  \frac{1}{N^{\frac{1}{2}}}\text{e}^{\frac{1}{2}(\mu - w_{N})}\right)
= N^{\frac{1}{2}}
\frac{\sinh \frac{1}{2} \Phi_N}{\frac{1}{2}\Phi_{N}}
\geq 
N^{\frac{1}{2}}.
$$

Thus the bound in 
Eq.~(\ref{UYCS}) may be replaced by
$$
\|N_{0} -\text{\textsc{m}}_{0}\|_{L^{2}} \cdot \|Q\|_{L^{2}} \geq
(\kappa\text{\textsc{m}}_{0})^{\frac{1}{2}} \cdot \|Q\|^{2}_{L^{2}},
$$
i.e., 
$$
\|Q\|_{L^{2}} \leq 
\frac{1}{(\kappa\text{\textsc{m}}_{0})^{\frac{1}{2}}} \cdot \|N_{0} -\text{\textsc{m}}_{0}\|_{L^{2}}.
$$
Putting this back into Eq.~\eqref{VYCS} we acquire
\[
\mathbb D_{A}^{2}(N_{0}, \text{\textsc{m}}_{0})
\leq
\frac{1}{(\kappa\text{\textsc{m}}_{0})^{\frac{1}{2}}} \cdot \|N_{0} -\text{\textsc{m}}_{0}\|^{2}_{L^{2}} = g^{2} \cdot 
\|N_{0} -\text{\textsc{m}}_{0}\|^{2}_{L^{2}}
\]
as stated.   \hspace{1 cm}  $\blacksquare$

From Claim 1 we have 
\begin{equation}
\label{UYTB}
h^2 \lambda^2 \cdot g^2 \|N_{0} -\text{\textsc{m}}_{0}\|^{2}_{L^{2}}\geq
 h^2 \lambda^2 \cdot \mathbb D_A^2 (N_0, \textsc{m}_0)
=\mathbb D^{2}_{A}(N_{0}, M_{\lambda}^{(0)})
=
\mathbb D_{A}^{2}(N_{0}, N_{1}),
\end{equation}
so our goal will be achieved if we can show that
$\mathbb D_{A}^{2}(N_{0}, N_{1})$ is of the same order as 
$h^{2}\|N_{0} - \text{\textsc{m}}_{0}\|^{2}_{L^{2}}$.  
To this end, we will now consider
$$
M_{\theta}^{(1)}  :=
(1 - h\theta)N_{1} + h\theta \text{\textsc{m}}_{0}.
$$
The strategy here is to show that if $\mathbb D_A^2(N_0, N_1)$ were not of the correct order of magnitude (according to the above stated goal) then $M_{\theta}^{(1)}$
would be a better minimizer for 
$\mathbb J_{A}(N_{0}, \cdot)$. 
In what follows, let us use the version of
$\mathbb J_{A}$ in which the current value of the free energy is subtracted off:
$$
\mathbb J_{A}(N_{0}, M_{\theta}^{(1)})  =  
\frac{1}{2}\mathbb D_{A}^2(N_{0}, M_{\theta}^{(1)}) 
+ h\left[\mathscr G_{\mu}(M_{\theta}^{(1)}) - \mathscr G_{\mu}(N_{0})\right].
$$  

We start with an upper bound on $\mathbb D^{2}_{A}(N_{0}, M_{\theta}^{(1)})$.  To this end, it is noted that since 
\[ \b{M_\theta^{(1)} - N_0 = (1-h\theta) (N_1 - N_0) + h\theta (\textsc{m}_0 - N_0)},\]
the driving field which achieves $M_{\theta}^{(1)}$ is given by
\b{$(1 - h\theta)\cdot h\Phi_{N_{1}} + h\theta Q$}.  Therefore 
$\mathbb D^{2}_{A}(N_{0}, M_{\theta}^{(1)})  =
(1 - h\theta)^{2} \cdot \mathbb D^{2}_{A}(N_{0}, N_{1})  + h^{2}\theta^{2} \cdot \mathbb D^{2}_{A}(N_{0}, \text{\textsc{m}}_{0}) + 2h^{2}\theta(1 - h\theta) \cdot
\langle \hspace{-.1cm} \langle  \Phi_{N_{1}}  , Q   \rangle  \hspace{-.1cm}  \rangle_{N_{0}}$.
We will bound the last term by
$2h\theta(1-h\theta) \cdot \mathbb D_{A}(N_{0}, N_{1})\mathbb D_{A}(N_{0}, \text{\textsc{m}}_{0})$: \b{This follows from the Cauchy--Schwarz inequality since e.g., $\langle \hspace{-.1cm} \langle  Q  , Q   \rangle  \hspace{-.1cm}  \rangle_{N_{0}} = \mathbb D_A^2 (N_0, \textsc{m}_0)$}.  (We note that one factor of $h$ has been absorbed into the term $\mathbb D_{A}(N_{0}, N_{1})$.)  I.e., we have the square of the triangle inequality: 
\[\begin{split} \mathbb D_A^2(N_0, M_\theta^{(1)}) \leq (1 - h\theta)^{2} \cdot \mathbb D^{2}_{A}(N_{0}, N_{1})  &+ h^{2}\theta^{2} \cdot \mathbb D^{2}_{A}(N_{0}, \text{\textsc{m}}_{0})\\
&~~~~~~+ 2h\theta(1-h\theta) \cdot \mathbb D_A(N_0, N_1) \mathbb D(N_0, \textsc{m}_0)
\end{split}\]
Meanwhile, by the convexity from Proposition \ref{BBGG},
$$\mathscr G_{\mu}(M_{\theta}^{(1)}) \leq (1-h\theta) \cdot \mathscr G_{\mu}(N_{1}) + h\theta \cdot \mathscr G_{\mu}(\text{\textsc{m}}_{0}).$$  

Putting the previous two displays together and subtracting off $\mathbb J_A(N_0, N_1)$, we have
\[\begin{split}
\mathbb J_{A}(N_{0}, M_{\theta}^{(1)}) &- \mathbb J_{A}(N_{0}, N_1) = \frac{1}{2} \left[\mathbb D_A^2(N_0, M_\theta^{(1)}) - \mathbb D_A^2(N_0, N_1) \right] + h \left[ \mathscr G_\mu(M_\theta^{(1)}) - \mathscr G_\mu(N_1) \right]\\
&\leq h\theta \cdot \mathbb D_{A}(N_{0}, N_{1}) \mathbb D_{A}(N_{0}, \text{\textsc{m}}_{0}) + 
\frac{1}{2} h^{2}\theta^{2} \cdot \left[\mathbb D_{A}(N_{0},\text{\textsc{m}}_0) - \mathbb D_{A}(N_{0}, N_{1})\right]^{2}
\notag
\\
& -h\theta \cdot \left[\mathbb D_{A}^{2}(N_{0}, N_{1}) + h(\mathscr G_{\mu}(N_{1}) - \mathscr G_{\mu}(N_{0}))\right] + h^{2}\theta \cdot \left[\mathscr G_{\mu}(\text{\textsc{m}}_{0}) - \mathscr G_{\mu}(N_{0})\right].  
\end{split}\]
Since $N_1$ is a minimizer, the right hand side is nonnegative.  In particular this is so when we divide by $h\theta$ and take the $\theta \to 0$ limit.  Thus
\begin{equation}\label{WIRKK}
\begin{split}
h[\mathscr G_{\mu}(N_{0}) - \mathscr G_{\mu}(\text{\textsc{m}}_{0})]
&\leq
\mathbb D_{A}(N_{0}, N_{1}) \mathbb D_{A}(N_{0}, \text{\textsc{m}}_{0})\\
 &\hspace{2cm}- \left[\mathbb D_{A}^{2}(N_{0}, N_{1}) + h(\mathscr G_{\mu}(N_{1}) - \mathscr G_{\mu}(N_{0}))\right].
\end{split}
\end{equation}

Next we have the following estimate relating $\mathscr G_\mu(N_0) - \mathscr G_\mu(\textsc{m}_0)$ and $\|N_0 - \textsc{m}_0\|_{L^2}^2$:

\noindent \textbf{Claim 2.} 
Under the conditions on $\text{\textsc{m}}_{0}$ and $\kappa$ in the statement of this theorem, there is a 
$\sigma = \sigma(\text{\textsc{m}}_{0}, \kappa) > 0$ such that
$$
\mathscr G_{\mu}(N_{0}) - \mathscr G_{\mu}(\text{\textsc{m}}_{0}) \geq
\sigma\|N_{0} - \text{\textsc{m}}_{0} \|^{2}_{L^{2}}.
$$
\noindent\textit{Proof of Claim.}  
This, it turns out, is a recapitulation of (the convexity) Proposition \ref{BBGG}.  If we write
$N_{0}  =  \text{\textsc{m}}_{0} + (N_{0} - \text{\textsc{m}}_{0})$ we can expand the free energy in powers of
$N_{0} - \text{\textsc{m}}_{0}$.  The first order term vanishes by stationarity while the interaction piece is exact at the quadratic order.  Now, pointwise, 
\[\begin{split}
(\text{\textsc{m}}_{0} + (N_{0} - \text{\textsc{m}}_{0})) \cdot \log(\text{\textsc{m}}_{0} &+ (N_{0} - \text{\textsc{m}}_{0}))\\
&= \text{\textsc{m}}_{0}\log\text{\textsc{m}}_{0} + \text{linear piece } + 
\frac{1}{2}\cdot \frac{(N_{0} - \text{\textsc{m}}_{0})^{2}}{\nu N_{0} + (1-\nu)\text{\textsc{m}}_{0}}
\end{split}\]
where $\nu \in [0,1]$ depends on the value of $N_{0}(x)$.  Thus we may write
$$
\mathscr G_{\mu}(N_{0}) - \mathscr G_{\mu}(\text{\textsc{m}}_{0})  =
\frac{1}{2}\left[\int_{\mathbb T_{L}^{d}}\frac{R_{0}^{2}(x)~dx}{\nu(x)N_{0} + (1-\nu(x))\text{\textsc{m}}_{0}} 
+ \int_{\mathbb T_{L}^{d}\times \mathbb T_{L}^{d}}
\hspace{-.25 cm}
W(x-y)
R_{0}(x)R_{0}(y)~dxdy\right]
$$
with $R_{0}(x)$ being temporary notation for $N_{0}(x) - \text{\textsc{m}}_{0}$.  The conclusion follows with
$$
\sigma  =  \frac{1}{2}
\left(
\frac{\kappa}{\text{\textsc{m}}_{0}} - \frac{1}{\vartheta^{\sharp}}
\right).
$$  
The stated claim has been established.  ~~~~$\blacksquare$
\begin{remm}
\b{For future reference we note that the estimates in Claim 1 and Claim 2 apply to any density $N \in \mathcal B_\kappa$ and not just $N_0$.  We also note that the constant $g^2 = (\kappa \textsc{m}_0)^{-1/2}$ does not depend on the particulars of $N_0$.}  
\end{remm}
Thus, dropping the $\mathbb D_A^2(N_0, N_1)$ term from Eq.~\eqref{WIRKK} and using Eq.~\eqref{PDFD}, we get
\begin{equation}
\label{XASA}
h\sigma\|N_{0} - \text{\textsc{m}}_{0}\|_{L^{2}}^{2}  
+h(\mathscr G_{\mu}(N_{1}) - \mathscr G_{\mu}(N_{0}))
\leq  \mathbb D_{A}(N_{0}, N_{1})\cdot g\|N_{0} - \text{\textsc{m}}_{0}\|_{L^{2}}.
\end{equation}
Now were it not for the small term $h(\mathscr G_\mu(N_1) - \mathscr G_\mu(N_0))$ on the left, we would obtain
a lower bound of $\frac{h\sigma}{g} \cdot \|N_{0} - \text{\textsc{m}}_{0}\|_{L^{2}}$ for $\mathbb D_{A}(N_{0}, N_{1})$
which by Eq.~\eqref{UYTB} would imply
$$
\lambda  \geq  \frac{\sigma}{g^{2}} := \lambda^{\dagger}.
$$
Since the small free energy difference term will appear at each stage of the iteration and there are of order $h^{-1}$ steps altogether, let us write Eq.~\eqref{XASA} in the form that it would appear without the abbreviations:   
\begin{equation}
\label{XBSB}
h\sigma\|N^{[h]}_{kh} - \text{\textsc{m}}_{0}\|_{L^{2}}^{2}  
+h\left(\mathscr G_{\mu}(N^{[h]}_{(k+1)h}) - \mathscr G_{\mu}(N^{[h]}_{kh})\right)
\leq  g\mathbb D_{A}\left(N^{[h]}_{kh}, N^{[h]}_{(k+1)h}\right)\|N^{[h]}_{kh} - \text{\textsc{m}}_{0}\|_{L^{2}}.
\notag
\end{equation}

Let us stipulate that, necessarily, for all times $t^{\prime} < t$, 
$N_{t^{\prime}} \neq \text{\textsc{m}}_{0}$  (indeed, otherwise there would be nothing to prove).
Thus, it is clear that
$$
\epsilon  :=  \inf_{k,h:  hk \leq t}\|N^{[h]}_{kh} - \text{\textsc{m}}_{0}\|_{L^{2}}^{2}
$$
is strictly positive.  We shall only consider $h$'s which satisfy $h < \epsilon^{2}$ and thus the above generalization of Eq.~\eqref{XASA} 
in combination with Eq.~\eqref{UYTB} yields the estimate
$$
h\lambda_{k+1}  \geq  \frac{\mathbb D_{A}\left(N_{hk}^{[h]},N_{h(k+1)}^{[h]}\right)}{g\|N_{hk}^{[h]} - \text{\textsc{m}}_{0}\|_{L^{2}}}
\geq
\left[\frac{h\sigma}{g^{2}} + \frac{h^{\frac{1}{2}}}{g^{2}}
\left(\mathscr G_{\mu}(N_{h(k+1)}^{[h]}) - \mathscr G_{\mu}(N_{hk}^{[h]})\right)\right].
$$
\b{(In the above we are using that $\mathscr G_{\mu}(N_{h(k+1)}^{[h]}) - \mathscr G_{\mu}(N_{hk}^{[h]}) \leq 0$ which is clear since otherwise $N_{hk}^{[h]}$ would've been a better minimizer for $\mathbb J_A(N_{hk}^{[h]}, \cdot)$ than $N_{h(k+1)}^{[h]}$.)}

Recalling the discussion surrounding the display labeled \eqref{DISP} (and iterating) we now have the estimate 
$$
\mathscr G_{\mu}(N^{[h]}_{(k+1)h}) - \mathscr G_{\mu}(\text{\textsc{m}}_{0})
\leq
\left[\mathscr G_{\mu}(N_0) - \mathscr G_{\mu}(\text{\textsc{m}}_{0})\right] \cdot 
\prod_{j = 1}^{k}(1 - h\lambda_{j}).
$$
We bound the product \b{(recall that $\lambda^\dagger = \sigma/g^2$)} as follows:
\begin{align}
\prod_{j=1}^{k}(1-h\lambda_{j})  &=  (1 - h\lambda^{\dagger})^{k}\cdot\prod_{j=1}^{k}  
\left(1 - \frac{h^{\frac{1}{2}}}{\b{g^2}(1 - h\lambda^{\dagger})}\left(\mathscr G_{\mu}(N_{h(k+1)}^{[h]}) - \mathscr G_{\mu}(N_{hk}^{[h]})\right)\right)
\notag
\\
&\leq
(1 - h\lambda^{\dagger})^{k} \cdot 
\text{Exp}\left[-\frac{h^{\frac{1}{2}}}{\b{g^2}(1 - h\lambda^{\dagger})}
\sum_{j = 1}^{k}\left(\mathscr G_{\mu}(N_{h(k+1)}^{[h]}) - \mathscr G_{\mu}(N_{hk}^{[h]})\right)\right].
\end{align}
The sum in the exponent is just the current free energy drop which may be bounded uniformly in $k$ by the total free energy drop, namely
 $\mathscr G_{\mu}(\text{\textsc{m}}_{0}) - \mathscr G_{\mu}(N_{0})$, and 
the pre--factor of $h^{\frac{1}{2}}$ causes this factor in the exponent to vanish in the $h\to 0$ limit.  Thus, as claimed, when we take $h\to 0$
$$
\mathscr G_{\mu}(N_{t}) - \mathscr G_{\mu}(\text{\textsc{m}}_{0})  \leq 
\left[\mathscr G_{\mu}(N_{0}) - \mathscr G_{\mu}(\text{\textsc{m}}_{0})\right]\cdot\text{e}^{-\lambda^{\dagger}t}.
$$
By the result displayed in Claim 2 (applied to $N_t$ instead of $N_0$) a similar estimate holds for 
$\|N_{t} - \text{\textsc{m}}_{0}\|_{L^{2}}^{2}$.   
\qed

Finally, we claim that in essence, the derivation featured above also holds for the \textit{actual}
$\mathbb D$--distance:
\begin{cor}
\label{corr}
With all notation as before, we have 
\[ \mathbb D^2(N_t, \textsc{m}_0) \leq \frac{g^2}{\sigma}[\mathscr G_\mu(N_0) - \mathscr G_\mu(\textsc{m}_0)] \cdot \text{e}^{-\lambda^\dagger t}.\]
\end{cor}
\begin{proof}
Let $N \in \mathcal B_{\kappa}$ and consider
$$
N_{s}^{\bullet}  =  (1 - s)N + s\text{\textsc{m}}_{0}.
$$
Let $Q_{s}^{\bullet}$ denote the corresponding advective potential
$$
\frac{\partial N_{s}^{\bullet}}{\partial s}  \equiv \text{\textsc{m}}_{0} - N
=
\nabla\cdot(N_{s}^{\bullet}\nabla Q_{s}^{\bullet}) - \Omega_{ N_{s}^{\bullet}}Q_{s}^{\bullet}.
$$
(Clearly, $Q_{s}^{\bullet}$ depends on $s$.)  Now going this route from 
$N\to \text{\textsc{m}}_{0}$ will not necessarily minimize the actual distance functional:
$$
\mathbb D^{2}(N, \text{\textsc{m}}_{0})  \leq  \int_{0}^{1} 
\langle \hspace{-.1cm} \langle  \nabla Q_{s}^{\bullet}  ,
  \nabla Q_{s}^{\bullet}  \rangle  \hspace{-.1cm}  \rangle_{N_{s}^{\bullet}} ~ds.
$$
Therefore an upper bound on the integrated inner product constitutes an upper bound on the actual distance.  

\b{To this end, noting that $N_s^\bullet \in \mathcal B_\kappa$, similar reasoning as in the proof of Claim 1 yields}
$\|Q_s^\bullet\|_{L^2} \leq g^2 \cdot 
\|N - \text{\textsc{m}}_{0}\|_{L^{2}}$.
On the other hand, 
\[\|N - \textsc{m}_0\|_{L^2} \cdot \|Q_s^\bullet\|_{L^2} \geq  -\int_{0}^{1}ds\int_{\mathbb T_{L}^{d}}(N - \text{\textsc{m}}_{0})Q_{s}^{\bullet}~dx
= \int_{0}^{1}\langle \hspace{-.1cm} \langle  \nabla Q_{s}^{\bullet}  ,
  \nabla Q_{s}^{\bullet}  \rangle  \hspace{-.1cm}  \rangle_{N_{s}^{\bullet}}~ds.\]
Combining the above estimates, we thus obtain an analogous conclusion to Claim 1:
\[ \mathbb D^2(N, \textsc{m}_0) \leq g^2 \cdot \|N - \textsc{m}_0\|_{L^2}^2.\]
By Claim 2, $g^2 \cdot \|N- \textsc{m}_0\|_{L^2}^2 \leq \frac{g^2}{\sigma} \cdot \left[\mathscr G_\mu(N) - \mathscr G_\mu (\textsc{m}_0)\right]$.
Thence we may conclude using iteration as in the proof of Theorem \ref{MDKR} that 
$$\mathbb D^{2}(N_{t}, \text{\textsc{m}}_{0})  \leq  
\frac{g^{2}}{\sigma}\left[\mathscr G_{\mu}(N_{0}) - \mathscr G_{\mu}(\text{\textsc{m}}_{0})\right] \cdot \text{e}^{-\lambda^{\dagger} t}.$$
\end{proof}

\noindent \emph{Proof of Theorem \ref{thmm}.}  The establishment of $\mathbb D(\cdot, \cdot)$ as a \emph{bona fide} distance is found in Appendix B and the convergence of the JKO type scheme is the content of Theorem \ref{FVJJ}.  Finally, Corollary \ref{corr} establishes the stated convergence for the distance $\mathbb D(\cdot, \cdot)$.  

\section{Appendix A}
In this appendix, we analyze the discrete time evolution of the fluid density as given in Eq.~\eqref{CVDA}.  
While this equation produces $N_{(k+1)h}$ from $N_{kh}$, in order to avoid clutter, we will set $k = 0$ -- and introduce various other abbreviations to be described shortly.  The ultimate result depends only on properties of $N_{k}$
(\textsc{aka} $N_{0}$) primarily the $\mathcal D_{2}$--norm (a Fourier norm) introduced before and again described below.  Thus, the principal difficulty will be to show that the relevant properties are preserved under iteration.  And, it turns out, it is too much to expect that this is achieved by having the incremental changes in e.g., $N_{0}$, $\nabla^{2}N_{0}$ etc., to always be of order $h$.  Thus a somewhat delicate (albeit presumably standard) ``cancelation'' must be exhibited in the course of our arguments.

\subsection{The Full Equation}  
Equation (\ref{CVDA}) is most conveniently expressed in terms of the variable $\Psi := \log N$.
For the purposes of this appendix, we will abbreviate 
$\Psi_{0} := \log N_{0}$
and
$w_{0} := W\ast{N_{0}}$
with similar notational conventions when 0--subscripts are replaced by 1's.  
In this language, Eq.\eqref{CVDA} reads
\begin{equation}
\label{XMVQ}
\begin{split}
\text{e}^{\Psi_{1} - \Psi_{0}} - 1 =
h[\nabla^{2}\Psi_{1} + \nabla^{2} w_{1} + \nabla \Psi_{1}\cdot \nabla \Psi_{0} &+ \nabla w_{1}\cdot \nabla \Psi_{0}]\\
&\b{-h[\text{e}^{-\Psi_{0}}\Omega_{N_{0}}(\Psi_{1} + w_{1} - \mu)]}.
\end{split}
\end{equation}
Introducing $h\psi := \Psi_{1} - \Psi_{0}$, \b{$hw_{\psi} := w_{1} - w_{0} = W\ast (\text{e}^{\Psi_{0}})(\text{e}^{h\psi} -1)$ and $\Omega_{0}:= \text{e}^{-\Psi_{0}}\Omega_{N_{0}}$}, Eq.~\eqref{XMVQ} now reads
\begin{align}
\label{XPOR}
\frac{\text{e}^{h\psi} - 1}{h}  =   
& h\nabla^{2}\psi +
[\nabla^{2}\Psi_{0} + |\nabla \Psi_{0}|^{2} + \nabla^{2}w_{0} + 
\nabla w_{0}\cdot \nabla\Psi_{0} - \Omega_{0}(\Psi_{0}+ w_{0} \b{- \mu})]
\notag
\\
+ &h[\nabla^{2}w_{\psi} + \nabla w_{\psi}\cdot \nabla \Psi_{0} + \nabla \psi\cdot \nabla \Psi_{0}
\b{- \Omega_0(\psi + w_\psi)}].
\end{align}
The advantage of using the $\Psi$--variables is now manifest: \b{On the right hand side of the equation, all the non--linearities are encoded into the function itself and do not involve the derivatives.  Note further that we have separated the $\Psi_0$--terms from the $\psi$--terms.}

\subsection{Norms}
\label{PLWF}
Our analyses will be essentially classical -- although it is conceivable that with greater effort, a more general treatment would be possible.  In any case we will start with an assumption on 
$\Psi_{0}$ which is slightly stronger than $H^1$. Specifically we will require that $\Psi_{0} \in \mathcal D_{2}$ as described below:

Let $f: \mathbb T_{L}^{d} \to \mathbb R$ have Fourier coefficients $\hat{f}(k)$.  Then 
$$
\|f\|_{\mathcal D_{0}} := 
\frac{1}{L^{d}}\sum_{k} |\hat{f}(k)|
$$
and, if this is finite, then we say $f\in \mathcal D_{0}$.  In general, 
$$
\|f\|_{\mathcal D_{m}} := 
\frac{1}{L^{d}}\sum_{k} |k|^{m}|\hat{f}(k)|
$$
defines the class $\mathcal D_{m}$.  It is noted that these norms obey the usual inequalities, e.g.,
$\|f\|_{\mathcal D_{1}}^{2} \leq \|f\|_{\mathcal D_{2}}\|f\|_{\mathcal D_{0}}$.  These norms also have derivation properties, e.g.,
$$\|fg\|_{\mathcal D_{1}} \leq  \|f\|_{\b{\mathcal D_{1}}}\|g\|_{\mathcal D_{0}} + \|f\|_{\b{\mathcal D_{0}}}\|g\|_{\mathcal D_{1}}.$$

Our precise assumption is that $\Psi_{0} \in \mathcal D_{2}$ with a bound on the norm that does not depend on 
$h$.  The latter is emphasized because, e.g., for the time interval $[0,T]$, we must accommodate the order of 
$Th^{-1}$ iterations of Eq.~\eqref{CVDA}.  Of course a single application is readily accomplished with the result
$\Psi_{1} \sim \Psi_{0} + h\cdot \left[\nabla^{2}\Psi_{0} + |\nabla\Psi_{0}|^{2} + \dots  - \Omega_0w_{0}\right]$.  But this perturbative result, in and of itself, cannot be expected to get us through too many iterations.  For us, among other small matters, the crucial requirement is to show that the actual $\Psi_{k}$'s also have $\mathcal D_{2}$--norms which, \b{for fixed $T$}, is uniformly bounded independent of $h$ (provided that $h$ is sufficiently small) in order that the above heuristic can be continued.  

The above notions will be placed on a more formal footing.  Let us amalgamate into a set \b{$\mathfrak D$ all the relevant input constants, so  the initial $\mathfrak D$ takes the form:}
$$
\b{\mathfrak D_0}  =  
\left\{
\|\Psi_{0}\|_{\mathcal D_{0}},  \|\Psi_{0}\|_{\mathcal D_{2}},
v_{0}, \dots, v_{4}, \|W\|_{\b{\mathcal D_2}}
\right\}
$$ 
where the $v_{m}$ are given by $v_{m}  := \sup_{k} |\hat{W}(k)\|k|^{m}$ and are assumed to be 
finite for $m \leq 4$.  
These are regarded as \textit{fixed} while the time step parameter is to be treated as a variable, albeit ``small''.
In the course of our analysis, various numbers will emerge which will depend on \b{$\mathfrak D_0$ but are uniformly bounded with respect to $h$.}  Then, these numbers are bounded provided the elements of \b{$\mathfrak D_0$} are bounded.  The time--step $h$ itself will be allowed to take on any value smaller than some $h_{0}$ which ultimately \textit{does}
depend on \b{the initial $\mathfrak D_0$}.  But, again, $h_{0}$ will be bounded (below) provided the elements of \b{$\mathfrak D_0$} are bounded (above).  These numbers provide us with the updated version of 
$\mathfrak D$, \b{denoted $\mathfrak D_1$,} which will also have elements which
have only incremented by the order of $h$. \b{We will continue this way to $\mathfrak D_2, \mathfrak D_3$, etc., all of which}, at least for a while,
may be regarded as 
bounded independently of $h$.
Thence, the whole process can be continued throughout some finite interval $[0,T]$, \b{leading to a set $\mathfrak D_\ell$ for each time step $\ell$ so that each element in $\mathfrak D_\ell$ is uniformly bounded.} \b{This way we have} the order of $h^{-1}T$ iterations, with bounds that will depend only on the initial \b{$\mathfrak D_0$} and, perhaps, $T$.

Of course only two of the elements of $\mathfrak D$ are destined to change; later these will be referred to as the 
\textit{mutable} elements.  Anticipated but conspicuously absent from the mutable elements of $\mathfrak D$ is the quantity 
$\|\Psi_{0}\|_{\mathcal D_{1}}$.  The reason is economical rather than esoteric:  Below begins the 
$\mathcal D_{0}$--analysis followed in Subsection \ref{OKIJ} by the $\mathcal D_{2}$--analysis which is still more substantial.  In principal, a $\mathcal D_{1}$ subsection could have been written which, presumably, would have been intermediate.  In practice, we are (at first \b{only}) interested in bounds which permit iteration of the process for \emph{some} positive macroscopic time.  Therefore it proves to be sufficient, even if less efficient, to use
$\|\Psi_0\|_{\mathcal D_{2}}^{\frac{1}{2}}\|\Psi_0\|_{\mathcal D_{0}}^{\frac{1}{2}}$ as an upper bound for 
$\|\Psi_{0}\|_{\mathcal D_{1}}$ in the places where such a bound on this quantity is required.

\subsection{Preliminary Analysis}
We start off with a bound on the $\mathcal D_{0}$--norm of $\psi$:
\begin{proposition}
\label{XTY}
There exist $h_{2} > 0, \mathfrak b_0 > 0$ such that for all $h \leq h_{2}$, \b{there is a solution $\psi$ to Eq.~\eqref{XMVQ}} with $\|\psi\|_{\mathcal D_{0}} \leq \mathfrak b_{0}$.  Further, both
$\mathfrak b_{0}$ and $h_{2}$ depend only on $\mathfrak D_0$ and are uniformly bounded for bounded ranges of these elements. \end{proposition}
\begin{proof}
We start with a rewrite of Eq.~\eqref{XPOR} so that it takes the form
\begin{equation}
\label{XD3W}
\psi - h\nabla^{2}\b{\psi}  =
A + hB_{\psi} - \frac{1}{h}\mathcal E_{2}(h\psi)
\end{equation}
where in the above
$\mathcal E_{2}(x) = \sum_{m\geq2}\frac{x^{m}}{m!}$ (and, for future reference, similarly for $\mathcal E_{1}$)
and $A$ and $B_{\psi}$ correspond to the appropriate bracketed terms in the above mentioned equation:
\[\b{\begin{split}A &= \nabla^2 \Psi_0 + |\nabla \Psi_0|^2 + \nabla^2 w_0 + \nabla w_0 \cdot \nabla \Psi_0 - \Omega_0 (\Psi_0 + w_0 - \mu)\\
B_\psi &= \nabla^2 w_\psi + \nabla w_\psi \cdot \nabla \Psi_0 + \nabla \psi \cdot \nabla \Psi_0 - \Omega_0(\psi + w_\psi). \end{split}}\]
Thus we may write
\begin{equation}
\label{XD3E}
\psi  =  L_{h}^{-1}\left[A + hB_{\psi} - \frac{1}{h}\mathcal E_{2}(h\psi)\right] := \mathscr L_h(\psi)
\end{equation}
where $L_{h} := 1 - h\nabla^{2}$.  We estimate the terms one at a time adding all the results.

Most terms are handled easily with the neglect of $L_{h}^{-1}$.  E.g.,
$$
\|L_{h}^{-1}(\nabla^{2}\Psi_{0})\|_{\mathcal D_{0}}
=
\frac{1}{L^{d}}\sum_{k}\frac{1}{1 + hk^{2}}\cdot k^{2}|\hat{\Psi}_{0}| \leq 
\frac{1}{L^{d}}\sum_{k} k^{2}|\hat{\Psi}_{0}|
=
\|\Psi_{0}\|_{\mathcal D_{2}}.
$$
\textcolor{blue}{(We note here that strictly speaking since $k$ is a vector, we should write $|k|^2$ in the above display, but we have suppressed these absolute values and will continue to do so when the context makes the meaning clear.)}
As a further illustration we have
$$\|L_{h}^{-1}\nabla^{2}w_{0}\|_{\mathcal D_{0}}
\leq
\|w_{0}\|_{\mathcal D_{2}} \leq v_{2}\|\text{e}^{\Psi_{0}}\|_{\mathcal D_{0}} 
\leq v_{2}\text{e}^{\|\Psi_{0}\|_{\mathcal D_{0}}}.$$  
All terms in $A$ can be handled this way.  Since the quantities stemming from the $A$ term are bounded by a function of elements of $\mathfrak D_0$, we have the same statement for $L_h^{-1} (A)$ and so we may write 
\textcolor{blue}{$\|L_h^{-1}(A)\|_{\mathcal D_0} \leq A_0$}
with
\begin{equation}\label{DKRI} 
\begin{split}A_0 = \|\Psi_0\|_{\mathcal D_2} + \|\Psi_0\|_{\mathcal D_1}^2 + v_2 e^{\|\Psi_0\|_{\mathcal D_0}} &+ 
\textcolor{blue}{v_1\text{e}^{\|\Psi_0\|_{\mathcal D_0}} \cdot \|\Psi_0\|_{\mathcal D_1}}\\
&+ 
\textcolor{blue}{ V_0}
  \cdot (\|\Psi_0\|_{\mathcal D_0} + 
\textcolor{blue}{v_0\text{e}^{\|\Psi_0\|_{\mathcal D_0}}}
  + \mu).
  \end{split}
\end{equation}
\textcolor{blue}{In the above, $V_0$ is a bound on $\|\Omega_0\|_{\mathcal D_0}$ in terms of the elements of $\mathfrak D_0$ which we shall not make explicit.  In any case, all terms have been entirely bounded in terms of quantities from $\mathfrak D_0$.}

The $B_{\psi}$--terms as well as the final term now involve $\psi$ itself.  Nevertheless, most of \b{these terms are estimated in a straightforward fashion.}  E.g., 
$$\b{\frac{1}{h}\|L_{h}^{-1}\mathcal E_{2}(h\psi)\|_{\mathcal D_{0}}
\leq \frac{1}{h}\mathcal E_{2}(h\|\psi\|_{\b{\mathcal D_0}})}$$ and similarly for most of the other $B_{\psi}$--terms. \b{For the $\nabla \psi \cdot \nabla \Psi_0$ term, in order to ensure that no $\|\psi\|_{\mathcal D_1}$ appears in the estimate, we use} 
\begin{align}
\label{BBBB}
\|hL_{h}^{-1}(\nabla \psi\cdot \nabla \Psi_{0})\|_{\mathcal D_{0}} 
& \leq
\frac{1}{L^{2d}}\sum_{k,q}
\left |
\frac{h}{1 + hk^{2}} \left(q \hat{\Psi}_{0}(q)\right) \cdot \left((k-q)\hat{\psi}(k-q) \right)
\right |
\notag
\\
& \hspace{-1cm}\leq h\|\psi\|_{\mathcal D_{0}}\|\Psi_{0}\|_{\mathcal D_{2}} +
 \frac{1}{L^{2d}}\sum_{k,q}\frac{h|k|}{1 + hk^{2}} \cdot |q \Psi_{0}(q)| \cdot |\hat{\psi}(k-q)|
 \\
 &\hspace{-1cm}\b{\leq \|\psi\|_{\mathcal D_0} \cdot \left(h\|\Psi_0\|_{\mathcal D_0} + \frac{1}{2} h^{1/2} \|\Psi_0\|_{\mathcal D_1}\right),}
 \notag
\end{align}
\b{where to handle the final term in Eq.~\eqref{BBBB} above, we have used $\frac{h|k|}{1 + hk^{2}} \leq \frac{1}{2}h^{\frac{1}{2}}$.
}

\b{We list bounds on the remaining $B_\psi$--terms below:}
\[\begin{split} \|hL_h^{-1} (\nabla^2 w_\psi)\|_{\mathcal D_0} &\leq \|w_\psi\|_{\mathcal D_2} \\
\|hL_h^{-1} (\nabla w_\psi \cdot \nabla \Psi_0)\|_{\mathcal D_0} &\leq \|w_\psi\|_{\mathcal D_0} \cdot \left(h\|\Psi_0\|_{\mathcal D_0} + \frac{1}{2} h^{1/2} \|\Psi_0\|_{\mathcal D_1}\right)\\
\|hL_h^{-1} (\Omega_0 \psi + w_\psi)) \|_{\mathcal D_0} &\leq h V_0 \cdot \left(\|\psi\|_{\mathcal D_0} + \|w_\psi\|_{\mathcal D_0}\right)
 \end{split}\]
and, as before, we may write final estimates for the $w_\psi$--terms:
\[ \|w_\psi\|_{\mathcal D_0} \leq v_0 e^{\|\Psi_0\|_{\mathcal D_0}} \cdot \frac{1}{h} \mathcal E_1(h\|\psi\|_{\mathcal D_0}), ~~~\|w_\psi\|_{\mathcal D_2} \leq v_2 e^{\|\Psi_0\|_{\mathcal D_0}}\cdot \frac{1}{h} \mathcal E_1(h\|\psi\|_{\mathcal D_0}).\] 

\b{Sorting all these terms, the ``bound'' now takes the form }
\begin{equation}
\label{MSGD}
\begin{split}
\|\psi\|_{\mathcal D_{0}} = \|\mathscr L_h(\psi)\|_{\mathcal D_0}
\leq A_0 + h\beta_{0}\|\psi\|_{\mathcal D_{0}}
&+ b_{0}h^{\frac{1}{2}}\|\psi\|_{\mathcal D_{0}} \\
&+ hG(\|\psi\|_{\mathcal D_{0}})
+ h^{1/2} g(\|\psi\|_{\mathcal D_0})
\end{split}
\end{equation}
where 
\[ \beta_0 = \|\Psi_0\|_{\mathcal D_2} + V_0,~~~b_0 = \frac{1}{2} \|\Psi_0\|_{\mathcal D_1}, \]
\begin{equation}\label{BLIH1}\begin{split} \b{G(\|\psi\|_{\mathcal D_0}) = v_0 e^{\|\Psi_0\|_{\mathcal D_0}}(\|\Psi_0\|_{\mathcal D_0}} &\b{+ V_0) \cdot \frac{1}{h} \mathcal E_1(h\|\psi\|_{\mathcal D_0}) } \\
&\b{+ v_2 e^{\|\Psi_0\|_{\mathcal D_0}} \cdot\frac{1}{h}  \mathcal E_1(h \|\psi\|_{\mathcal D_0}) + \frac{1}{h^2}  \mathcal E_2(h\|\psi\|_{\mathcal D_0})} \end{split}\end{equation}
and 
\begin{equation}\label{BLIH2} \b{g(\|\psi\|_{\mathcal D_0}) = \frac{1}{2}  v_0 e^{\|\Psi_0\|_{\mathcal D_0}} \|\Psi_0\|_{\mathcal D_1} \cdot\frac{1}{h}  \mathcal E_1(h\|\psi\|_{\mathcal D_0}).}\end{equation}

All constants and functions in the estimate depend (uniformly) only on the parameters in $\mathfrak D_0$ 
and the quantities $G$ and $g$ tend down to other such constants as $h\to0$: 
\[\b{ \frac{1}{h} \mathcal E_1(h\|\psi\|_{\mathcal D_0}) \longrightarrow \|\psi\|_{\mathcal D_0},~~~ \frac{1}{h^2} \mathcal E_2(h\|\psi\|_{\mathcal D_0}) \longrightarrow \frac{1}{2} \|\psi\|_{\mathcal D_0}.}\]
Importantly, the quantity $\|\psi\|_{\mathcal D_{1}}$ does \textit{not} appear in the estimate.
Thus, we may tentatively 
conclude that $\|\psi\|_{\mathcal D_{0}} \lesssim A_{0}$.  However, it is noted that given the form of the right hand side of the display in Eq.~\eqref{MSGD} there is also the possible interpretation of a trivial (i.e., infinite) bound, an issue we now address.

Let us denote the upper bound on $\|\psi\|_{\mathcal D_0}$ from Eq.~\eqref{MSGD} by $\Xi_h(\|\psi\|_{\mathcal D_{0}})$ and let us examine the corresponding recursive equation 
\begin{equation}
\label{YZBQ}
\zeta_{0}  =  \Xi_h(0) = A_0,~~~\zeta_{k+1}  =  \Xi_h(\zeta_{k}).
\end{equation}
\b{If the iterates were to approach a fixed point at $x$, we would have }
\[ \b{x = \Xi_h(x) = A_0 + h^{1/2} \left[b_0 x + g(x)\right] + h\left[\beta x + G(x)\right].}\]
\b{Clearly, for $h =0$, the equation above is satisfied at $x = A_0$. 
For $h > 0$, starting from $x=0$, the right hand side would still exceed the left hand side till $x = A_0$ and certainly the right hand side would dominate for very large values of $x$} (indeed, the function $\Xi(h, x)$ is increasing and convex in $x$ because $G(x)$ and $g(x)$ are both convex). 

However, we claim that for $h$ sufficiently small, the two functions are guaranteed to cross at some point after $A_0$:

\noindent\b{\textbf{Claim.} For any $\eta > 0$, there is some $h_\eta > 0$ such that for all $h < h_\eta$, there is some $x_h < A_0 + \eta$ such that $\Xi_h(x_h) = x_h$.}

\noindent\b{\emph{Proof of Claim.} Observing that $\Xi_h'(x)$ is increasing it follows that for every $\eta > 0$, 
\[ \Xi_h(A_0 + \eta) \leq A_0 + \Xi_h'(A_0 + \eta) \cdot (A_0 + \eta). \]
Thus if we choose $h_\eta > 0$ sufficiently small so that $\Xi_{h_\eta}'(A_0 + \eta) \cdot (A_0 + \eta) < \eta$, then $\Xi_{h_\eta}(A_0 + \eta) < A_0 + \eta$ (we note also that $\Xi'_h(x)$ is monotonically \emph{decreasing} in $h$ so once some $h_\eta$ is found we have the result for all $h < h_\eta$).  Since $\Xi_h(A_0) > A_0$, the required fixed point exists by continuity of $x - \Xi_h(x)$. }
\qed 
\\
It follows \b{from the convexity of $\Xi_h$} and from the claim that there is some $h_{1} > 0$ such that for $h < h_{1}$ there is a $\zeta_{\sharp} = \zeta_{\sharp}(h)$
which is the unique stable fixed point of Eq.~\eqref{YZBQ} so that if 
$\zeta_{k} < \zeta_{\sharp}$ then 
$\zeta_{k} < \zeta_{k+1} < \zeta_{\sharp}$ and $\lim_{k\to\infty} \zeta_{k}  =  \zeta_{\sharp}$ \b{-- it is clear that $\zeta_k < \zeta_{k+1} = \Xi_h(\zeta_k)$ since for $x < \zeta_\sharp$ we have $\Xi_h(x) > x$; on the other hand, by monotonicity of $\Xi_h$, $\zeta_{k+1} = \Xi_h(\zeta_k) < \Xi(\zeta_\sharp) = \zeta_\sharp$}.  

Recall from Eq.~\eqref{XD3E} that we have $\psi = \mathscr L_h(\psi)$.  Thus, we may define the iterates 
\[\psi_0 = L_h^{-1}(A_0),~~~ \psi_{k+1} = \mathscr L_h(\psi_k) \]
so that by Eq.~\eqref{MSGD} and the nature of $\zeta_\sharp$, we have that for all $k$, 
\[ \|\psi_{k+1}\|_{\mathcal D_0} = \|\mathscr L_h(\psi_k)\|_{\mathcal D_0} \leq \Xi_h(\|\psi_k\|_{\mathcal D_0}) < \zeta_\sharp.\]
  Let $\psi$ be a weak limit of the $\psi_k$'s.  It remains to identify $\psi$ with the object featured in Eq.~\eqref{XD3W}.

\b{To this end we consider 
$\delta_{k} := \psi_{k} - \psi_{k-1}$.  Since $\mathscr L_h(\psi) = A + hB_\psi - \frac{1}{2} \mathcal E_2(h\psi)$, }
\[ \b{\delta_{k+1} = \mathscr L_h(\psi_k) - \mathscr L_h(\psi_{k-1}) = h(B_{\psi_{k}} - B_{\psi_{k-1}}) - \frac{1}{h} \left[\mathcal E_2 (h\psi_{k}) - \mathcal E_2(h\psi_{k-1})\right].}\]
\b{From our previous estimates, it follows that }
\begin{equation}\label{KHN} \b{\begin{split}
\|\delta_{k+1}\|_{\mathcal D_0} &\leq b_0 h^{1/2} \|\delta_k\|_{\mathcal D_0} + h \beta_0 \|\delta_k\|_{\mathcal D_0} + h \tilde G(\|\delta_k\|_{\mathcal D_0}, \|\psi_{k-1}\|_{\mathcal D_0}) + h^{1/2} \tilde g(\|\delta_k\|_{\mathcal D_0}, \|\psi_{k-1}\|_{\mathcal D_0}) \\
&:= h^{1/2} \|\delta_k\|_{\mathcal D_0} \cdot \Gamma_h(\|\delta_k\|_{\mathcal D_0}, \|\psi_{k-1}\|_{\mathcal D_0}), \end{split}}\end{equation}
\b{where $\tilde G(\cdot)$ and $\tilde g(\cdot)$ are defined analogously as $G(\cdot)$ and $g(\cdot)$ in Eq.~\eqref{BLIH1} and \eqref{BLIH2}, corresponding now to estimates involving \emph{differences} of $\psi_k$'s, e.g.,}
\[ \b{h(w_{\psi_{k}} - w_{\psi_{k-1}}) = W*(\text{e}^{\Psi_0 + h \psi_{k-1}}) (\text{e}^{h\delta_k} - 1).}\]
\textcolor{blue}{Since $\tilde G(\|\delta_k\|_{\mathcal D_0}, \|\psi_{k-1}\|_{\mathcal D_0})$ and $\tilde g(\|\delta_k\|_{\mathcal D_0}, \|\psi_{k-1}\|_{\mathcal D_0})$ tend to definitive constants which are uniformly bounded as $h$ goes to zero (c.f., the discussion following Eq.~\eqref{BLIH1} and \eqref{BLIH2}) so does $\Gamma_h(\|\delta_k\|_{\mathcal D_0}, \|\psi_{k-1}\|_{\mathcal D_0})$.  It follows from Eq.~\eqref{KHN} that $\delta_k$ tends to zero as $k$ tends to infinity for all $h$ sufficiently small.}

More precisely, consider $\zeta_{\sharp\sharp}$ which is the limit as $h\to h_{1}$ of $\zeta_{\sharp}(h)$.  Let $h_{2}$ be defined by 
$$
[h_{2}]^{\frac{1}{2}}\times [
\hspace{-.4cm}
\sup_{\substack{ h < h_{1}\\ b < \zeta_{\sharp\sharp},a < 2\zeta_{\sharp\sharp} }}
\hspace{-.4cm}
\Gamma_h(a, b)] = 1.
$$
\b{By Eq.~\eqref{KHN}, the above choice of $h_2$ implies that for $h < h_{2}$, there is some $\alpha_h < 1$ such that $\|\delta_{k+1}\|_{\mathcal D_0}  < \alpha_h \|\delta_k\|_{\mathcal D_0}$ for all $k$. }
It follows that $\psi_{k}$ converges (strongly in $\mathcal D_{0}$) to the $\psi$ given in
Eq.~\eqref{XD3E} and for $h < h_{2}$ we have (for all $j$) the  $\psi_{j}$'s  and the limiting $\psi$ are bounded in $\mathcal D_0$ by 
$\zeta_{\sharp}(h_{2})$.  Moreover, all parameters, $h_{1}$, $h_{2}$ \dots $\zeta_{\sharp}(h_{2})$
depend only on the parameters of $\mathfrak D_0$ and are uniformly controlled by these elements.  
\end{proof}

\subsection{Advanced Analysis}
\label{OKIJ}
The situation concerning the $\mathcal D_{2}$--norm of $\psi$ will not be as straightforward as that of the above -- indeed, there is no hope for a result analogous to Proposition \ref{XTY}.
In particular, let us investigate the very first term
\begin{equation}
\label{PU9Y}
\psi_{\star} := L_{h}^{-1}(\nabla^{2}\Psi_{0}).
\end{equation}
While it is clear that $\|\psi_{\star}\|_{\mathcal D_{2}} < \infty$, this norm might well be divergent as 
$h\downarrow 0$; e.g., 
$\|h\psi_{\star}\|_{\mathcal D_{2}}$ could be a sublinear power of $h$.  However, as will be demonstrated, \b{if $\Psi_0$ has this behavior}, these circumstances are actually beneficial. \b{Indeed, due to the positivity of the operator $-\nabla^2$, adding $\psi_\star$ would \emph{reduce} the overall magnitude of the Fourier coefficients: Explicitly, let us define the ``preliminary correction''}
$$
\Psi_{\star} := \Psi_{0} + h\psi_{\star}.
$$
Then
\begin{equation}
\label{ABFG}
\hat{\Psi}_{\star}(k)  = \hat{\Psi}_{0}(k)  - \frac{hk^{2}}{1 + hk^{2}}\hat{\Psi}_{0}(k)
  =  \frac{\hat{\Psi}_{0}(k)}{1 + hk^{2}}.
\end{equation}
i.e., the magnitude of \textit{every} non--zero mode has been reduced.

Hence, the task at hand will be to show that the rest of $\psi$ does not disrupt this \b{beneficial} effect.  Specifically, defining
$$
\psi_{\bullet} := \psi - \psi_{\star},
$$
our aim is to show that the difference, $\|\psi_{\bullet}\|_{\mathcal D_{2}} - \|\psi_{\star}\|_{\mathcal D_{2}}$,
is either negative or of order unity.  We remark that in contrast to the preceding analysis, there is no reason to expect matching with powers of 
$h$.  Thus, we will be working directly with 
$h\psi_{\star}$, $h\psi_{\bullet}$, etc., even though, at times, appearances of $h$, e.g., multiplying both sides of an equation, may seem redundant.   

The preliminary challenge arises from the inhomogeneous terms.  We define $r_{\bullet}$ and $s_{\bullet}$ via:
$$
hr_{\bullet}
:=
hL^{-1}_h
(|\nabla\Psi_{0}|^{2})
\hspace{.5 cm}
\text{and}
\hspace{.5 cm}
hs_{\bullet}
:=
hL^{-1}_h
(\nabla\Psi_{0}\cdot\nabla w_{0}).
$$
Our first goal is an estimate on their $\mathcal D_{2}$ norms.  We start by invoking the relevant length scale for these problems:
\begin{definition}
Let 
$p_{0} = p_{0}(h)$ be such that
$$
\frac{1}{L^{d}}\sum_{p: |p| \geq p_{0}}|p \hat{\Psi}_{0}(p)|  \geq h
$$
while without the last shell,
$$
\frac{1}{L^{d}}\sum_{p: |p| > p_{0}}|p \hat{\Psi}_{0}(p)| <  h.
$$
\end{definition}
\noindent \textbf{Claim} A$_{1}$.  
There is an $a$ depending only on $\mathfrak D_0$ such that if 
$p_{0} > ah^{-\frac{1}{2}}$ then 
$$
\|h\psi_{\star}\|_{\mathcal D_{2}} \geq 
2
\left(\|hr_{\bullet}\|_{\mathcal D_{2}} + \|hs_{\bullet}\|_{\mathcal D_{2}}\right).
$$

\noindent\textit{Proof of Claim.}  
We first note that, \textit{a priori}, $\|hr_{\bullet}\|_{\mathcal D_{2}}$ and $\|hs_{\bullet}\|_{\mathcal D_{2}}$
do not exceed the order of 
$h^{\frac{1}{2}}$.  Indeed, we write

\begin{equation}
\label{PYVW}
hk^{2}\hat{r}_{\bullet}(k)  =  
-\frac{1}{L^{2d}}\frac{hk^{2}}{1 + hk^{2}}
\sum_{q}q\hat{\Psi}_{0}(q)\cdot(k-q)\hat{\Psi}_{0}(k-q)
\end{equation}so, taking absolute values etc., and bringing one factor of $k$ inside the sum,
$$
\b{|hk^{2}\hat{r}_{\bullet}(k)|
\leq
\frac{1}{L^{2d}}\frac{h|k|}{1 + hk^{2}}
\sum_{q}\left(q^{2}|k-q| + (k-q)^{2}|q|\right) \cdot |\hat\Psi_{0}(q)\|\hat\Psi_{0}(k-q)|.}
$$
Using $h^{\frac{1}{2}}|k|/(1 + hk^{2}) \leq \frac{1}{2}$, and summing over $k$, we are left with 
\b{$\frac{1}{2}h^{\frac{1}{2}}\times 2\times \|\Psi_{0}\|_{\mathcal D_{1}}\|\Psi_{0}\|_{\mathcal D_{2}}$:}
$$
\|hr_{\bullet}\|_{\mathcal D_{2}}
\leq  h^{\frac{1}{2}} \cdot \|\Psi_{0}\|_{\mathcal D_{1}}\|\Psi_{0}\|_{\mathcal D_{2}}.
$$
Similarly, 
$$
\|hs_{\bullet}\|_{\mathcal D_{2}}
\leq  \frac{1}{2}h^{\frac{1}{2}}\cdot \left(\|\Psi_{0}\|_{\mathcal D_{1}}\|w_{0}\|_{\mathcal D_{2}} + \|\Psi_{0}\|_{\mathcal D_{2}}\|w_{0}\|_{\mathcal D_{1}}
\right)
$$
On the other hand,
$$
\|h\psi_{\star}\|_{\mathcal D_{2}}
\geq
\frac{1}{L^{d}}\sum_{k: |k| \geq p_{0}}
\frac{hk^{2}}{1 + hk^{2}} \cdot
k^{2}|\hat{\Psi}_{0}(k)|
\geq
\frac{1}{L^{d}} \frac{hp_{0}^{3}}{1 + hp_{0}^{2}}
\sum_{k: |k| \geq p_{0}}
|k\|\hat{\Psi}_{0}(k)|
\geq
\frac{h^{2}p_{0}^{3}}{1 + hp_{0}^{2}}.
$$
Thus, if $p_{0} \geq ah^{-\frac{1}{2}}$ where $a$ is given by 
$$
\b{\frac{a^{3}}{1 + a^{2}}  =  
 2\left[\|\Psi_{0}\|_{\mathcal D_{1}}\|\Psi_{0}\|_{\mathcal D_{2}}
+
\frac{1}{2}(\|\Psi_{0}\|_{\mathcal D_{1}}\|w_{0}\|_{\mathcal D_{2}} + \|\Psi_{0}\|_{\mathcal D_{2}}\|w_{0}\|_{\mathcal D_{1}})
\right]}
$$
the claim is established.  \hspace{1 cm}  $\blacksquare$

\vspace {.5 cm}
We may thus proceed under the assumption that
$p_{0} \leq ah^{-\frac{1}{2}}$ since otherwise, the $r_{\bullet}$ and $s_{\bullet}$ terms are well in hand.  

\noindent \textbf{Claim} A$_{2}$. Our next claim is that, under the assumption $p_{0} \leq ah^{-\frac{1}{2}}$,
both $r_{\bullet}$ and $s_{\bullet}$ admit the bounds 
\begin{align}
\|hr_{\bullet}\|_{\mathcal D_{2}}  \leq  C_{r}\|hL_{h}^{-1}\Psi_{0}\|_{\mathcal D_{3}} + hc_{r}
\notag
\\
\|hs_{\bullet}\|_{\mathcal D_{2}}  \leq  C_{s}\|hL_{h}^{-1}\Psi_{0}\|_{\mathcal D_{3}} + hc_{s}
\end{align}
where $C_{r}, \dots, c_{s}$ are constants which depend only on $\mathfrak D_0$.

\noindent\textit{Proof of Claim.}  
Let us proceed with the analysis of Eq.~\eqref{PYVW} taking absolute values etc., and summing over $k$ at fixed $q$.  
First, we investigate the region where $|k-q| > p_{0}$.  Here we may use 
$hk^{2}/[1 + hk^{2}] < 1$ leaving us with 
$$
\frac{1}{L^{d}}|q \hat{\Psi}_{0}(q)| \cdot \frac{1}{L^{d}}\sum_{k: |k-q| > p_{0}}|(k-q) \hat{\Psi}_{0}(k-q)|
\leq
\frac{1}{L^{d}}|q \hat{\Psi}_{0}(q)| \cdot h.
$$
The summation over $q$ gives the bound $h\|\Psi_{0}\|_{\mathcal D_{1}}$ which is part of the $c_{r}$--term.

What remains to be estimated is the quantity
$$
\frac{1}{L^{2d}}
\sum_{k,q: |k-q| \leq p_{0}}
\frac{hk^{2}}{1 + hk^{2}} \cdot 
|q \hat{\Psi}_{0}(q)| \cdot |(k-q)\hat{\Psi}_{0}(k-q)|.
$$
\b{Similarly to the above, if we first restrict summation over $q$ to $|q| > p_0$, we may divest of the factor $hk^{2}/(1 + hk^{2})$ and, as an upper bound, summing over $(k-q)$ yields a factor of $\|\Psi_0\|_{\mathcal D_1}$ which is then multiplied by a factor of $h$ from the sum over $q$.  Thus we arrive at another estimate of $h\|\Psi_{0}\|_{\mathcal D_{1}}$ which we add to the $c_{r}$--term.}

We are left with the case where $|q| < p_0$ and $|k-q| < p_{0}$.  Here, for the $k^{2}$ in the numerator we write
$k^{2} = q^{2} + 2q \cdot (k-q) + (k-q)^{2}$ giving us three terms to estimate the first of which is
$$
\frac{1}{L^{2d}}\sum_{k,q:|q|, |k-q| \leq p_{0}}\frac{h}{1 + hk^{2}} \cdot 
|q^{3} \hat{\Psi}_{0}(q)| \cdot
|(k-q) \hat{\Psi}_{0}(k-q)|
$$
Now $\frac{1}{1 + hk^{2}} < 1$ and also $1 + hq^{2} \leq 1 + a^{2}$ so the upshot is that the above term is bounded above by
$$
\frac{1}{L^{2d}}h(1 + a^{2})\sum_{k,q}\frac{1}{1 + hq^{2}} \cdot 
|q^{3} \hat{\Psi}_{0}(q)| \cdot
|(k-q) \hat{\Psi}_{0}(k-q)|
$$
where we have now relaxed the restriction on the range of summation.  Summing over $k$ we acquire our first contribution to $C_{r}$, namely $(1 + a^{2})\|\Psi_{0}\|_{\mathcal D_{1}}$.

The second term is the quantity
$$
\frac{2}{L^{2d}}\sum_{k,q:|q|, |k-q| \leq p_{0}}\frac{h}{1 + hk^{2}} \cdot 
|q^{2} \hat{\Psi}_{0}(q)| \cdot
|(k-q)^{2} \hat{\Psi}_{0}(k-q)|.
$$
Here we can relax the restriction over the summation and use $\frac{1}{1 + hk^{2}} < 1$ to arrive at the bound of 
$2h\|\Psi_{0}\|^{2}_{\mathcal D_{2}}$ which is another contribution to the $c_{r}$--term.  
Our third term is identical to the first with the roles of $q$ and $k-q$ switched and may be estimated by the same procedure.  

The analysis of $s_{\bullet}$ follows a similar set of procedures.  We will dispense with the details and state the result:
$$
C_{s}  =  (1 + a^{2}) \cdot \|w_{0}\|_{\mathcal D_{1}}
$$
and 
$$
c_{s}  = \b{2} \|w_{0}\|_{\mathcal D_{1}}  +  2\|\Psi_{0}\|_{\mathcal D_{2}}\|w_{0}\|_{\mathcal D_{2}}
+ \|\Psi_{0}\|_{\mathcal D_{1}}\|w_{0}\|_{\mathcal D_{3}}.
$$
The claim is established.  \hspace {1 cm}  $\blacksquare$

Thus so far, on the basis that $p_{0} \leq ah^{-\frac{1}{2}}$, we now have 
the $r_{\bullet}$  and $s_{\bullet}$ terms essentially bounded by 
$\|hL_{h}^{-1}\Psi_{0}\|_{\mathcal D_{3}}$.  The next step is the following:

\noindent \textbf{Claim} A$_{3}$.  \b{If $p_0 \leq ah^{-1/2}$, then either}
$$
\|h\psi_{\star}\|_{\mathcal D_{2}} > (C_{r} + C_{s}) \cdot \|hL_{h}^{-1}\Psi_{0}\|_{\mathcal D_{3}} 
$$
(where the difference may be considerable) or both $\|\psi_{\star}\|_{\mathcal D_{2}}$ and 
$\|L_{h}^{-1}\Psi_{0}\|_{\mathcal D_{3}}$ are bounded above by constants depending only on $\mathfrak D_0$.

\noindent\textit{Proof of Claim.}  
We \b{are to compare}:
\begin{equation}\label{CMRI}
\frac{q^{4}}{1 + hq^{2}}|\hat{\Psi}_{0}(q)|
\hspace{.25cm}
\text{vs.}
\hspace{.25cm}
(C_{r} + C_{s}) \frac{|q^{3}|}{1 + hq^{2}}|\hat{\Psi}_{0}(q)|;
\end{equation}
obviously if $|q| \geq (C_{r} + C_{s})$ the terms contributing to $\|\psi_{\star}\|_{\mathcal D_{2}}$ are dominant and we are done.  Let us define $q_{0} := 2(C_{r} + C_{s})$
and write $\|\psi_{\star} \|_{\mathcal D_{2}}= \underline{a} + b$ where 
$$
\underline{a}  =  \sum_{|q| \leq q_{0}}\frac{q^{4}}{1 + hq^{2}}|\hat{\Psi}_{0}(q)|,
\hspace{.75 cm}
b  =  \sum_{|q| > q_{0}}\frac{q^{4}}{1 + hq^{2}}|\hat{\Psi}_{0}(q)|
$$
with a similar decomposition ($|q| \leq q_{0}$, $|q| > q_{0}$) for $(C_{r} + C_{s}) \cdot \|L_{h}^{-1} \Psi_0\|_{\mathcal D_{3}}$ denoted by 
$\underline{\alpha}$ and $\beta$.  
So, let us suppose $\underline{a} + b \leq \underline{\alpha} + \beta$.
Since we have arranged $b \geq 2\beta$
this implies that $\underline{a} \leq \underline{\alpha} - \beta$ and hence $\underline{\alpha} \geq \beta$
and so
$$
\|\psi_{\star}\|_{\mathcal D_{2}}
\leq
2\underline{\alpha}  =  2(C_{r} + C_{s})\sum_{|q| \leq q_{0}}\frac{|q|^{3}}{1 + hq^{2}}|\hat{\Psi}(q)| 
\leq q_{0}^{2} \cdot \|\Psi_{0}\|_{\mathcal D_{2}}.
$$
I.e., $\|h\psi_{\star}\|_{\mathcal D_{2}}$ is actually of order $h$.  The same bound (and conclusion) holds for 
$\|L^{-1}_{h}\Psi_{0}\|_{\mathcal D_{3}}$ which (also) does not exceed
$2\underline{\alpha}$.
\hspace{.5 cm}
$\blacksquare$

\b{With these results in hand, we can now establish:}

\begin{proposition}
\label{DBFS}
The $\mathcal D_{2}$--norms of $\Psi_{0}$ and its successor $\Psi_{1}$, acquired after one iteration of the discretization, satisfy
$$
\|\Psi_{1}\|_{\mathcal D_{2}}  -  \|\Psi_{0}\|_{\mathcal D_{2}}
\leq
\mathfrak b_{2} h
$$
where $\mathfrak b_{2} > 0$ depends only on the elements of $\mathfrak D_0$.  
\end{proposition}
\noindent We reiterate that the left hand side in the above display can be considerably negative.
\begin{proof}
\b{Since $\Psi_1 = \Psi_0 + h \psi$, we certainly have $\|\Psi_1\|_{\mathcal D_2} - \|\Psi_0\|_{\mathcal D_2} \leq \|h\psi\|_{\mathcal D_2}$. 
Next we recall
\[ h\psi_\star = h L_h^{-1}(\nabla^2 \Psi_0),~~~h\hat\psi_\star(k) = - \frac{hk^2}{1 + hk^2} \hat \Psi_0(k), ~~\mbox{for all } k.\]
It follows that if we write as described before $h\psi = h\psi_\star + h\psi_\bullet$, then 
\[ |\hat\Psi_1(k)| \leq |\hat \Psi_0(k)| \cdot \left(1 - \frac{hk^2}{1+hk^2}\right) + |h\hat \psi_\bullet(k)| = |\hat \Psi_0(k)| - |h\hat \psi_\star(k)| + |h\hat \psi_\bullet(k)|.\]
Multiplying by $k^2$ and summing we see that 
\[ \|\Psi_1\|_{\mathcal D_2}  - \|\Psi_0\|_{\mathcal D_2} \leq  \|h\psi_\bullet\|_{\mathcal D_2} - \|h\psi_\star\|_{\mathcal D_2}.\]
It follows from the above display that in case $\|h\psi_\bullet\|_{\mathcal D_2}$ is \emph{not} of order $h$, then the proof of the proposition amounts to establishing the statement that
$\|h\psi_{\bullet}\|_{\mathcal D_{2}}  -  \|h\psi_{\star}\|_{\mathcal D_{2}} \leq \mathfrak b_2 h$.}

We have 
\[\begin{split} h\psi_\bullet &= h\psi - h\psi_\star = hr_\bullet + hs_\bullet\\
& + hL_h^{-1} \left[\nabla^2w_0 -  \Omega_0(\Psi_0 + w_0 - \mu)\right] + h^2 (\nabla \psi \cdot \nabla \Psi_0 + \nabla \Psi_0 \cdot \nabla w_\psi + \nabla^2 w_\psi) - \mathcal E_2(h\psi).
\end{split}\]
There are three terms in the expression for $h\psi_{\bullet}$ which must be dealt with explicitly:
These are the $r_{\bullet}$ and $s_{\bullet}$--terms as well as the term
$h^{2}\nabla\psi\cdot\nabla\Psi_{0}$.
All other terms can be handled with straightforward methods.  We shall be content with a couple of examples:  
$$
\|hL_{h}^{-1}(\nabla^{2}w_{0})\|_{\mathcal D_{2}}  =
\frac{1}{L^{d}}\sum_{k}\frac{hk^{2}}{1 + hk^{2}}\cdot k^{2}|\hat{w}_{0}(k)|
\leq \frac{h}{L^{d}}\sum_{k}k^{4}|\hat{W}(k)| \cdot |\hat{N}_{0}|
\leq
hv_{4}\text{e}^{\|\Psi_{0}\|_{\mathcal D_0}}
$$
and
\begin{align}
\|h^{2}L_{h}^{-1}(\nabla \Psi_0&\cdot \nabla w_\psi)\|_{\mathcal D_{2}}
=
\frac{h}{L^{2d}}\sum_{k}\frac{hk^{2}}{1 +hk^{2}}
\sum_{q}|q\hat{\Psi}_{0}(q)| \cdot |(k-q)\hat{w}_{\psi}(k-q)|
\notag
\\
&\leq
h v_{1}\|\Psi_{0}\|_{\mathcal D_{1}} \cdot \text{e}^{\|\Psi_{0}\|_{\mathcal D_{0}}}\frac{1}{h}\mathcal E_{1}(h\|\psi\|_{\mathcal D_0})
\leq
h v_{1}\|\Psi_{0}\|_{\mathcal D_{1}} \cdot
\text{e}^{\|\Psi_{0}\|_{\mathcal D_{0}}}  \frac{1}{h}\mathcal E_{1}(h\mathfrak b_{0})
\notag
\end{align}
(The quantity $\mathfrak b_{0}$ is defined in the statement of Proposition \ref{XTY}, \b{i.e., $\|\psi\|_{\mathcal D_0} \leq \mathfrak b_0$.})
The result is that we may bound (the sum of) all these terms by an $h\tilde{A}(h)$ with $\tilde{A}$ bounded and tending to 
some $\tilde{A}(0)$ as $h\to 0$.  This leaves -- in addition to the $r_{\bullet}$ and $s_{\bullet}$--terms -- the quantity
$h^{2}(\nabla\Psi_0\cdot\nabla\psi)$ which we now estimate:  Writing $\psi = \psi_{\star} + \psi_{\bullet}$, we have \b{(again using $\frac{h|k|}{1 + hk^2} \leq \frac{1}{2} h^{1/2}$)}
\begin{align}
\label{CMMM}
h^{2}\|L_{h}^{-1}(\nabla \Psi_0 \cdot \nabla \psi_{\bullet})\|_{\mathcal D_2}
&=
\frac{1}{L^{2d}}\sum_{k,q}\frac{h^{2}k^{2}}{1 + hk^{2}} \cdot |q\hat{\psi}_{\bullet}(q)| \cdot |(k-q) \hat{\Psi}_{0}(k-q)|
\notag
\\
&\leq \frac{1}{L^{2d}}\frac{1}{2}h^{\frac{3}{2}}\sum_{k,q}
|\hat{\psi}_{\bullet}(q) \hat{\Psi}_{0}(k-q)| \cdot \left(q^{2}|k-q| + (k-q)^{2}|q|\right)
\notag
\\
&\b{\leq 
\frac{1}{2}h^{\frac{1}{2}} \left(\|h\psi_{\bullet}\|_{\mathcal D_{2}}\|\Psi_{0}\|_{\mathcal D_{1}} +
h^{\frac{1}{2}}\|h\psi_{\bullet}\|^{\frac{1}{2}}_{\mathcal D_{2}}\mathfrak b^{\frac{1}{2}}_{0}\|\Psi_{0}\|_{\mathcal D_{2}}
\right).}
\end{align}
\textcolor{blue}{In the last step we have used
$\|\psi_{\bullet}\|_{\mathcal D_{0}} \leq \mathfrak b_{0}$ which is admissible 
since in the derivation in Proposition \ref{XTY} of
$\|\psi\|_{\mathcal D_0} \leq \mathfrak b_0$ we estimated the absolute value of each
successive term the \textit{first} of which (c.f., Eq.~\eqref{DKRI}) was exactly
$\psi_\star$.  Thus the bound derived in Proposition \ref{XTY} actually amounts to 
the stronger bound 
$\|\psi_{\star}\|_{\mathcal D_{0}} + \|\psi_{\bullet}\|_{\mathcal D_{0}} \leq \mathfrak b_{0}$.}  We acquire an estimate similar to that in Eq.~\eqref{CMMM} for the $\psi_{\star}$--term (explicitly, Eq.~\eqref{CMMM} with $\psi_\star$ replacing $\psi_\bullet$).

We amalgamate our upper bound on $\|h\psi_{\bullet}\|_{\mathcal D_{2}}$:
\begin{align}
\|h\psi_{\bullet}\|_{\mathcal D_{2}} \leq
\tilde{A}h  +  \|hr_{\bullet}\|_{\mathcal D_{2}} 
&
+ \|hs_{\bullet}\|_{\mathcal D_{2}}
+ \frac{1}{2}h^{\frac{1}{2}}
\left(\|h\psi_{\bullet}\|_{\mathcal D_{2}}\|\Psi_{0}\|_{\mathcal D_{1}} +
h^{\frac{1}{2}}\|h\psi_{\bullet}\|^{\frac{1}{2}}_{\mathcal D_{2}}\mathfrak b^{\frac{1}{2}}_{0}\|\Psi_{0}\|_{\mathcal D_{2}}
\right)
\notag
\\
 & +
\frac{1}{2}h^{\frac{1}{2}}
\left(\|h\psi_{\star}\|_{\mathcal D_{2}}\|\Psi_{0}\|_{\mathcal D_{1}} +
h^{\frac{1}{2}}\|h\psi_{\star}\|^{\frac{1}{2}}_{\mathcal D_{2}}\mathfrak b^{\frac{1}{2}}_{0}\|\Psi_{0}\|_{\mathcal D_{2}}
\right).
\end{align}
Let us discuss the term(s) in the last line of the above display: \b{We emphasize that the last two bracketed terms are identical under the exchange $\|\psi_\star\|_{\mathcal D_2} \leftrightarrow \|\psi_\bullet\|_{\mathcal D_2}$. If the bracketed term on the last line (the $\psi_\star$ terms) exceeds the corresponding bracketed term just preceding (that have $\psi_{\star}$ replaced by $\psi_{\bullet}$) then we would immediately conclude that $\|\psi_\bullet\|_{\mathcal D_2} \leq \|\psi_\star\|_{\mathcal D_2}$ and we would be done.  Therefore let us assume that this is not the case}.  

Moreover, 
if 
$h^{\frac{1}{2}}\|h\psi_{\bullet}\|_{\mathcal D_{2}}^{\frac{1}{2}}\mathfrak b_{0}^{\frac{1}{2}}\|\Psi_{0}\|_{\mathcal D_{2}}
\geq \|h\psi_{\bullet}\|_{\mathcal D_{2}}\|\Psi_{0}\|_{\mathcal D_{1}}$ \b{
this would imply (using $\|\Psi_0\|_{\mathcal D_1} \leq \|\Psi_0\|_{\mathcal D_0}^{1/2} \|\Psi_0\|_{\mathcal D_2}^{1/2}$) that
$\|h\psi_{\bullet}\|_{\mathcal D_{2}}  \leq h\|\Psi_0\|_{\mathcal D_2}$}; 
once $\|h\psi_{\bullet}\|_{\mathcal D_{2}}$ is of order $h$, it is no longer important 
whether or not it exceeds $\|h\psi_{\star}\|_{\mathcal D_{2}}$, \b{so we may assume this is also not the case.}
Thus, there is no loss of generality if we proceed under both \b{(negative)} assumptions replacing \b{(as an upper bound)} the two bracketed terms in the above display by $2h^{\frac{1}{2}}
\|h\psi_{\bullet}\|_{\mathcal D_{2}}\|\Psi_{0}\|_{\mathcal D_{1}}$.  In this way we we arrive at the tentative estimate
\begin{equation}
\label{XZIU}
\b{\mathfrak a_h \|h\psi_\bullet\|_{\mathcal D_2}:= (1 - 2h^{\frac{1}{2}}\|\Psi_{0}\|_{\mathcal D_{1}}}) \cdot 
\|h\psi_{\bullet}\|_{\mathcal D_{2}} \leq
 \|hr_{\bullet}\|_{\mathcal D_{2}} 
+ \|hs_{\bullet}\|_{\mathcal D_{2}} + \tilde Ah.
\end{equation}

Next, we may assume that, as discussed in Claim A$_{1}$,  the quantity $p_{0}$ does not exceed 
$ah^{-\frac{1}{2}}$ since otherwise, automatically,  
$\|hr_{\bullet}\|_{\mathcal D_{2}} 
+ \|hs_{\bullet}\|_{\mathcal D_{2}}$ is dominated by \b{$\frac{1}{2}\|h\psi_{\star}\|_{\mathcal D_{2}}$} and the inequality in 
Eq.~\eqref{XZIU} becomes \b{$\|h\psi_\bullet\|_{\mathcal D_2} \leq \left(\tilde A + \frac{1}{2} \|\Psi_0\|_{\mathcal D_0}\right)\mathfrak a_h^{-1} \cdot h$, which is of the type we wanted.}
Thus assuming $p_0 \leq ah^{-1/2}$ and \b{using Claim A$_{2}$,} our tentative estimate becomes
$$
\b{\mathfrak a_h \|h\psi_{\bullet}\|_{\mathcal D_{2}}}
\leq (C_{r} + C_{s})\|hL_{h}^{-1}\Psi_{0}\|_{\mathcal D_{3}} + Ah
$$
where $A$ has been modified from $\tilde{A}$ by the addition of $(c_{r} + c_{s})$.

The conclusion is now inevitable.  From \textcolor{blue}{Claim A$_{3}$} we have that if the  first term on the right of the previous display 
exceeds $\|h\psi_{\star}\|_{\mathcal D_2}$ then both terms (and hence all terms) \b{are bounded by a $\mathfrak D_0$--dependent constant times $h$}; otherwise, 
this term is bounded by $\|h\psi_{\star}\|_{\mathcal D_{2}}$ and we conclude that 
\b{$$\|h\psi_\bullet\|_{\mathcal D_2} \leq \mathfrak a_h^{-1} \left(\|h\psi_\star\|_{\mathcal D_2} + Ah\right) $$}
and again we have an inequality of the type we wanted when $\|h\psi_{\star}\|_{\mathcal D_{2}}$ is relatively large.  
\end{proof}

To summarize, so far \b{we have the following results for one timestep of the iteration:}

\begin{corollary} \label{CEB}
Consider Eq.~\eqref{CVDA} with all elements of $\mathfrak D_0$ finite.  Then there is some $h_0 = h_0(\mathfrak D_0)$ such that for all $h < h_0$:

i) There is a classical \b{(i.e., $\mathcal D_2$, which implies the usual $\mathscr C^2$)} solution $N_1 =\mbox{e}^{\Psi_1}$ which is bounded below; 

ii) $\|\Psi_1\|_{\mathcal D_1} - \|\Psi_0\|_{\mathcal D_1}$ and $\|N_1 - N_0\|_{\infty}$ are bounded from above by a constant depending only on the elements of $\mathfrak D_0$ times $h$.
\end{corollary}

\begin{proof}
\b{Most of this follows from the above.  For i), the existence of a solution $\Psi_1 = N_0 + h \psi$ is given by Proposition \ref{XTY} and the solution is classical by Proposition \ref{DBFS};  the lower bound on $N_1$ certainly follows since $\psi$ has bounded $\mathcal D_0$--norm by Proposition \ref{XTY}. As for ii), we have  by Propositions \ref{XTY} and \ref{DBFS} and Cauchy--Schwarz that $\|\Psi_1\|_{\mathcal D_1} - \|\Psi_0\|_{\mathcal D_1} \leq \|h\psi\|_{\mathcal D_1} \leq \mathfrak b_0 \mathfrak b_2 h$; finally, for the $L^\infty$--bound we have $|N_1 - N_0| \leq N_0 [ e^{h \mathfrak b_0} - 1]$.}
\end{proof}

\b{Under the \emph{assumption} that the discretization process persists for macroscopic times (i.e., the order of $h^{-1}$ iterations) we will show that the bounds derived so far also allow us to establish the needed convergence to the continuum result of Theorem \ref{FVJJ}. The questions which pertain to the long time survival of the iteration process will be postponed till the next subsection.}   

\noindent \textit{Proof of Theorem \ref{FVJJ}, item (A).}  As before, we let $\Psi_{t}$ denote the limiting quantity which satisfies the appropriate version
of Eq.~\eqref{CVDA}).  
%
\b{We will first establish} uniform convergence in the $\mathcal D_{0}$--norm, which may be expressed via
\begin{equation}\label{KRKE}
\lim_{h\to 0} \sup_{t\in [0,T]}\|\Psi_{t} - \Psi_{t}^{[h]}\|_{\mathcal D_{0}}
= 0.
\end{equation}
First, let $h_{j}$ denote a sequence tending to zero (always below $h_{T}$) where it may be envisioned that in the above, the superior $h\to 0$ limit is achieved.  Let $t_{j}$ denote an integer \b{(multiple of $h$)} time closest to the time where the $h = h_{j}$ supremum in the above display is to be found.  

\b{It follows from the weak convergence of $\Psi_{t}^{[h]}$ to $\Psi_t$ established in the proof of Theorem \ref{FVJJ} (in Section \ref{SecConvergence}) that for all $t$ and $q$, $\hat\Psi_t^{[h]}(q) = \int_{\mathbb T_L^d} \Psi_t^{[h]}(x) e^{iqx}~dx \longrightarrow \int_{\mathbb T_L^d} \Psi_t(x) e^{iqx}~dx = \hat \Psi_t(q)$, i.e.,}
$$
\hat{\Psi}_{t}^{[h]}(q)
\to 
\hat{\Psi}_{t}(q).
$$
\b{Now the uniform (in $h$, for $h$ small) bound on $\|\Psi_{t}^{[h]}\|_{\mathcal D_2}$ from Proposition \ref{DBFS} gives a so--called tightness condition: Indeed, if $\sum_k |k|^2  |\Psi_t^{[h]}| < C$ then we have that $\sum_{k > k_0} |\Psi_t^{[h]}| \leq C/k_0^2$ which can be made arbitrarily small by choosing $k_0$ sufficiently large. Then by the above convergence of modes the truncated sum of differences $\sum_{k \leq k_0} |\Psi_t^{[h]} - \Psi_t|$ tends to zero as $h$ tends to zero.  We can therefore conclude that}
$$
\limsup_{j\to\infty}\|\Psi_{t_{j}} - \Psi_{t_{j}}^{[h_{j}]}\|_{\mathcal D_{0}} = 0.
$$

Next we let $t^{\dagger}$ = $\lim_{j\to \infty} t_{j}$,
and then the \b{limit in Eq.~\eqref{KRKE}} is seen to be zero by \b{an} application of the triangle inequality:
\[ \begin{split}
\lim_{j\to\infty}\|\Psi_{t^{\dagger}} - \Psi_{t_{j}}\|_{\mathcal D_{0}}
\hspace{.125cm}
=
\hspace{.125cm}
\lim_{j\to\infty}\|\Psi_{t^\dagger} - \Psi^{[h_{j}]}_{t^{\dagger}}\|_{\mathcal D_{0}}
\hspace{.125cm}
&+ \hspace{.125cm}
\lim_{j\to\infty}\|\Psi^{[h_{j}]}_{t_{j}} - \Psi^{[h_{j}]}_{t^{\dagger}}\|_{\mathcal D_{0}}
\hspace{.125cm}\\
&\hspace{2cm}+ \hspace{.125cm}
\lim_{j \to \infty} \|\Psi_{t_j}^{[h_j]} - \Psi_{t_j} \|_{\mathcal D_0},
\end{split}\]
\b{where the middle term is zero since from Corollary \ref{CEB}, ii) we have that the term of interest is bounded by $\left[\frac{t_j - t^\dagger}{h_j}\right] \cdot \mathfrak b_0 h_j$ (here $[x]$ is the least integer so that $[x] \geq x$).}

\b{Finally, since Proposition \ref{DBFS} gives uniform boundedness (in $h$, for $h$ small) of $\|\Psi_t^{[h]}\|_{\mathcal D_2}$, together with the above uniform $\mathcal D_0$--convergence result, the strong $\mathcal D_1$--convergence follows from the Cauchy--Schwarz inequality: 
$$\|\Psi_t^{[h]} - \Psi_t\|_{\mathcal D_1} \leq \|\Psi_t^{[h]} - \Psi_t\|_{\mathcal D_0}^{1/2} \cdot \|\Psi_t^{[h]} - \Psi_t\|_{\mathcal D_2}^{1/2}.$$}
\qed
\subsection{Viability of Iterations}

For $h$ sufficiently small, we may envision a few runs of the process.  \b{After one step, we will have} an updated version of $\mathfrak D_0$ in which some of the parameters, \b{i.e., $\|\Psi_1\|_{\mathcal D_0}$ and $\|\Psi_1\|_{\mathcal D_2}$, have changed; we call these the \emph{mutable parameters}.}  And, if $h$ is still small enough this will allow (even according to the bounds) further iterations of the process.  In any case if $k$ iterations of the process are allowed, let us denote by
\b{$\mathfrak D_{t}^{[h]}$} the current values of the parameters where $h\leq(k-1)h \leq t < kh$:
\[ \mathfrak D_t^{[h]} = \{ \|\Psi_{k-1}\|_{\mathcal D_0}, \|\Psi_{k-1}\|_{\mathcal D_2},  v_0, \dots, v_4, \|W\|_{\mathcal D_2}\}.\] 
\b{(This definition is consistent with denoting the original
$\mathfrak D$ we started with by $\mathfrak D_{0}$, as we have done in the previous subsections.)}  Here, let us introduce the notion of \textit{viability}:

\begin{definition}\label{PVV}
Let $\mathfrak D_{t}^{[h]}$ be defined as above and $h$ considered \emph{fixed}.  Then the process is deemed to be \textit{viable} for $h$
if, on the basis of the \emph{bounds} derived in the preceding two subsections \b{(not necessarily the actual values)}

(a) $\mathfrak D_t^{[h]}$ permits an iteration of the process; further, still on the basis of these estimates for elements of $\mathfrak D_{t+h}^{[h]}$ \b{(i.e., considering the \emph{estimates} for $\mathfrak D_t^{[h]}$ to be playing the role of $\mathfrak D_0$ and used to estimate the  elements of $\mathfrak D_{t+h}^{[h]}$)} 

(b) an \emph{additional} iteration is possible.
\end{definition}

\medskip
It is noted that \b{by Corollary \ref{CEB}, given \emph{any} $\mathfrak D$ with finite elements, the process is viable \emph{if} $h$ is sufficiently small.}  However, this is far from what is needed since we must consider many iterations of the process \b{at \emph{fixed} $h$}.  The following represents a midway goal of this appendix:
\begin{proposition}
\label{PYBV}
Consider the setup encoded in Eq.~\eqref{CVDA} as has been described.  Then there exists a strictly positive 
$\textsc{t} = \textsc{t}(\mathfrak D_{0})$ such that for all $h$ sufficiently small, 
the process is viable up till time $\textsc{t}$, i.e.,
the elements of 
$\mathfrak D_{\textsc{t}}^{[h]}$ allow for continued iteration of the process.
\end{proposition}

It is reemphasized that whenever $h$ is small enough so that the above statement holds, the conclusion pertains to the order of $\textsc{t}h^{-1}$ iterations of the process.

\begin{proof}
Let $H_0 > 0$ denote a number which is larger than all the mutable parameters in
$\mathfrak D_{0}$ -- and indeed might be regarded as considerably larger.  After an iteration of the process,
assuming $h$ is small enough to allow such,
 the mutable parameters will in all likelihood have changed.  So let us thus define $\mathbb H(H,h)$ so that $h\mathbb H$ is the maximum upward change of these mutable parameters, according to the bounds derived \b{in Propositions \ref{XTY} and \ref{DBFS}}, were they all equal to $H$ in the first place.   
Due to monotonicity based on inefficiency,  it is clear that if in
$\mathfrak D_{t}^{[h]}$ all mutable parameters are less than or equal to $H$, then in
$\mathfrak D_{t + h}^{[h]}$ none of them exceeds 
$H + h\mathbb H(H,h)$.  Moreover, it is clear that the $h \to 0$ limit of $\mathbb H(H, h)$ is finite, i.e., 
$\mathbb H(H,h)$ may be considered to be uniformly bounded \b{in $h$}.  

\b{As an explicit example, suppose it were the case that $\|\Psi_0\|_{\mathcal D_2} > \|\Psi_0\|_{\mathcal D_0}$, then we set $H_0 = \|\Psi_0\|_{\mathcal D_2}$ and perform the estimates in Propositions \ref{XTY} and \ref{DBFS} with $H_0$ playing the role of \emph{both} $\|\Psi_0\|_{\mathcal D_0}$ and $\|\Psi_0\|_{\mathcal D_2}$, yielding bounds $H_0 + h\mathbb H_0^{[0]}, H_0 + h\mathbb H_0^{[2]} $, respectively.  Let us suppose e.g., that $\mathbb H_0^{[0]} > \mathbb H_0^{[2]}$, then we would set $\mathbb H(H_0, h) = \mathbb H_0^{[0]}$.  In this way we arrive at $H_1 = H_0 + h\mathbb H(H_0, h)$. }

Now consider \b{$\mathbb H(2H_0, \cdot)$} and let $h_{2}^{\dagger}$ be small enough so that for all
$h \leq h_{2}^{\dagger}$, provided all mutable parameters in $\mathfrak D$ do not exceed $2H_0$, the process is still viable.  I.e., informally, if $2H_0$ is ``small enough for $h_{2}^{\dagger}$, then so is $2H_0 + h_{2}^{\dagger}\mathbb H(2H_0, h_{2}^{\dagger})$''.
Finally, let 
$$
\mathbb H_{2}^{\dagger}  = \sup_{h < h_{2}^{\dagger}}\mathbb H(2H_0, h).
$$

The following is now clear:  Starting at $\mathfrak D_{0}$ -- with all mutable parameters less than $H_0$, and $h \leq h_{2}^{\dagger}$, we may certainly \b{iterate the above described process to yield $H_2, H_3$, etc.,} until -- according to the derived bounds -- one of our mutable parameters reach $2H$, \b{i.e., some $m$ such that $H_m \leq 2H, H_{m+1} > 2H$.}  This implies there will be at least $m$ permitted iterations of the process where $m$ is the largest integer smaller than
$h^{-1}H/\mathbb H_{2}^{\dagger}$, i.e., $\textsc{t} \gtrsim H/\mathbb H_{2}^{\dagger}$.
\end{proof}
%
It might be envisioned that going to smaller and smaller time steps will allow for indefinite extension of the simulation times.  While this is true, and the subject of our next proposition, this cannot be proved on the basis of the bounds on the process that have so far been derived. Indeed, on adhering to the above, in the $h\to 0$ limit we would anticipate the bound on $H$ \b{(which is now considered to be a function of time)} provided by
$$
\frac{dH}{dt}  = \mathbb H(H, 0).
$$
However, such an equation \b{may very well} diverge in finite time as indeed would a ``more accurate'' equation/bound involving  all mutable parameters separately.  The needed additional ingredient 
is provided by the convergence to and the properties of the \emph{limiting} Eq.~\eqref{XYZ}.

Since both $h$ and times will be varying in the next proposition, we shall indicate the former by bracketed superscripts and the latter by subscripts indicating macroscopic times.  Thus, e.g., 
$\Psi_{t}^{[h]}$ denotes the \b{(piecewise constant)} function ``$\Psi$'' obtained after $k$ iterations of the process with time step $h$ for time $t$ if we have $hk \leq t < h(k+1)$.

\begin{proposition}
\label{TEE}
Let $ T > 0$ be arbitrary.  Then there exists $h_{ T} > 0$ such that for all 
$h \leq h_{ T}$, the process described by Eq.~\eqref{CVDA} survives at least up till time $T$, \b{i.e., 
$$\sup_{0 < h < h_T} \max_{t \in [0, T]} \|\mathfrak D_t^{[h]}\|_{\infty} < \infty.$$  Thus we may perform} the order of
$h^{-1}T$ iterations.
\end{proposition}
\begin{proof}
The proof relies on the fact that the continuous time equation, Eq.~\eqref{GCD} lasts indefinitely and enjoys smoothing properties.  In particular, at positive times the functions $\Psi$ etc., have their $n\text{th}$ derivatives in $L^{1}(\mathbb T_{L}^{d})$ for all $n$ (\cite{esk}), hence all the $\mathcal D_{k}$--norms are finite.  Of course for the purposes of this proof, we are only concerned with the $\mathcal D_{0}$ through $\mathcal D_{2}$ norms and their roles as elements of 
$\mathfrak D$.  

Consider $T > 0$, our fixed macroscopic time.  Let  $0 < t_{0} < \textsc{t}(\mathfrak D_0)$ (as in Proposition \ref{PYBV}) and $t_{1} > T$.  We define
\b{$\alpha$} to be the supremum of the continuous time versions of the relevant
$\mathcal D_{0}$, $\mathcal D_{1}$ and $\mathcal D_{2}$ norms. \b{If the statement of the proposition were false for the time $T$, then there exists a sequence $h_k \rightarrow 0$ and a sequence of times $(t_k)\subseteq [0, T]$ such that $\lim_{k\rightarrow \infty} \|\mathfrak D_{t_k}^{[h_k]}\|_\infty = \infty$.}  Let \b{$H > 2\alpha$} be a quantity like that employed in the proof of Proposition \ref{PYBV} and let us define the times 
$\tau_{H}^{\dag}(h)$ and $\tau_{2H}^{\dag\hspace{-.05 cm}\dag}(h)$:  

 \b{$\bullet$   \hspace{.025 cm} $\tau_{2H}^{\dag\hspace{-.05 cm}\dag}(h)$ is such that at this time, the maximal element of the appropriate $\mathfrak D$, \b{i.e., $\mathfrak D_{\tau_{2H}^{\dag\hspace{-.05cm}\dag}(h)}^{[h]}$,} is less than $2H$, but one time step later, some element of $\mathfrak D_{\tau_{2H}^{\dag\hspace{-.05cm}\dag}(h) + h}^{[h]}$ exceeds $2H$ for the first time in the process;}
 
 \b{$\bullet$  \hspace{.025 cm} $\tau_{H}^{\dag}(h)$ is such that at this time, the maximal element of $\mathfrak D_{\tau_{H}^\dag(h)}^{[h]}$ exceeds $H$, however one time step prior, all the mutable elements of $\mathfrak D_{\tau_H^\dag - h}^{[h]}$ were below $H$.  }
 
 \b{Altogether we certainly have $\tau_H^\dag(h) \leq \tau_{2H}^{\dag\hspace{-.05cm}\dag}(h) + h$; it can further be demonstrated that in fact $\tau_{2H}^{\dagger\dagger}(h) - \tau_H^\dagger(h)$ is of order unity: Returning to the context of the proof of Proposition \ref{PYBV}, let us say that we have $h$ sufficiently small so that uniformly in $h$ $\mathbb H(H, h) \leq \mathbb H_1$ and $\mathbb H(2H, h) \leq \mathbb H_2$ (so that $\mathbb H_1 \leq \mathbb H_2$ by monotonicity of the function $\mathbb H(H, h)$ in $H$). Now we certainly have $\|\mathfrak D_{\tau_H^\dagger}\|_\infty \leq H + h\mathbb H_1$ and $\|\mathfrak D_{\tau_{2H}^{\dagger\dagger}}\|_\infty \geq 2H - h\mathbb H_2$.  Thus if we had chosen $h$ smaller (if necessary) so that $3h\mathbb H_2 \ll H$ then $\|\mathfrak D_{\tau_H^\dagger(h) + h}\|_\infty \leq H + h\mathbb H_1 + h\mathbb H(H + h\mathbb H_1, h) \ll 2H - h\mathbb H_2$ and so the conclusion follows.}
  
 The assumed falsehood of the statement of this proposition implies that these times exist, are well defined and satisfy
$$
\limsup_{h\to 0} \tau_{2H}^{\dag\hspace{-.05 cm}\dag} \leq T.
$$
Thus we have a family of compact intervals 
$[\tau_{H}^{\dag} \hspace{.025 cm}(h), \tau_{2H}^{\dag\hspace{-.05 cm}\dag}(h)]$
which, \b{as established above,} are non--empty and of size uniformly bounded below. Let us start by restricting to a subsequence of $h$'s -- which we will not adorn with further labels -- in which the intersection of these subsequent intervals contains an interval to which we will restrict our attention.  Now, it is emphasized, the totality of all iterations in the subsequence under consideration
is countable.  

In the intersection of the above mentioned regions, the iteration process is certainly viable and hence
the convergence result of Theorem \ref{FVJJ}, item (A) may be applied.  \b{I.e.,  here we have strong convergence to the continuum equation: }
$$\lim_{h\rightarrow 0} \sup_{t \in [0, T]} \|\Psi_t - \Psi_t^{[h]}\|_{\mathcal D_0} = 0 \mbox{~~~and~~~} \lim_{h \rightarrow 0} \|\Psi_t - \Psi_t^{[h]}\|_{\mathcal D_1} = 0.$$
Also $\|\Psi_{t}^{[h]}\|_{\mathcal D_{1}}$ and 
$\|\Psi_{t}^{[h]}\|_{\mathcal D_{2}}$ are bounded so these converge weakly 
(along any $t_{j}\to t, h_{j}\to 0$ subsequence) to their continuum values.  Thus, further restricting the subsequence of $h$'s if necessary, 
$\|\Psi_{t}^{[h]}\|_{\mathcal D_{0}}$, $\|\Psi_{t}^{[h]}\|_{\mathcal D_{1}} < \alpha$ for all $t$ and $h$.  
However, since something in the $\mathfrak D_{t}^{[h]}$'s must be greater than $H$ (since $\|\mathfrak D_t^{[h]}\|_\infty \rightarrow \infty$ by assumption) it is evident that we have
$\|\Psi_{t}^{[h]} \|_{\mathcal D_{2}} >H$.
This implies that these objects are \emph{not} converging strongly in $\mathcal D_{2}$.    

Let us summarize the strategy for the remainder of this proof.  We will show that the purported circumstances imply that among the iterative corrections, the dominant term, by far, is $\psi_{\star}$ 
(c.f., Eq.~\eqref{PU9Y} and Eq.~\eqref{ABFG}).  Thence $\Psi_{t + h}^{[h]}$ is given, in essence, by
$(\Psi^{[h]}_{t + h})_{\star}$ which, we remind the reader, enjoys a reduction in \textit{all} $k\neq 0$
Fourier modes.  So, in particular, we will show
$\|\Psi_{t + h}^{[h]}\|_{\mathcal D_{2}} < \|\Psi_{t}^{[h]}\|_{\mathcal D_{2}}$,
indicating that the time $\tau_{2H}^{\dag\hspace{-.05 cm}\dag}(h)$ is never reached, effecting a contradiction.  \b{Much of reasoning here will be similar to the estimations in Proposition \ref{DBFS} so we shall be succinct.}
In what follows, we shall make statements which, properly speaking hold for all but a finite number of $h$'s and time intervals.  We shall abbreviate by saying ``for all'', automatically going to subsequences if necessary.

Our first claim is that $\|L_{h}^{-1}(\nabla^{2}\Psi_{t}^{[h]})\|_{\mathcal D_{2}}$
(corresponding to the $\psi_{\star}$--term, c.f., Eq.~\eqref{PU9Y})
 is, in essence, indefinitely large.  To this end, let $\mathcal Q$ denote a fixed large number
the necessary size of which will be specified eventually.  If, we suppose, that for infinitely 
many $h$'s, \b{the sum for $\|\Psi_t^{[h]}\|_{\mathcal D_2}$ truncated at $\mathcal Q$ satisfies}
$$
\sum_{|q| \leq \mathcal Q}|q^2\hat{\Psi}_{t}^{[h]}(q)|  > \frac{1}{2}H
$$
then \b{since $H > 2\alpha$} this would imply that any limit of
$\Psi_{t}^{[h]}$ would have $\mathcal D_{2}$--norm in excess of $\alpha$.   
Thus we have, without loss of generality \b{it must be the ``tail'' which diverges, i.e., } for all $h$ and $t$,
\begin{equation}
\label{KLT}
\sum_{|q| > \mathcal Q}|q^2\hat{\Psi}_{t}^{[h]}(q)|  \geq \frac{1}{2}H.
\end{equation}

Now for all $h$ sufficiently small and $\mathcal Q$ fixed, it is clear that
$k^{2}/[1 + hk^{2}] > (\frac{1}{2}\mathcal Q)^{2}$
whenever $k > \mathcal Q$ and thus $\|L_{h}^{-1}(\nabla^{2}\Psi_{t}^{[h]})\|_{\mathcal D_{2}}
\gtrsim H\mathcal Q^{2}$: Indeed, 
\begin{equation}
\label{DDA}
\|L_{h}^{-1}(\nabla^{2}\Psi_{t}^{[h]})\|_{\mathcal D_{2}}
\geq
\sum_{|q| > \mathcal Q}\frac{q^{2}}{1 + hq^{2}} \cdot |q^2\hat{\Psi}_{t}^{[h]}(q)|
\geq
\frac{1}{8}\mathcal Q^{2}H.
\end{equation}
This is deemed to be larger than \b{(the bounds on)} all peripheral terms \b{which consist of all terms in $\psi_\bullet$ except the $r_\bullet$ and $s_\bullet$ terms and also the $c_r$ and $c_s$ terms from Claim $A_2$; all these terms are at most multiples of $H$.}  The only possible difficulties concern the terms $r_{\bullet}$ and $s_{\bullet}$.  According to one scenario, namely \b{$p_{0} > ah^{-\frac{1}{2}}$} (c.f., Claim A$_{1}$ and noting the factor of 2)
these terms could only account for half of the term 
$\|L_{h}^{-1}(\nabla^{2}\Psi_{t}^{[h]})\|_{\mathcal D_{2}}$ and so \b{(sending $\mathcal Q^2$ to $2\mathcal Q^2$ if necessary)} the remainder is more than sufficient 
for all else.  

Otherwise, \b{when $p_0 \leq a h^{-\frac{1}{2}}$} it is recalled \b{(see Claim $A_2$)}
the added \b{$r_{\bullet}$ and $s_{\bullet}$} terms have $\mathcal D_{2}$--norms bounded by the \b{$c_s$ and $c_r$ terms plus the term}
$$
(C_{r} + C_{s})\sum_{q}\frac{1}{1 + hq^{2}} \cdot |q^3\hat{\Psi}_{t}^{[h]}(q)|.
$$
\b{It remains to bound the terms in the last display.}
As for the range $|q| \leq q_{0}$ (where we recall \b{from the proof of Claim $A_3$ that }$q_{0} = 2(C_{r} + C_{s})$) we may bound the corresponding contribution of the above by $8(C_{r} + C_{s})^{3}\|\Psi_{t}^{[h]}\|_{\mathcal D_{1}}$.  \b{By proper choice of $\mathcal Q$, this can be made to be} negligibly small compared 
to $\mathcal Q^{2}H$.
In the range $q_{0} < |q| \leq \mathcal Q$, the terms contributing to 
$\|L_{h}^{-1}(\nabla^{2}\Psi_{t}^{[h]})\|_{\mathcal D_{2}}$ dominate their counterparts in the above display and so we may ignore these differing contributions.  

This leaves us with 
$|q| > \mathcal Q$ where it may be asserted that
$$
\frac{q^{4} - (C_{r} + C_{s})q^{3}}{1 + hq^{2}} \geq
q^{2}\left[(\frac{1}{2}\mathcal Q)^{2} - (C_{r} + C_{s})\mathcal Q\right].
$$
\b{The previous expression comes directly from Eq.~\eqref{CMRI}: By Eq.~\eqref{DDA}, it cannot be the case that \emph{everything} is bounded by multiples of $H$, so in the context of Claim $A_3$, we are in the case where $\|h\psi_\star\|_{\mathcal D_2} > (C_r + C_2) \|h L_h^{-1} \Psi_0\|_{\mathcal D_3}$.}
This leaves us an overall \textit{excess} at least as large as
$\left[(\frac{1}{2}\mathcal Q)^{2} - (C_{r} + C_{s})\mathcal Q\right]\times \frac{1}{2}H$ \b{(the last expression comes from the previous display and Eq.~\eqref{KLT}}).  It is thus seen that for $\mathcal Q$ chosen to be large enough, the increment for
$\|\Psi_{t}^{[h]}\|_{\mathcal D_{2}}$ on each step of the iteration is negative \b{(we again remind the reader of Eq.~\eqref{ABFG} and the discussions immediately following)}
\b{and so indeed $\|\psi_{t+h}^{[h]}\|_{\mathcal D_2} < \|\psi_t^{[h]}\|_{\mathcal D_2}$}.  
\end{proof}

We can now extend Corollary \ref{CEB} to arbitrary macroscopic times:

\begin{corollary} \label{CEB2}
Consider Eq.~\eqref{CVDA} with all elements of $\mathfrak D_0$ finite and let $T > 0$.  Then there is some $h_T = h_T(\mathfrak D_0)$ such that for all $h < h_T$ and $t < T$:

i) There is a classical solution $N_t^{[h]} =\mbox{e}^{\Psi_t^{[h]}}$ which is bounded below; 

\textcolor{blue}{ii) $\|\Psi_{t+h} - \Psi_t\|_{\mathcal D_0} \leq \mathfrak{b}_0 h$;}

\textcolor{blue}{iii) $\|N_{t+h}^{[h]} - N_t\|_{\infty}$ and $\|\Psi_{t+h}^{[h]}\|_{\mathcal D_1} - \|\Psi_t^{[h]}\|_{\mathcal D_1}$ are bounded from above by a constant depending only on the elements of $\mathfrak D_t$ times $h$.}

\noindent In the above, the notation $N_t^{[h]}$ etc., is as in the proof of Theorem \ref{FVJJ}.
\end{corollary}

\begin{proof}
\b{With the results of Proposition \ref{TEE} etc., in hand, the proof of i) and ii) follow \emph{mutatis mutantis} from the proof of Corollary \ref{CEB} whereas item iii) is simply Proposition \ref{XTY} stated for arbitrary $t < T$. }
\end{proof}

\section{Appendix B}
In this appendix -- which is not \textit{essential} for this work but is requisite for completeness -- we
present the basic properties of the distance function on $\mathcal B\times \mathcal B$ (particularly that it actually is a distance).
\b{Our result, concerning the realization of the minimization program defining the distance, is in the spirit of \cite{BB}:}

\begin{proposition}
\label{XVAB}
For $ N_{0}, N_{1} \in \mathcal B$, consider $\mathbb D^{2}(N_{0}, N_{1})$ as given in 
Eq.~\eqref{BDF}.  Then the infimum in this equation is achieved by minimizing among velocity fields that are derived from potentials.
\end{proposition}

\begin{proof}
Let $N_{t}$ denote a path in $\mathcal B$ from $N_{0}$ to $N_{1}$ as described in Eq.~\eqref{BDF} which we suppose is driven by fields \b{$(V, Q) \in \mathscr V({N_{0,}N_{1}})$:}
\b{$$
\frac{\partial N_{t}}{\partial t} + \nabla \cdot (N_{t}V) = -\Omega_{N_{t}}Q.
$$}
Now let $\phi$ denote a velocity potential which also produces the path $N_{t}$ (as in the derivations following Eq.~\eqref{BDF}):
$$
\frac{\partial N_{t}}{\partial t}  =  \nabla \cdot (N_{t}\nabla \phi) - \Omega_{N_{t}}\phi.
$$
Multiplying both of the above by $-\phi$ and integrating by parts, we have, for a.e. \hspace{-.1 cm}$t$, 
\begin{equation}\label{PUNN}\b{
\langle \hspace{-.1cm} \langle  (-\nabla \phi, \phi), (-\nabla \phi, \phi) \rangle  \hspace{-.1cm}  \rangle_{N_{t}}
=
\langle \hspace{-.1cm} \langle  (- \nabla \phi, \phi), (V, Q) \rangle  \hspace{-.1cm}  \rangle_{N_{t}}.
}\end{equation}
I.e., the difference between $(V, Q)$ and $(-\nabla\phi, \phi)$ is orthogonal to $(-\nabla \phi, \phi)$.  Now
\[ \b{\langle \hspace{-.1cm} \langle  (V + \nabla \phi, Q -\phi),
(V + \nabla \phi, Q-\phi) \rangle  \hspace{-.1cm}  \rangle_{N_{t}} \geq 0.}\]
Expanding the above and using Eq.~\eqref{PUNN}, we conclude 
\[ \b{\langle \hspace{-.1cm} \langle  (V, Q), (V, Q) \rangle  \hspace{-.1cm}  \rangle_{N_{t}} \geq \langle \hspace{-.1cm} \langle  (-\nabla \phi, \phi), (-\nabla \phi, \phi) \rangle  \hspace{-.1cm}  \rangle_{N_{t}}. } \]
Thence, at least from the perspective of a minimization program, we \b{may restrict attention to} gradient fields.  
\end{proof}

Here we establish the so--called indiscernible property of 
$\mathbb D(\cdot,\cdot)$ as stated below.  In what follows, we will actually make use of the finite range assumption on $W(\cdot)$.

\begin{proposition}
\label{ADA}
Let $N_{0}, N_{1} \in \mathcal B$ with 
$N_{0}\neq N_{1}$.  Then
$$
\mathbb D^{2}(N_{0}, N_{1})  \neq 0.
$$
\end{proposition}

\begin{proof}
Assuming $\mathbb D^{2}(N_{0}, N_{1}) = 0$, 
let $N_{t}^{(k)}$ be a minimizing sequence of paths in $\mathcal B$ 
connecting $N_{0}$ and $N_{1}$.  We denote by $\Psi^{(k)}_{t}$
the associated driving potentials.  By our assumption, it is the case that
$\varepsilon_{k}(t)$ defined by
\begin{equation}\label{EIRKK}
\varepsilon_{k}^{2}(t)  : = - \int_{\mathbb T_L^d}\Psi^{(k)}_{t}\frac{\partial N_{t}^{(k)}}{\partial t}~dx
= \langle \hspace{-.1cm} \langle
\nabla \Psi^{(k)}_{t}, \nabla \Psi^{(k)}_{t}
\rangle  \hspace{-.1cm}  \rangle_{N_{t}}
\end{equation}
satisfies
$$
0 = \lim_{k\to\infty}\int_{0}^{1}\varepsilon^{2}_{k}(t)~dt.
$$
\b{The idea is then to estimate the mass evolution of $N_t^{(k)}$ using the equation to eventually arrive at the conclusion that $N_0 = N_1$.}

\b{We start by defining a ``localized'' mass of $N$.  }For $x_{0}\in \mathbb T_{L}^{d}$, let $B_{a}(x_{0})$ denote the ball of radius $a$ centered at $x_{0}$ where $a$ denotes the interaction radius of $W$.    
Let $\varphi(x)$ denote any positive $\mathscr C^{2}$ function which is identically one on $B_{a}(0)$ and decreases  to zero  outside, specifically in 
$B_{2a}(0)\setminus B_{a}(0)$.  
For brevity, we use
$\varphi_{x_{0}}(x) := \varphi (x - x_{0})$.  For $N\in \mathcal B$ we will write
$$
p_{N}(x_{0})  :=  \int_{\mathbb T_{L}^{d}}\varphi_{x_{0}}N~dx
$$
which, it is noted, is an upper bound on 
the $N$--measure of 
$B_{a}(x_{0})$ (and a lower bound on the $N$--measure of
$B_{2a}(x_{0})$).  
Moreover, it is noted that 
$p_{N}(x)$ is a continuous function of $x$.

For $x\in \mathbb T_{L}^{d}$, $t\in [0,1]$ and $k$ an integer let us abbreviate
$p_{N_{t}^{(k)}}(x)$ by $p_{t, k}(x)$.  It is observed that (for fixed $k$) 
$p_{t,k}(x)$ is a continuous function on $[0,1] \times \mathbb T_{L}^{d}$.
Indeed, for fixed $x_{0}$ we can estimate the evolution of 
$p_{t,k}(x_{0})$.
We have
$$
-\frac{d }{dt}p_{t,k}(x_{0})  =  
\langle \hspace {-.1cm} \langle \nabla \Psi_{t}^{(k)}, \nabla \varphi_{x_{0}}
\rangle  \hspace{-.1cm}  \rangle_{N_{t}^{(k)}}
$$
so that by Eq.~\eqref{EIRKK} and Cauchy--Schwarz, 
\begin{equation}\label{DKRRR}
\left |
\frac{d }{dt}p_{t,k}(x_{0})
\right |
 \leq \varepsilon_{k}(t)
\langle \hspace{-.1cm}\langle \nabla \varphi_{x_{0}} , \nabla \varphi_{x_{0}}
\rangle\hspace{-.1cm}\rangle_{N_{t}^{(k)}}^{\frac{1}{2}}
\end{equation}
for a.e.~$t$.  We now examine 
$\langle\hspace{-.1cm}\langle \nabla \varphi_{x_{0}} , \nabla \varphi_{x_{0}}
\rangle\hspace{-.1cm}\rangle_{N_{t}^{(k)}} = \int_{\mathbb T_L^d} N_t^{(k)} |\nabla \varphi_{x_0}|^2 + \Omega_{N_t^{(k)}} \varphi_{x_0}^2~dx$.  As for the gradient term, let us write  
$|\nabla \varphi |^2 \leq g_{\varphi} \varphi$ for some constant $g_\varphi$
and so
$$
\int_{\mathbb T_L^d}
N_{t}^{(k)}|\nabla \varphi_{x_{0}}|^{2}
~dx  \leq  g_\varphi\cdot p_{t,k}.
$$

For the second term, we first claim that 
\[ \Omega_N \leq \text{e}^{\frac{1}{2}|\mu - w_N|} \cdot \frac{1}{2} (1 + N).\]
Indeed, writing $\Omega_N = N^{\frac{1}{2}} \sinh (\frac{1}{2} \Phi_N)/\frac{1}{2} \Phi_N \leq N^{\frac{1}{2}} \cosh \frac{1}{2} \Phi_N$, the result follows immediately.  Also, for $N$ fixed,  we have
$$
|w_N(x_{0})| \leq w_0\cdot p_{N}(x_{0}),
$$
with $w_0$ being the $\mathscr C^0$--norm of $W$.  Indeed, 
$$
|w_N(x_{0})|  =  \left|\int_{\mathbb T_L^d} W(x_{0} - y)N_{t}(y) ~dy\right|
\leq \int_{\mathbb T_L^d} |W(x_{0} - y)| N_{t}(y)~dy  \leq w_{0}N_{t}\left[B_{a}(x_{0})\right]
$$
and we conclude by recalling that $p_N(x_0)$ is an upper bound of $N_t\left[B_a(x_0)\right]$.  The previous two observations then yield the preliminary estimate
$$
\int_{\mathbb T_L^d} \Omega_{N_{t}^{(k)}}\varphi_{x_{0}}^{2}~dx
\leq
\frac{1}{2}\int_{\mathbb T_L^d} \text{e}^{\frac{1}{2}|\mu - w_N|} \cdot 
(1+N_t^{(k)})\varphi_{x_{0}}^{2} ~dx
\leq
\frac{1}{2}\text{e}^{\frac{1}{2}\mu}\int_{\mathbb T_L^d}\text{e}^{\frac{1}{2}w_{0} p_{t,k}} \cdot (1+N_t^{(k)}) \varphi_{x_{0}}^{2}
~dx.
$$

Now, consider $I_{0}$ defined by
$$
I_{0} := \max_{x\in\mathbb T_{L}^{d}} p_{k, 0}  \equiv 
\max_{x\in\mathbb T_{L}^{d}}\int_{\mathbb T_L^d}N_0 \varphi_{x}~dx
$$
which is manifestly independent of $k$. It is clear that for any given $x$ if we define 
$$
t_k^{\sharp}(x)  := \sup\{t\in [0,1]  \mid p_{k,t}(x) < 2I_{0} \}
$$
then $t_k^{\sharp}(x) > 0$ (indeed, we have the above explicit bound on $|\frac{d}{dt} p_{t, k}|$).  Moreover, it can easily be established 
using the continuity of $p_{k,\cdot}(\cdot)$ that
$$
t^{\flat}_{k}  :=  \inf_{x\in\mathbb T_{L}^{d}} t_k^{\sharp}(x)
$$
is strictly positive.  But \textit{a priori} $t^{\flat}_{k}$ is not necessarily uniformly positive in $k$; notwithstanding we will show,
under the hypothesis $\mathbb D^{2}(N_{0}, N_{1}) = 0$,
that for all $k$ sufficiently large, 
$t^{\flat}_{k} \equiv 1$.

Indeed, provided $t < t^{\flat}_{k}$, we may estimate the final term in the estimate prior to the definition of $I_{0}$ as follows:
\begin{equation}\label{EKRHH}
\int_{\mathbb T_L^d} \Omega_{N_{t}^{(k)}}\varphi_{x_{0}}^{2}~dx
\leq
\frac{1}{2}\text{e}^{\frac{1}{2}|\mu|+w_0I_{0}}
\int_{\mathbb T_L^d}(1 + N_t^{(k)})\varphi_{x_{0}}^{2}~dx 
\leq
c_{1}\text{e}^{c_{2}I_{0}}(c_{3} + I_{0})
\end{equation}
for finite constants $c_{1} \dots c_{3}$ which do not depend on $k$ or $t$.
So, recalling Eq.~\eqref{DKRRR}, we may write, for $t < t^{\flat}_{k}$, 
\[\begin{split}
p_{k,t}(x_{0})  &\leq 
p_{k,0}(x_{0}) +
[c_{4}I_{0} + c_{1}\text{e}^{c_{2}I_{0}}(c_{3} + I_{0})]^{\frac{1}{2}}
\cdot\b{\int_{0}^{t}\varepsilon_{k}(s) ~ds},
\end{split}\]
with $c_{4}$ similar to the above $c$'s \b{corresponding to the $|\nabla \varphi_{x_0}|^2$ term}.  It is also noted that the first term on the right is independent of $k$ and bounded by $I_{0}$.
Next let $\gamma_{k}$ be defined by 
$$
\left[c_{4}I_{0} + c_{1}\text{e}^{c_{2}I_{0}}(c_{3} + I_{0})\right]^{\frac{1}{2}}
\cdot\int_{0}^{1}\varepsilon_{k}(t) dt := \gamma_{k}I_{0}
$$
where it is noted that the upper limit of the integration is $t = 1$.  
We have, for all $k$ sufficiently large that $\gamma_{k} < 1$ (since $\lim_{k \rightarrow \infty} \int_0^1 \e^2_k(t)\rightarrow 0$) and we have  
at $t = t^{\flat}_{k}$ that for any $x$,
$$
p_{k, t^{\flat}_{k}}  \leq  I_{0}(1 + \gamma_{k}) \b{< 2I_0}
$$
which necessitates $t^{\flat}_{k} = 1$.  

We note from Eq.~\eqref{EKRHH} that 
$\Omega_{N_{t,k}}$ is bounded by 
$\b{\frac{1}{2}}\text{e}^{\frac{1}{2}|\mu|}\text{e}^{w_{0}I_{0}}(1+N_{t,k})$, i.e., for any positive (and, e.g., $\mathscr C^{2}$) function $f$,
$$
\int_{\mathbb T_L^d} \Omega_{N_{t,k}}f~dx \leq
\b{\frac{1}{2}}\text{e}^{\frac{1}{2}|\mu|}\text{e}^{w_{0}I_{0}}\int_{\mathbb T_L^d} (1+ N_{t,k})f~dx.
$$
In particular, with $f \equiv 1$ we find that the total mass $\mathbb M_{t,k}$
satisfies the differential inequality
$$
\frac{d\mathbb M_{t,k}}{dt}= \int_{\mathbb T_L^d} \Omega_{N_t} \Psi_t~dx\leq
\b{\frac{1}{\sqrt 2}\text{e}^{\frac{1}{4}|\mu|}\text{e}^{\frac{1}{2}w_{0}I_{0}}}
\left[L^{d} + \mathbb M_{t,k}\right]^{\frac{1}{2}} \cdot \varepsilon_{k}(t).
$$
\b{Since $t_k^\flat = 1$, certainly $p_{t,k}(x) < 2I_0$, so we have that e.g., $\mathbb M_{t, k} \leq 2\frac{L^d}{|B_a(0)|} \cdot 2I_0$.}
Therefore, defining 
$$
\vartheta_{k} := \int_{0}^{1}\varepsilon_{k}(t)~dt \propto \gamma_{k}
$$
(and so $\vartheta_{k} \to 0$ as $k\to\infty$)
we learn
\begin{equation}\label{DKRIO}
\mathbb M_{t,k}  \leq  \mathbb M_{0} + c \cdot \vartheta_{k}
\end{equation}
where $c$ is another constant depending on $N_{0}$, the total volume and other particulars but is independent of $k$ and $t$.

The proof is now easily finished.  Let $\eta$ denote any
$\mathscr C^{2}$ function.
Then for any $k$,
$$
\int_{\mathbb T_L^d} (N_{0} - N_{1})\eta~dx  = 
\b{\int_0^1} \langle\hspace{-.1cm}\langle\nabla \Psi_{k,t}, \nabla \eta
\rangle\hspace{-.1cm}\rangle_{N_{t,k}}~dt.
$$
\b{The right hand side can easily be bounded:}
\[ \begin{split}
\int_0^1 \langle\hspace{-.1cm}\langle\nabla \Psi_{k,t}, \nabla \eta
\rangle\hspace{-.1cm}\rangle_{N_{t,k}}~dt &= \int_{\mathbb T_L^d}  (\nabla \Psi_t \cdot \nabla \eta)N_t + \int \Omega_{N_t} \Psi_t \eta~dxdt\\
&\leq \|\eta\|_{\mathscr C^1}^{\frac{1}{2}} \mathbb M_{t, k}^{\frac{1}{2}} \cdot \e_k(t) +  \|\eta\|_{\mathscr C^0} \left[L^d + \mathbb M_{t, k}\right]^{\frac{1}{2}}\cdot \e_k(t).
\end{split}
\]
\b{Eq.~\eqref{DKRIO} gives the necessary bound for $\mathbb M_{t, k}$ and so letting $k\to\infty$, we learn}
$
\int \eta ~dN_{0} = \int \eta ~dN_{1}
$
which, since $\eta$ is arbitrary, establishes $N_{0} = N_{1}$
\end{proof}

\begin{thm} The function $\mathbb D(\cdot,\cdot)$ defines a distance on $\mathcal B$.
\end{thm}
\begin{proof}  
Let $N_{0}, N_{1} \in \mathcal B$. To help with the abbreviation of the forthcoming,
let us name by $\mathbb E$ the functional whose infimum produces
$\mathbb D(N_{0}, N_{1})$.  By Proposition \ref{XVAB} we may regard 
\textit{potentials} as the 
arguments of this functional:
$$
\mathbb E^{2}(Q) = \mathbb E^{2}_{N_{0},N_{1}}(Q)  =  \int_{0}^{1}
\langle \hspace{-.1cm} \langle \nabla Q  , \nabla Q  \rangle  \hspace{-.1cm}  \rangle_{N_{t}}~dt
=
 \int_{0}^{1}\int_{\mathbb T_{L}^{d}}
N_{t}|\nabla Q|^{2} + \Omega_{N_{t}}Q^{2}~dxdt
$$ 
where it is noted but notationally suppressed that 
$-\nabla Q \in \mathscr V(N_{0}, N_{1})$ (where $\mathscr V$ is as in Eq.~\eqref{BDF}).
We may also make the trivial addition of allowing the potential to achieve $N_{1}$ at times $T$ other than $t = 1$ in which case the functional becomes 
$$\mathbb E^2(Q) = T\int_0^T\langle \hspace{-.1cm} \langle \nabla Q  , \nabla Q  \rangle  \hspace{-.1cm}  \rangle_{N_{t}}~dt.$$  
Beyond the indiscernible property established above, we must show  
that
$\mathbb D(N_{0}, N_{1})  =  \mathbb D(N_{1}, N_{0})$ and establish the triangle inequality.  The first follows immediately from ``time reversal symmetry''; e.g., on $[0,1]$, 
$t^{\prime} = 1-t$, $K(t^{\prime})  =  -Q(1-t^{\prime})$ gives 
$\mathbb E^{2}_{N_{0}, N_{1}}(Q)  =  \mathbb E^{2}_{N_{1}, N_{0}}(K)$ and the result follows.

As for the triangle inequality, we shall be as succinct as possible since the result follows a transcription of the standard derivation from Riemannian geometry.  When the time interval is $[0, T]$, we  define
$\underline{\mathbb E}(Q)$ 
by taking the square root of the integrand in the definition of 
$\mathbb E^{2}(Q)$:
\[ \underline{\mathbb E}(Q) =  \int_{0}^{T} 
\sqrt{\langle \hspace{-.1cm} \langle \nabla Q  , \nabla Q  \rangle  \hspace{-.1cm}  \rangle_{N_{t}}}~dt.\]
We denote the corresponding minimized object by
$\underline{\mathbb D}(N_{0}, N_{1})$.  It is noted that $\underline{\mathbb E}(Q)$ is completely invariant under the full set of time changes: If $\vartheta(\tau) = \frac{dt(\tau)}{d\tau}$ and 
$$t \to \tau(t), ~~~Q(t) \to K(\tau)  =  \vartheta(\tau) \cdot Q(t(\tau)),$$ 
then 
$\underline{\mathbb E}(Q) = \underline{\mathbb E}(K)$ with $K$ driving $N_{0}$ to $N_{1}$ on the interval $[0, \tau(T)]$.

By convexity we have $\mathbb E^{2}(Q) \geq \underline{\mathbb E}^2(Q)$ and so
$\mathbb D^{2}(N_{0}, N_{1}) \geq \underline{\mathbb D}^2(N_{0}, N_{1})$.
  On the other hand, defining 
 $$\underline{\mathbb E}_{t} := \int_0^t\langle \hspace{-.1cm} 
  \langle \nabla Q, \nabla Q 
  \rangle  \hspace{-.1cm}  \rangle_{N_{t^{\prime}}}^{\frac{1}{2}}~dt^{\prime},~~~t \leq T,$$
and reparameterizing with 
$$\tau = \tau(t)  =  \underline{\mathbb E}_{t},~~~K(\tau) = \left[\left(\frac{d\underline{\mathbb E}_{t}}{dt}\right)^{-1}(t(\tau))\right] \cdot Q(t(\tau)),$$ it is seen that in the new variables, all integrands are identically one and so we have 
$$
\underline{\mathbb E}^{2}(Q)  =  \underline{\mathbb E}^{2}(K) =
\underline{\mathbb E}(K)\int_{0}^{\underline{\mathbb E}_{T}}d\tau =
\tau(T)\int_{0}^{\tau(T)} \hspace{-.2 cm}
\langle \hspace{-.1cm} \langle  \nabla K , \nabla K  \rangle  \hspace{-.1cm}  \rangle_{N_{\tau}}~d\tau
=  \mathbb E^{2}(K).
$$
Taking the infimum over $K$'s (or $Q$'s) we arrive at
$\mathbb D(N_{0}, N_{1})  =  \underline{\mathbb D}(N_{0}, N_{1})$.  The triangle inequality is immediate since given $N_{0}, N_{1}, N_{2} \in \mathcal B$ we can attempt to minimize 
$\underline{\mathbb E}_{N_{0}, N_{2}}(\cdot)$ by considering paths which visit $N_{1}$ on the way to 
$N_{2}$ and so we conclude that
$\mathbb D(N_{0}, N_{2}) \leq \mathbb D(N_{0}, N_{1}) + \mathbb D(N_{1}, N_{2})$.
\end{proof}

\noindent {\large \textbf{Acknowledgements.}}

\footnotesize \noindent L.~C.~ was supported under the NSF Grant PHY 1205295 and to some extent DMS 0805486.

\noindent H.~K.~L. was partially supported under the NSF PostDoctoral Research Fellowship 1004735 while in residence at Caltech.

\textcolor{blue}{We thank the referee for various useful comments and suggesting a slight reformulation of one of the variational problems.}

\end{document}